\documentclass[a4paper]{amsart}
\usepackage[affil-it]{authblk}

\usepackage{amsmath, amsthm, amssymb}
\usepackage{amssymb}
\usepackage{graphicx}
\usepackage{subfigure}
\usepackage[export]{adjustbox}
 \usepackage{titling}

\usepackage[ansinew]{inputenc}
\usepackage{amsfonts}
\usepackage{amssymb,amsmath, amsthm}
\usepackage{bbm}
\usepackage{mathrsfs}

\usepackage{latexsym}
\usepackage{latexsym}
\usepackage{epsfig}
\newcommand {\debeq}	{\begin{eqnarray*}}
\newcommand {\fineq}	{\end{eqnarray*}}

\newcommand     {\eps}  {\epsilon}
\newcommand     {\vareps}       {\varepsilon}

\newcommand	{\tr}{\mathbbm{t}}

\newcommand	{\PP}{\mathbb{P}}
\newcommand	{\EE}{\mathbb{E}}

\newcommand{\mud}{g}

\newtheorem{result}{Result}
\newcommand\bi{\begin{itemize}}
\newcommand\ei{\end{itemize}}

\def\addsymbol #1: #2#3{$#1$ \> \parbox{5in}{#2 \dotfill \pageref{#3}}\\}

\newtheorem{fact}{Fact}

          \newtheorem{teo}{Theorem}[section]
          \newtheorem{defin}[teo]{Definition}
          \newtheorem{prop}[teo]{Proposition}
\newtheorem{cor}[teo]{Corollary}

          \newtheorem{con}{Conjecture}
          \newtheorem{cond}{Condition}
          
          \newtheorem	{lem} 	[teo]	{Lemma}
          
          \newtheorem{rmk}[teo]{Remark}

          \newcommand{\bfact}{\begin{fact}}
          \newcommand{\efact}{\end{fact}}
          \newcommand{\tB}{\tilde B}
          \newcommand{\tV}{\tilde V}
          \newcommand{\ts}{\tilde s}

          \newcommand{\beq}{\begin{equation}}
          \newcommand{\eeq}{\end{equation}}
          \newcommand{\beqn}{\begin{eqnarray}}
          \newcommand{\beqnn}{\begin{eqnarray*}}
          \newcommand{\eeqn}{\end{eqnarray}}
          \newcommand{\eeqnn}{\end{eqnarray*}}
          \newcommand{\bprop}{\begin{prop}}
          \newcommand{\eprop}{\end{prop}}
          
          \newcommand{\bc}{\be\begin{array}{r@{\,}c@{\,}l}}
\newcommand{\ec}{\end{array}\ee}

          \newcommand{\bcor}{\begin{cor}}
          
          \newcommand{\ecor}{\end{cor}}
          \newcommand{\bcon}{\begin{con}}
          \newcommand{\econ}{\end{con}}
          \newcommand{\bcond}{\begin{cond}}
          
          \newcommand{\econd}{\end{cond}}
          \newcommand{\bteo}{\begin{teo}}
          \newcommand{\eteo}{\end{teo}}
          \newcommand{\brm}{\begin{rmk}}
             
          \newcommand{\erm}{\end{rmk}}
          \newcommand{\blem}{\begin{lem}}
          \newcommand{\elem}{\end{lem}}

          \newcommand{\ben}{\begin{enumerate}}
          \newcommand{\een}{\end{enumerate}}
          \newcommand{\bei}{\begin{itemize}}
          \newcommand{\eei}{\end{itemize}}
	 
          \newcommand{\bdf}{\begin{defin}}
          \newcommand{\edf}{\end{defin}}
          
          \renewcommand{\=}{&=&}
          \renewcommand{\>}{&>&}
          \renewcommand{\le}{\leq}
          \newcommand{\+}{&+&}

          \newcommand{\R}{{\mathbb R}}
          
          \newcommand{\E}{{\mathbb E}}
          
          \renewcommand{\P}{{\mathbb P}}

          \newcommand{\N}{{\mathbb N}}
          
          \renewcommand{\S}{{s}}
          
          \newcommand{\cP}{{\mathcal P}}

           \newcommand{\Sm}{Smoluchowski }
          
	\newcommand{\tC}{\tilde C}

          \newcommand{\cE}{{\mathcal E}}
          \newcommand{\cC}{{\mathcal C}}

            \newcommand{\cS}{\mathcal S}

          \newcommand{\D}{{\mathcal D}}

          \newcommand{\cT}{{\mathcal T}}

 \newcommand{\Hc}{\mathscr{H}}

\newcommand{\btt}{\begin{theorem}}
\newcommand{\ett}{\end{theorem}}

\newcommand{\be}{\begin{equation}}
\newcommand{\ee}{\end{equation}}

\newcommand\no{\nonumber}

\newcommand\tPi{\tilde \Pi^n}

\newcommand\tmu{\tilde \mud^n}

          \newcommand{\cD}{\mathcal D}
          
          \newcommand{\cL}{{\mathcal L}}

          \newcommand\sqr{\vcenter{
          \hrule height.1mm
          \hbox{\vrule width.1mm height2.2mm\kern2.18mm\vrule width.1mm}
          \hrule height.1mm}}        % This is a sli

\title{Coagulation-transport equations and the nested coalescents}
\author{Amaury Lambert $^{\dagger \star}$, Emmanuel Schertzer$^{\dagger \star}$}

\date{\today}

\begin{document}

\maketitle{}
{$^{\dagger}$Laboratoire de Probabilit\'es, Statistiques et Mod\'elisation (LPSM), Sorbonne Universit\'e,
CNRS UMR 8001, Paris, France}

{$^{\star}$Center for Interdisciplinary Research in Biology (CIRB), Coll\`ege de France, CNRS UMR 7241,
PSL Research University, Paris, France}

\paragraph{\bf Abstract.}

The nested Kingman coalescent describes the dynamics of particles (called genes) contained in larger components (called species), where pairs of species coalesce at constant rate and pairs of genes coalesce at constant rate provided they lie within the same species. We prove that starting from $rn$ species, the empirical distribution of species masses (numbers of genes$/n$) at time $t/n$ converges as $n\to\infty$ to a solution of the deterministic coagulation-transport equation
$$
 \partial_t d \ = \  \partial_x ( \psi d ) \ + \ a(t)\left(d\star d - d \right),
$$
where $\psi(x) = cx^2$, $\star$ denotes convolution and $a(t)= 1/(t+\delta)$ with $\delta=2/r$. The most interesting case when $\delta =0$ corresponds to an infinite initial number of species. This equation describes the evolution of the distribution of species of mass $x$, where pairs of species can coalesce and each species' mass evolves like $\dot x = -\psi(x)$. We provide two natural probabilistic solutions of the latter IPDE and address in detail the case when $\delta=0$. The first solution is expressed in terms of a branching particle system where particles carry masses behaving as independent continuous-state branching processes. The second one is the law of the solution to the following McKean-Vlasov equation
$$
dx_t \ = \ - \psi(x_t) \,dt \ + \ v_t\,\Delta J_t 
$$
where $J$ is an inhomogeneous Poisson process with rate $1/(t+\delta)$ and $(v_t; t\geq0)$ is a sequence 
of independent rvs such that $\cL(v_t) = \cL(x_t)$. We show that there is a unique solution to this equation and we construct this solution with the help of a marked Brownian coalescent point process. When $\psi(x)=x^\gamma$, we show the existence of a self-similar solution for the PDE which relates when $\gamma=2$ to the speed of coming down from infinity of the nested Kingman coalescent. %We also propose some conjectures related to the previous results.

\vspace{.5cm}



\paragraph{\bf Acknowledgements.} The authors thank \textit{Center for Interdisciplinary Research
in Biology} (CIRB, Coll\`ege de France) for funding.

\setcounter{tocdepth}{1}
\tableofcontents

\section{Introduction and informal description of the main results}

\subsection{The nested Kingman coalescent}

The Kingman coalescent \cite{K82} is a stochastic process describing the dynamics of a system of coalescing particles, where each pair of particles independently merges at constant rate. It originates from population genetics, where it is used to model the dynamics of gene lineages in the backward direction of time, thus generating a random genealogy. This model can be enriched by embedding gene lineages into species (each gene belongs to a living organism which in turn belongs to some species). \emph{Nested coalescents} were recently introduced \cite{BDLS18} to model jointly the genealogy of the genes and the genealogy of the species. The nested Kingman coalescent is the simplest example of a nested coalescent, where both the gene tree and the species tree are given by (non-independent) Kingman coalescents:
\begin{itemize}
\item  Each pair of species independently coalesces at rate $1$, and when two species coalesce into one so-called mother species, all the genes they harbor are pooled together into the mother species (but do not merge). See Fig. \ref{fig:KiK}. 
\item Conditional on the species tree, each pair of gene lineages \emph{lying in the same species} independently coalesces at rate $c$.
\end{itemize}
It is easy to see (and well-known) that $\infty$ is an entrance boundary for the number of particles in the Kingman coalescent, it is said that the Kingman coalescent \emph{comes down from infinity}. It is further known that for the Kingman coalescent started at $\infty$, called the \emph{standard} coalescent, the  number of particles $K_t$ behaves as $t\to 0$ like the solution to 
\be \label{Kingman:descent-from-infinity}
\dot x=-x^2/2, \ \ x_0=+\infty
\ee so that $tK_t$ converges to 2 a.s. \cite{B09}. In the nested Kingman coalescent starting from infinitely many species containing infinitely many genes, the number of species behaves like $2/t$ and the mass of each species decreases at the same speed, but is constantly replenished with the genes of other species upon species coalescences. However, because of the domination by the standard coalescent, it is reasonable to conjecture that each species at time $t$ still carries of the order of $1/t$ genes, so that the total number of genes at time $t$ scales like $1/t^2$. This conjecture has been confirmed to hold in a recent work \cite{BRSS18}. 

\begin{figure}[h] \includegraphics[width=13cm]{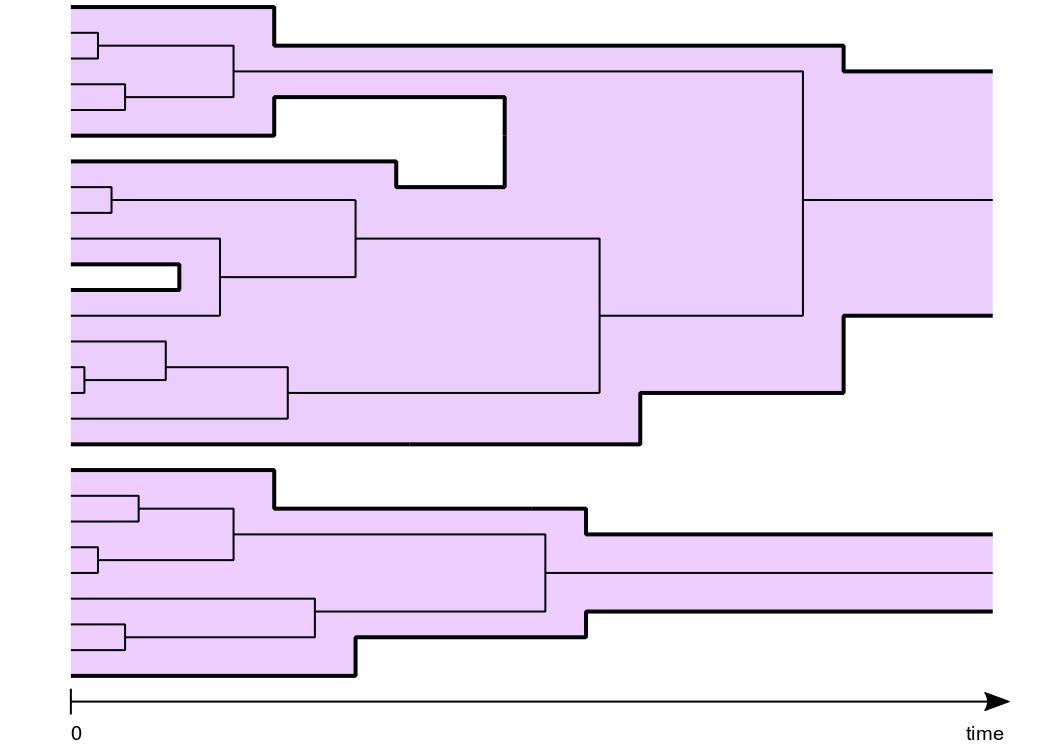} 
 \caption{The nested coalescent.} 
 \label{fig:KiK}\end{figure}

The starting point of this paper was the study of the distribution (rather than the total mass) of species masses at small times in the nested Kingman coalescent, for arbitrary initial conditions. This led us to a journey through \Sm coagulation-transport PDEs and McKean-Vlasov equations in which we develop new techniques that are interesting \emph{per se} and could presumably be applied to more general nested coalescents than the nested Kingman coalescent. We now describe informally the main results of this work.
% In particular, 
%\begin{itemize}
%\item the speed of coming down from infinity in standard coalescents is described by an ODE (see e.g., (\ref{Kingman:descent-from-infinity})
%for the Kingman coalescent).
%\item for nested coalescents,  the speed of coming down from infinity is described in terms of an IPDE (Integro Partial Differential Equation)
%raising non-trivial existence and uniqueness issues.
%\end{itemize}

For a random variable (rv) $X$, $\cL(X)$ denotes the law of $X$.  We let $M_F(\R^+)$ (resp., $M_P(\R^+)$)
denote the set of finite measures (resp. probability measures) on $\R^+$.

\begin{defin}
We denote by $\mathscr{H}$ the set of increasing homeomorphisms $\psi:\R^+\rightarrow \R^+$. In particular, for any $\psi \in \Hc$, $\psi$ is continuous, $\psi(0)=0$, $\psi(x)>0$ for all $x>0$ and $\lim_{x\to\infty}\psi (x) = +\infty$.
The following  (optional) condition
$$
\int_1^\infty\frac{dx}{\psi(x)}<\infty,
$$
will be called \emph{Grey's condition} and sometimes abbreviated as $\int^\infty1/\psi<\infty$. 
\end{defin}

\subsection{Convergence to the \Sm equation} 
\label{sec:convergence}

We first consider a nested coalescent with large but finite number of initial species. More formally, 
we consider a sequence of nested Kingman coalescents indexed by $n$.  Let $s_t^n$ be the number of species at time $t$ (in the model indexed by $n$).
Let $\Pi_t^n$ be the vector of size $s_t^n$ recording the number of gene lineages in each species. We call this vector the genetic composition vector.
We wish to investigate the dynamics of the distribution of species masses on a time scale $O(1/n)$ and rescaling the number of gene lineages by $1/n$. Namely, we define
\be\label{eq:scaling-quantities}
%\ts_t^n = s^n_{t/n}, \ \ \tPi_t \ = \ \frac{1}{n} \cPi_{t/n}, \ \ \mbox{ and } \tmu_t \ = \ \frac{1}{\ts_{t}}  \sum_{i=1}^{\ts_{t}} \delta_{\tPi_{t}(i)}
g_t^n = \frac{1}{s_{t}} \sum_{i=1}^{s_{t}} \delta_{\Pi_{t}(i)/n}
 \quad \mbox{ and } \quad \tmu_t \ = \ \ g_{t/n}^n
\ee
that is, $g^n$ is the empirical distribution of the number of gene lineages per species  renormalized by $n$ and $\tmu$ is obtained from $g^n$ by rescaling time by $n$.

\begin{result}[Theorem \ref{teo-local-1}] \label{teo:conv-1-measures}
Assume that there exist two deterministic quantities $r\in(0,\infty)$ and $\nu\in M_P(\R^+)$ such that
\begin{itemize}
\item[(i)] $\frac{s_0^n}{n} \to r$ in $L^{2+\eps}$
\item[(ii)] $ \tmu_0 \Longrightarrow \nu$ as $n\to \infty$  in the weak topology for finite measures.
\end{itemize}
%(plus some extra mild assumptions.) 
Then the sequence of rescaled empirical measures $\left(\tmu_t; t\geq0\right)$ converges to the unique solution 
of the following IPDE (Integro Partial Differential Equation) 
\begin{eqnarray} 
\ \ \partial_t d(t,x) \ = \  \partial_x ( \psi d )(t,x) \ + \ \frac{1}{t+\delta} \left(d\star d(t,x) \ - \   d(t,x) \right)\quad t, x\geq0  \label{eq:Sch}
\end{eqnarray}
with initial condition $d(0,x)\,dx=\nu(dx)$. Here, $d\star d(t,x) = \int_{0}^x  d(t,x-y) d(t,y)\, dy$ denotes the convolution product, $\psi(x) = \frac{c}{2}x^2$ and $\delta =2/r$. % (called inverse population size).
 \end{result}
 \begin{rmk}
The solution displayed in Result \ref{teo:conv-1-measures} is the unique ``weak" solution of (\ref{eq:Sch}). The notion of weak solution will be made precise  in forthcoming Definition \ref{def:weak-eq}.
\end{rmk}

The latter result provides a natural interpretation of the two terms on the RHS of (\ref{eq:Sch}). The transport term is interpreted 
as the number of gene lineages inside each species  obeying the asymptotical  dynamics (\ref{Kingman:descent-from-infinity})
 whereas the coagulation term is due to the coalescence of species lineages. Finally, $2 \delta=1/r$ is  the inverse population size.
 By a slight abuse of language, the parameter $\delta$ will be referred to as the inverse population size in the rest of this manuscript.

  \begin{rmk}
\label{rmk:Smol}
Consider the more usual coagulation-transport equation 
\be
\ \ \partial_t n\ = \ \partial_x (\psi n ) \ + \ \frac{1}{2} n\star n  \ - \    \hat n (t) n, \nonumber
\ee
where $\hat n(t) =\int_0^\infty n(t,x)\, dx$. Here, pairs of clusters coalesce at rate $1$ and $n(t,x)$ is the \emph{amount} (rather than the density) of clusters of mass $x$ at time $t$. %the mass of each cluster's mass is depleted at a rate equal to $\psi(x)$ 
Then informal computations yield that $\partial_t \hat n=-\frac{1}{2} \hat n^2$, so that 
%\subsection{Motivation: \Sm equation with mass transport} Let $\gamma>1$, $c>0$ and define $\psi(x)= c x^\gamma$.  In order to motivate the main object of this work, let us consider  the \Sm coagulation equation with mass transport 
%\be\label{eq:oo}
%\forall t>0, \forall x\geq0, \ \ \partial_t n (t,x) \ = \ \partial_x (\psi n )(t,x) \ + \ \frac{1}{2}  \int_{0}^x  n(t,x-y) n(t,y) dy  \ - \    n(t,x) \int_{0}^\infty  n(t,y)  dy.
%\ee
%If we think of $n(t,x) dx$ as the ``number'' of clusters carrying a mass in an interval of size $dx$ around $x$ at time $t$, then the previous equation can be interpreted as the following dynamics:  clusters coalesce at rate $1$, and the mass of each cluster is depleted  at a rate proportional to $\mbox{mass}^\gamma$. 
%Let us now write $\hat n(t) = \left<n(t,x)dx,1\right>=\int_{\R^+} n(t,x) dx$ interpreted as the the total number of clusters in the system. Then, informal computations yield that
%\begin{eqnarray*}
%\partial_t \hat n(t) \ = \   \partial_t \int_{\R^+} \hat n (t,x) 
%& = & \int_{\R^+} \partial_x (\psi(x) n )(t,x) \ + \ \frac{1}{2} \int_{\R^+} \int_{0}^x  n(t,x-y) n(t,y) dy  \ - \   \int_{\R+} n(t,x) \int_{0}^\infty  n(t,y)  dy \\
%& = & \ \frac{1}{2} \int_{\R^+} \int_{0}^x  n(t,x-y) n(t,y) dy  \ - \   \int_{\R+} n(t,x) \int_{\R^+}  n(t,y)  dy dx \\ 
%& = & -\frac{1}{2} \hat n^2(t)
%\end{eqnarray*}
%where the first and second lines are obtained by assuming (informally) that $\psi(x)n(x,t)$ vanishes at $\infty$ and that $\int_{R^+} n(t,x) dx<\infty$ for every $t>0$. 
%Solving for $\hat n$ we find that
\[ \hat n(t) \ = \ \frac{2}{t+\delta}  \ \ \mbox{ with $\delta := 2/\hat n(0)$.}  \]
Motivated by the previous heuristics, define $d(t,x) = n(t,x) \frac{(t+\delta)}{2}$ so that the function $d$ is interpreted as the \emph{density} of clusters of mass $x$ at time $t$. Straightforward computations then show that $d$  satisfies \eqref{eq:Sch}. This gives an additional motivation for studying \eqref{eq:Sch} and an additional justification for why $\delta$ is called the inverse population size. 
%\begin{eqnarray}\forall t>0, \forall x\geq0, \ \ \partial_t d(t,x) \ = \  \partial_x ( \psi d )(t,x) \ + \ \frac{1}{\delta+t} \left(\int_{0}^x  d(t,x-y) d(t,y) dy \ - \   d(t,x) \right) \nonumber \\ 
%\mbox{with $d(0,x)dx=\nu(dx)/\left<\nu,1\right>$,  and   $\delta=2/\left< \nu(dx),1\right>$}. \label{eq:Sch}
%\end{eqnarray}
%In the following, we will focus on the resolution of the latter equation instead of (\ref{eq:oo}). As we shall see, our motivation stems from the fact that $d$ relates directly to the speed of coming down from infinity in the nested coalescents model. Making a slight abuse of terminology, (\ref{eq:Sch}) will be referred to as the \Sm equation with (initial) inverse population size $\delta$. Note that the latter equation becomes degenerate at $t=0$ when $\delta=0$. The case $\delta=0$ will be be referred to as the infinite population regime and will also be discussed later on.
%The case when $\psi(x)=cx$ is exacty solvable and has been analyzed in \cite{FL05}. \marginpar{check REF}
\end{rmk}

\subsection{Coming down from infinity} Let us now motivate Equation (\ref{eq:Sch}) when $\delta=0$. Since $\delta$ is interpreted as the inverse population size, this equation will be referred to as the infinite population ($\infty$-pop.) \Sm equation. In the previous section, we started with a finite but large population. Not surprisingly, the $\infty$-pop. \Sm equation arises when the initial number of species is infinite.

 \begin{result}[Theorem \ref{prop:sequential}, Theorem \ref{teo:speed}] \label{res:conv:infinite-pop}
Consider a nested coalescent with the following two properties.
\begin{itemize}
\item[(i)] $s_0=\infty$
\item[(ii)] Each species contains at least one gene lineage.
\end{itemize}
Then the sequence of rescaled empirical measures $\left(\tmu_t; t\geq0\right)$ converges to the unique weak $\infty$-pop. solution 
of (\ref{eq:Sch}). 

As an application of this result, we can derive the speed of coming down from infinity 
in the nested coalescent.
If $\rho_{t}$ denotes the number of gene lineages at time $t$, 
\be\label{eq:speed-cdi}\frac{1}{n^2} \rho_{t/n} \ \Longrightarrow \ \frac{2}{t} \int_{0}^\infty x \mu_t^{(0)}(dx)  = \frac{2\E(\Upsilon)}{t^2}<\infty, \ \ \mbox{as $n\to\infty$}.\ee
where $\Upsilon$ is a random r.v. which is characterized in Result \ref{result:self-similar} below.
%
%where $\Upsilon$ is the rv defined in  Subsection \ref{subsect:infty-pop} in the case $\gamma=2$.
%See Theorem \ref{prop:sequential} for a precise statement.
 \end{result}

\begin{rmk}
The $\infty$-pop. solution displayed in Result \ref{res:conv:infinite-pop} is the unique weak solution of (\ref{eq:Sch}) with $\delta =0$ ($\infty$-population), which is ``proper" , in a sense that it has a  ``non-degenerate'' initial condition, which will be made precise  in forthcoming Definition \ref{def:weak-eq}.
\end{rmk}

\begin{rmk}
Result \ref{teo:conv-1-measures} required to ``tune'' the initial number of species and the number of gene lineages per species in order to get a nondegenerate  scaling limit at time $t/n$ as $n\to\infty$. Indeed, according to Result \ref{teo:conv-1-measures}(i)(ii), both quantities must be of order $n$. In contrast, a fact that stands out in Result \ref{res:conv:infinite-pop} is that
 the $\infty$-pop. equation arises without any delicate scaling. (Compare conditions (i)(ii) in Result \ref{teo:conv-1-measures} and in Result \ref{res:conv:infinite-pop}.)
\end{rmk}

\subsection{General coagulation-transport equation and solution classes}

Motivated by the previous convergence results, we will consider (\ref{eq:Sch})
with a general depletion term $\psi\in \Hc$. As already discussed, Equation \eqref{eq:Sch} describes the density of clusters of mass $x$, where pairs of clusters coalesce at rate $1$ (coagulation) and each cluster's mass evolves like $\dot x = -\psi(x)$ (transport).  
  It will be referred to as the \emph{\Sm equation} with (initial) \emph{inverse population size} $\delta$ (and depletion term $\psi$). The case $\delta =0$ corresponds to an infinite initial number of species. Note that this specific case raises some important issues since  \eqref{eq:Sch} becomes degenerate at $t=0$. (One of the main contributions of the present work is to make sense of such a degenerescence -- see  below.)

In the next two sections, we will define two types of solution of the IPDE (\ref{eq:Sch}).
(with a general transport term.) Namely, 
\begin{enumerate} 
\item[(1)] The notion of weak solution which is the usual framework in the PDE literature and which arose  in our previous convergence results. See Definition \ref{def:weak-eq};
\item[(2)] A more restrictive class of solutions that we call McKean-Vlasov solutions, and that relates (\ref{eq:Sch}) to a natural McKean-Vlasov process describing the evolution of a `typical cluster' in the population. See Definition \ref{def:mc-kean-sense}.
\end{enumerate}

\subsection{Weak solutions}\label{sect:weak}

In this section, we assume that $\psi$ is the \emph{Laplace exponent of a spectrally positive and (sub)critical L\'evy process} $Y$. Note that in particular $\psi \in \Hc$. This choice of $\psi $  should encompass cases of interest regarding the descent from infinity of nested coalescents where the intra-species coalescence mechanism is more general than the Kingman coalescent. 
See Section \ref{sect:conj}(3) for a discussion on a natural conjecture extending the convergence results of the previous two sections
to general nested $\Lambda$-coalescents.

Let us now proceed with the definition of weak solutions in the sense of measures. We are interested in solutions to \eqref{eq:Sch} which have total mass 1 at all times, which implies that $\int_0^\infty \partial_x( \psi d )(t,x)\,dx=0$ or equivalently that $\lim_{x\to\infty} \psi(x) \,d(t,x)=0$ (since $\psi(0) = 0$). 
To be more specific, we will say that $f$ is a test-function iff $f\in\cC^1(\R^+)$, and  further $f$ and $\psi f'$ are bounded.
(Think of $f(x)=\exp(-\lambda x)$ for $\lambda\geq0$.)
Integrating both sides of \eqref{eq:Sch} with respect to such a test-function $f$
and performing an integration by parts yields the following definition in the spirit of \cite{N99}, 
%\[ \int f(v) d(t,v) dv  =  \int f(u) d(0,u) du  -  \int_{0}^t \int \psi(v) f'(v) d(s,v) dv ds \ + \int_0^t \frac{1}{\delta+s} \int \int \left(f(v+u) - f(v)\right) d(s,v) d(s,u)dv du. \] 
which also follows the usual framework of the PDE literature.
\begin{defin}\label{def:weak-eq} 
\label{def:weak1} Let $\delta>0$ and $\nu\in M_P(\R^+)$. We say that a probability-valued process $(\mu_t; t\geq0)$ is a weak solution of \eqref{eq:Sch} with initial condition $\nu$ if for every test-function $f$ and every $t\geq0$
\be
\left<\mu_t,f\right> = \left<\nu, f\right>-\int_{0}^t \left<\mu_s, \psi f'\right> \, ds+ \ \int_0^t \left(\frac{1}{s+\delta} \left<\mu_s\star \mu_s, f\right> - \left<\mu_s,f\right>\right)\, ds,
%\int f(v) \mu_t(dv) = \int f(v)\nu(dv) -  \int_{0}^t \int \psi(v) f'(v) \mu_s(dv) ds \ + \int_0^t a(s) \int \int \left(f(v+u) - f(v)\right) \mu_s(dv) \mu_s(du). 
\label{eq:weak}
\ee 
where we used the notation $\left<\mu, f\right>=\int_{\R^+}f(x) \,\mu(dx)$ for any finite measure $\mu$.

Let $\delta=0$. We say that a probability-valued process $(\mu_t; t>0)$ is a weak solution of \eqref{eq:Sch}   if for every test-function $f$ and every $s,t>0$:
\be
\left<\mu_t,f\right>= \left<\mu_s,f\right> -\int_{s}^t \left<\mu_u,\psi f'\right> \, du \ + \  \ \int_s^t \frac{1}{u} (\left<\mu_u\star \mu_u,f\right> - \left<\mu_u,f\right>) \, du,
%\int f(v) \mu_t(dv) = \int f(v)\nu(dv) -  \int_{0}^t \int \psi(v) f'(v) \mu_s(dv) ds \ + \int_0^t a(s) \int \int \left(f(v+u) - f(v)\right) \mu_s(dv) \mu_s(du). 
\label{eq:weak}
\ee
\begin{itemize}
\item We say that the solution is a \emph{dust solution} iff $\mu_t\to \delta_0$ in the weak topology as $t\downarrow 0$.
% for every $a>0$, $\lim_{t\to0} \mu_t\left([a,\infty)\right)=0$, but for every 
%$t>0$, $\mu_t\left((0,\infty)\right)>0$. 
\item We say that the solution is \emph{proper} otherwise.
\end{itemize}
\end{defin}

\begin{rmk}
The previous terminology is borrowed from fragmentation theory \cite{Ber06}.
\end{rmk}

Let $(\mu_t;t\ge 0)$ be a weak solution to \eqref{eq:Sch}. In view of the convolution product in \eqref{eq:Sch} and \eqref{eq:weak}, it is natural to consider the Laplace transform $u(t,\lambda)$ of $\mu_t$, namely
$$
u(t,\lambda) \ = \ \int_{\R^+} e^{-\lambda x}\, \mu_t(dx) \qquad \lambda, t\ge 0.
$$
Taking $f(x)=\exp(-\lambda x)$ in \eqref{eq:weak} yields that the Laplace transform $u(t,\lambda)$ satisfies the non-linear IPDE
\begin{equation}
\label{eq:Laplace-pde}
\partial_t u \ =  \ \lambda A^\psi u +a(t) (u^2-u), 
\end{equation}
with initial condition $u(0,\lambda) = \int_{\R^+}e^{-\lambda x}\, \nu(dx)$, 
where $A^\psi$ is the generator of $Y$ and $a(t)=  1/(t+\delta)$. In particular, it is crucial that $\lambda A^\psi$ is the generator of the Continuous-State Branching Process (CSBP) $Z$ with branching mechanism $\psi$ \cite{L67}. When $\gamma=2$, $A^\psi$ acting on $\cC^2(\R^+)$ functions coincide with the differential operator $\partial^2/\partial^2\lambda$, $Y$ is Brownian motion and $Z$ is the Feller diffusion \cite{F51}.

By a martingale approach, we can prove the uniqueness (under mild, natural conditions) of the solution of \eqref{eq:Laplace-pde} when $\lambda A^\psi$ is replaced with any generator $A$ of a non-negative Feller process and $a$ is a general continuous function. 

%and express it in terms of the extinction of mass before time $t$ of a branching particle system, where particles carry masses evolving independently according to  $A$. See Lemma \ref{lem:uniqueness-pde}. 

%We further identify the unique weak solution to \eqref{eq:Sch} as soon as $\psi$ is the Laplace exponent of a L\'evy process such that $\int^\infty(dx/\psi(x))<\infty$ (Grey's condition) and $a$ is any non-negative and continuous function. See Theorem \ref{thm:uniqueness-weak}.
%

\begin{result}[Theorem \ref{thm:uniqueness-weak}, Theorem \ref{thm:uniqueness-weak-infinite}, Theorem \ref{teo:dust}]
\label{ref:u-sol-u}
Assume that $\psi$ is the \emph{Laplace exponent of a spectrally positive and (sub)critical L\'evy process}.
\begin{enumerate}
\item[($\delta>0$)] There exists  a unique weak solution solution to (\ref{eq:Sch}) with initial condition $\nu$. 
\item[($\delta=0$)] Assume that $\psi$ satisfies Grey's condition, i.e.., that $\int^\infty 1/\psi<\infty$. Then 
\begin{itemize}
\item There exists a unique
proper solution. See Theorem \ref{thm:uniqueness-weak-infinite}(ii) for a probabilistic expression of the solution.
\item In the stable case (when $\psi(x)=c x^\gamma$), there exist infinitely many
dust solutions.
\end{itemize}
\end{enumerate}
\end{result}

\begin{rmk} The existence of infinitely many dust solutions is reminiscent of a similar behavior for 
the Boltzmann equation \cite{W99}.
\end{rmk}

\begin{result}[Theorem \ref{thm:self-similar-weak}]\label{result:self-similar}
Let  $\psi(x)=c x^\gamma$ with $\gamma\in(1,2]$.  (The case $\gamma=2$ corresponds to the nested Kingman coalescent.) 
The proper solution $(\mu_t; t>0)$
is self-similar in the sense that there exists a r.v. $\Upsilon$ such that
\[\forall t>0, \ \mu_t = \cL\left( t^{-\beta} \Upsilon\right). \]
Further, 
\begin{enumerate}
\item  With $h(x) = \E(\exp(-x \Upsilon))$ then $h$ is the unique solution of the ODE described in (\ref{eq:ODE-branching}). 
\item  $-h'(0)=\E(\Upsilon) <\infty$ can be expressed as the measure of mass extinction under the $0$-entrance measure of a branching CSBP with branching mechanism $\psi$. See also Theorem \ref{thm:uniqueness-weak-infinite}(ii) for a detailed description of $\Upsilon$
in terms of a Branching CSBP.
\end{enumerate}
\end{result}

\begin{rmk}
The variable $\Upsilon$ also arises in the study  \cite{BRSS18} of the speed of coming down from infinity of the nested Kingman coalescent. However,
$\Upsilon$ is not characterized in terms of an IDE or an excursion measure. (as in items (1) and (2) of Result \ref{result:self-similar}.) 
Instead, it is characterized as the fixed point of a distributional transformation. (see Theorem 1 and Eq (1) in \cite{BRSS18}.)
\end{rmk}

\subsection{McKean-Vlasov (MK-V) equation}
Recall that in Section \ref{sect:weak}, we have restricted our attention to the case where $\psi$
is the Laplace exponent of a L\'evy process. In this section, we only assume 
that $\psi\in \Hc$. (Additional assumptions will be needed when considering uniqueness in $\infty$-pop. regime.) 

Let us now consider the McKean--Vlasov  equation 
\be\label{eq:mckean} dx_t \ = \ - \psi(x_t) \,dt \ + \ v_t\,\Delta J^{(\delta)}_t, \ \mathcal{L}(x_0) = \nu \ee
where $J^{(\delta)}$ is an inhomogeneous Poisson process with rate $1/(t+\delta)$ at time $t$, and $(v_t; t\geq0)$ is a sequence 
of independent random variables (indexed by $\R^+$) with the property that $\cL(v_t) = \cL(x_t)$,
and $\nu\in M_P(\R^+)$.

\begin{result}[Theorem \ref{teo:mckean}, Corollary \ref{cor:existence-uniqueness}] 
\label{result5}
For any $\psi\in\Hc$ and $\delta>0$, there exists a unique  solution to MK-V (\ref{eq:mckean}) with initial condition $\nu\in M_P(\R^+)$. 
\label{result:dinit:mc-v}
\end{result}

Informally, one can think of $(x_t; t>0)$ as the evolution of the mass of a focal cluster in  the population: mass is depleted at rate $\psi(x_t)$ and upon coalescence, occurring at rate $1/(t+\delta)$, the cluster gains a mass with law ${\mathcal L}(x_t)$, i.e., the cluster (or species) gains a random mass whose law is the one of another `typical' and independent cluster at time $t$.  

It is straightforward to check 
from It\^o's formula that $(\mu_t:=\cL(x_t); t\geq0)$
is also a weak solution of (\ref{eq:Sch}) 
with inverse population size $\delta$ and initial measure $\nu$ (in the sense of Definition \ref{def:weak1}).
See Lemma \ref{lem:relation} for more details.
This motivates the following definition.
%
%the relation with (\ref{eq:Sch}) goes by a 
%straightforward application of It\^o's formula, since
%for any test function in $f$, 
%$$\E\left(f\left(x_t\right)\right) \ = \ \E\left(f\left(x_0\right)\right)  -  \int_{0}^t \E(\psi(x_s) f'(x_s) ) ds \ + \int_0^t \frac{1}{\delta+s} \int \E\left(f(x_s+u) - f(x_s)\right) \mu_s(du)$$
%which can be rewritten as (\ref{eq:weak}) with $\nu_t=\cL(x_t)$. This shows that if $(x_t; t\geq0)$ is a solution to MK-V  \eqref{eq:Sch}, then the measure-valued process $(\mu_t:=\cL(x_t); t\geq0)$ is also a \marginpar{I removed the notion of MK-V solution} weak solution of (\ref{eq:Sch}) (in the sense of Definition \ref{def:weak1}). 
%In other words, solutions in the MK-V sense are stronger than in the weak sense, and one can recover weak solutions from the MK-V equation by averaging over the population.

\begin{defin}\label{def:mc-kean-sense}
Let $\delta>0$ and $\nu\in M_P(\R^+)$. 
We say that $\left(\mu_t; t>0\right)$ is a MK-V solution of the \Sm equation (\ref{eq:Sch}) iff $(\mu_t; t\geq0)$ is the law of a solution to the MK-V equation \eqref{eq:mckean} (with the same parameters). Note that any MK-V solution is also a weak solution of the IPDE, and as a consequence, the class  of  MK-V solutions is smaller than its weak counterpart.
\end{defin}
 
In Section \ref{sect:MK-V}, we study the solutions to the MK-V equation when $\delta>0$ under minimal assumptions on $\psi$. In particular, we show that the solution can be constructed naturally from the Brownian Coalescent Point Process (CPP) of Popovic \cite{P04}. 
This construction is reminiscent of the construction of the solution to the classical \Sm equation when the coagulation kernel $K$ is equal to $1$ \cite{A99}. 
This explicit representation will allow us to develop some coupling techniques to investigate the $\infty$-pop. regime ($\delta=0$) and the long-time behavior of the solution to \eqref{eq:Sch} when $\delta>0$. 
(We also note that at the intuitive level, our representation of MK-V solutions in terms of the CPP
can be understood in the light of \cite{LS16}
where it is shown that the Kingman coalescent at small scales can be 
described in terms of the CPP. In our setting the CPP describes the limiting species coalescent.)

Let us now consider the $\infty$-pop. MK-V equation   in more detail.
More precisely, we consider (\ref{eq:mckean}) when $\delta=0$
and {\it with no prescription of the initial condition}.
Analogously to Definition \ref{def:weak-eq}, we can define proper and dust solutions for the $\infty$-pop. MK-V equation.
By arguing as before, if $(x_t, t>0)$ is a proper (resp., dust) solution to MK-V, then $(\mu_t:= \cL(x_t), t>0)$
is a proper (resp., dust) solution  to the $\Sm$ equation (\ref{eq:Sch}). In the same vein as Definition \ref{def:mc-kean-sense}, 
the law of such processes will be referred to as solutions of the IPDE in the MK-V sense.
%solution iff it satisfies 
%(\ref{eq:mckean}) on any subinterval of $(0,\infty)$
%and $\cL(x_t)\neq \delta_0$ for every $t>0$ (where $\cL(x_t)$ denotes the law of $x_t$).

%As already explained earlier, the case $\delta=0$ can be thought of as
%the initial number of clusters being  infinite. Both Equations (\ref{eq:Sch}), corresponding to weak solutions, and \eqref{eq:mckean}, corresponding to MK-V, are now degenerate at $t=0$. 
%
%We will show that there exists  a non trivial solution (in the sense that $\mu_t\neq \delta_0$) to the (degenerate) differential equation
%\begin{eqnarray}\forall t>0, \forall x\geq0, \ \ \partial_t d(t,x) \ = \  \partial_x ( \psi d )(t,x) \ + \ \frac{1}{t} \left(d\star d(t,x) \ - \   d(t,x) \right). \nonumber \end{eqnarray}

\begin{result}[Proposition \ref{lem:marking}, Theorem \ref{teo:existence:global:solution}, Theorem \ref{teo:dust}] \label{result:existence-infinite}
Assume that $\delta=0$ and assume that $\psi\in \Hc$ 
satisfies Grey's condition and is convex. Then
\begin{enumerate}
\item There exists at least one proper solution to the $\infty$-pop. MK-V equation. 
\item In the stable case
$\psi(x) = c x^\gamma$ with $\gamma>1$, there is a unique proper solution $(x^{(0)}_t; t\geq0)$. 
Further, this process 
is self-similar in the sense that 
$$
\cL\left(x_t^{(0)}\right) \ = \ \cL\left( t^{-\beta} \Upsilon\right), \ \ \ \ \mbox{where} \ \  \beta :=\frac{1}{\gamma -1}, 
$$ 
where $\Upsilon$ is a positive rv (depending implicitly on $c$ and $\gamma$).
\item In the stable case, there exists infinitely many dust solutions.
\end{enumerate}
\end{result}

\begin{rmk}
Note that in the finite population case ($\delta >0$) the uniqueness of the MK-V solution advertized in Result \ref{result5} holds for any $\psi\in\Hc$ while in the $\infty$-pop. case ($\delta=0$), we have only been able to show uniqueness of the proper solution for $\psi(x) = c x^\gamma$ with $\gamma>1$ (Result \ref{result:existence-infinite}). We nevertheless conjecture that it holds in the general case.
\end{rmk}

\begin{rmk}
When $\gamma\in(1,2]$, the definition of $\Upsilon$ in Results \ref{result:self-similar} and \ref{result:existence-infinite} must  coincide. 
This follows from the fact that any solution in the MK-V sense must be a solution in the weak sense, and
the uniqueness of a weak proper solution when $\gamma\in(1,2]$.  See Result \ref{ref:u-sol-u}.
\end{rmk}

\begin{result}[Theorem \ref{teo:solution1}]\label{res:long-time}
Let $\gamma>1$ and recall $\beta =1/(\gamma -1)$.
Assume that $\psi(x) = c x^\gamma$.
Let $\nu\in M_P(\R^+)$ such that $\nu \neq\delta_0$ and let $x^{(\delta)}$
be the MK-V solution with inverse population size $\delta>0$ and initial probability measure $\nu$. Then
\[\lim_{t\to\infty}  t^{\beta} \ x_t^{(\delta)} \  = \ \Upsilon, \ \ \mbox{in law} \]
where $\Upsilon$ is defined in Result \ref{result:existence-infinite}.
In particular, this shows that as $t\to\infty$, the typical mass of a cluster goes to $0$ as $O(t^{-\beta})$.
\end{result}

%In what follows, we will construct two probabilistic solutions to \eqref{eq:Sch}. It is only at the end of the paper that we use these solutions to study the small time behavior of the nested Kingman coalescent.

%\subsection{Two probabilistic solutions to the \Sm equation \eqref{eq:Sch}}

\subsection{Discussion and conjectures}\label{sect:conj}
%Let us make a rough summary of our results concerning the solutions to \eqref{eq:Sch}.
%\begin{itemize}
%\item
%For any initial probability measure $\nu$ and when $\delta>0$, there exists a unique weak solution to \eqref{eq:Sch} whenever $\psi$ is the Laplace exponent of a spectrally positive L\'evy process. When  $\int^\infty dx/\psi(x)<\infty$ (Grey's condition) there exists a unique continuous function. If instead $a$ is only defined on $(0,\infty)$ and not integrable at $0^+$, there exists a unique solution in the absence of initial condition. 
%\item
%For any initial probability measure $\nu$, there exists a unique MK-V solution to \eqref{eq:Sch} whenever $\psi(x) =cx^\gamma$ for $\gamma >1$ and $a(t) = 1/(t+\delta)$ for $\delta >0$.  If instead $\delta=0$, there exists a unique solution in the absence of initial condition.
%\item The range of parameters considered in the last two items intersect for $\psi(x) = cx^\gamma$ with $\gamma \in(1,2]$ and $a(t) =1/(t+\delta)$ with $\delta \ge 0$.
%\end{itemize}

We have called `infinite population regime' the case where the coagulation term becomes degenerate at $t=0$. The uniqueness of solutions to the \Sm equation that we obtain in these cases in the apparent absence of initial condition is actually due to the fact that this initial condition can be seen as a Dirac mass at $\infty$, which `comes down from infinity', in the sense that $\mu_t$ is a probability measure on $[0,\infty)$ at any positive time $t$.

In this regard, Grey's condition  $\int^\infty 1/\psi<\infty$ is not anecdotal. It ensures that $\dot x=-\psi(x)$ comes down from infinity, so that in the \Sm equation, the antagonism between mass transport towards 0 and the increase of mass by coagulation is dominated by the transport term, in such a way that the mass distribution over species converges to the Dirac mass at 0 as $t\to\infty$. We conjecture that when $1/\psi$ is not integrable at $\infty$, coagulation overwhelms transport, in such a way that the mass distribution densifies around large masses as $t\to\infty$. We further conjecture that in this case, our results concerning uniqueness of solutions in the infinite population regime will not hold any longer. 

% Note that the uniqueness of the weak solution guarantees the uniqueness of the MK-V solution and that the existence of a MK-V solution guarantees the existence of a weak solution. In particular, there exists a weak solution even when $\psi(x) =cx^\gamma$ for $\gamma>2$ (but we have no proof of its uniqueness) and there is at most one MK-V solution when $\psi$ is a Laplace exponent satisfying Grey's condition (but we have no proof of its existence). Our arguments in the MK-V case could certainly be extended to more general conditions, so it is natural to conjecture the existence and uniqueness of both weak and MK-V solutions for any continuous function $a$ (not integrable at $0^+$ in the infinite population regime) as soon as  $\int^\infty dx/\psi(x)<\infty$.

In addition to the previous remarks, and in view of our results, it is natural to make the additional following conjectures. 
\begin{enumerate}
\item The notion of weak and MK-V solutions coincide.
(We only showed that
if $(x_t;t\geq0)$ is a solution to MK-V, then $(\cL(x_t); t\geq0)$ is a weak solution.) 
\item There exists a unique proper weak and MK-V solution for any convex $\psi\in\Hc$ satisfying Grey's condition.
\item  Let $\Lambda$ be a finite measure on $[0,1]$.  Results  \ref{teo:conv-1-measures} and \ref{res:conv:infinite-pop} concerning the convergence of the rescaled distribution of species masses at small times to the solution of the \Sm equation extends to the case when the genes undergo a $\Lambda$-coalescent process coming down from infinity (and the species still undergo a Kingman coalescent) to the same \Sm equation where $\psi(\lambda) = c\lambda^2$ is replaced with 
\be\label{eq:laplace-of-levy}
\psi (\lambda)  = \int_{(0,1]} (e^{-\lambda r}-1+\lambda r)\,r^{-2}\Lambda (dr),
\ee
which is the Laplace exponent of a spectrally positive L\'evy process. This is due to the fact that
in the standard $\Lambda$-coalescent coming down from infinity, the number of lineages
obeys the asymptotical dynamics $\dot x = -\psi(x)$ \cite{BBS07, BBL10}.
\item The entrance law at $\infty$ (in the sense that the number of species at time $t$ goes to $\infty$ as $t\downarrow 0$) of the nested Kingman coalescent is unique.
\item Recall the notion of dust solution of Definition \ref{def:weak-eq}. For a given dust solution, 
we conjecture the existence of a scaling such that the nested Kingman coalescent converges to this solution.
\item For the sake of simplicity, we only considered in the manuscript a Kingman species coalescent. For a general $\Lambda$-coalescent, we expect the same type of results to hold. More precisely, the coagulation term in the IPDE should be replaced the coagulation kernel described in \cite{BL06} Proposition 3.
\end{enumerate}

%
%\subsection{Outline}
%
%\begin{enumerate}
%\item In Section \ref{sect:MK-V}, we provide a proof for uniqueness and existence of the MK-V problem.
%\item In Section \ref{sect:infty-pop}, we prove an existence and uniqueness result for the $\infty$-pop. regime for the MK-V problem.
%\item In Section \ref{sect:applications}, we apply the results of Section \ref{sect:infty-pop} 
%to investigate the long-time behavior of the \Sm equation (\ref{eq:Sch}).
%\item In Section \ref{sect:special-case}, we investigate the special case where $\psi(x) = c x^\gamma$
%with $\gamma\in(1,2]$.
%\item The rest of the paper is dedicated to the nested Kingman coalescents model and its relation
%to the \Sm equation when $\psi(x) = \half c x^2$. 
%\end{enumerate}

%We say that a measure valued process $(\nu_t; t\geq0)$ is in $\cE$ iff 
%its Laplace process 
%is $C^{1,2}(\R^+_*, \R^+_*)$, i.e. once differentiable with respect to time on $(0,\infty)$, and twice differentiable with respect to $\lambda$ on $\R^+_*$. 

\section{Weak solutions and branching CSBP}
\label{sect:special-case}

\subsection{Main assumptions}
In this section, we consider the case where $\psi$ is of the form
$$\psi(x) \ = \ ax + b x^2 \ + \ \int_{(0,\infty)} (\exp(-rx)-1+rx) \,\pi(dr)$$
where $b\geq0$, $\pi$ is a $\sigma$-finite measure on $(0,\infty)$
such that $\int_{\R^+} (r \wedge r^2) \pi(dr) <\infty$, and $a=\psi'(0^+)\ge 0$, so that $\psi\in \Hc$ is the Laplace exponent of a spectrally positive (sub)critical L\'evy process $Y$ with 
%with finite, constant expectation, that is 
$$
\EE_0(\exp(-\lambda Y_t)) = e^{t\psi(\lambda)}\qquad t,\lambda \geq 0.
$$ 
In particular, we can recover the case $\psi(x) = cx^2$ by taking $\pi=0$ and we can recover the cases
$\psi(x) = c x^\gamma$ for $\gamma\in(1,2)$ by taking $b=0$ and the jump measure of the form $\pi(dx)=\frac{\bar c}{x^{\gamma+1}}$, where $\bar c$ is some positive constant.
Note that under our assumptions, for any $t_0\in \R$ and  $x_0>0$, there exists a unique solution $(v(t);t\ge t_0)$ to the ODE
\[\dot v \ = \ -\psi(v), \  v(t_0)=x_0,\]
and that $\lim_{t\to\infty} v(t) = 0$.

\subsection{Laplace transform of weak solutions}

We let $A^\psi$ denote the generator of the L\'evy process $Y$. There is a Feller process with generator $A$ given by $Af(\lambda)= \lambda A^\psi f(\lambda)$, also known as the  CSBP (Continuous State Branching Process) with branching mechanism $\psi$. 
For any $\mu\in M_F(\R^+)$,  
$f(\lambda) \ = \ \int_{\R^+}  \exp(-\lambda x) \mu(dx)$
is in the domain of the generator $A$, and further
$$\forall \lambda\ge  0, \ \ A f(\lambda) \ = \ \lambda \int_{\R^+} \psi(x) \exp(-\lambda x) \mu(dx). $$
Let us now consider $(\mu_t; t\geq0)$ a weak solution of the \Sm equation and set
$$
u(t,\lambda) \ = \ \int_{\R^+} e^{-\lambda x}\, \mu_t(dx) \qquad \lambda, t\ge 0.
$$
Since $f(x)=\exp(-\lambda x)$ with $\lambda>0$ is a test-function then plugging this choice of $f$ into  \eqref{eq:weak}
shows, after differentiation with respect to $t$, that $u$ satisfies \eqref{eq:Laplace-pde} for all $\lambda>0$ and $t\geq 0$, that is,
$$
\partial_t u \ =  \ A u +a(t) (u^2-u), 
$$
with initial condition $u(0,\lambda) = \int_{\R^+}e^{-\lambda x}\, \nu(dx)$, where $a(t)=  1/(t+\delta)$. 

%In this section on weak solutions $a$ can actually be any continuous, non-negative function. We will specify assumptions on $a$ when needed.

In the next subsection, we recall some well-known facts on the  CSBP with branching mechanism $\psi$. Then we introduce a closely related object: the branching CSBP.  In Subsection \ref{sec:uniqueness-weak}, we prove that there exists a unique weak solution to \eqref{eq:Sch} which can be expressed in terms of a branching particle system.
% under the mere assumption that $a$ is a non-negative continuous function. 
 In Subsection \ref{sec:uniqueness-weak-infinite}, we focus on proper weak solutions in the infinite population of the \Sm equation (\ref{eq:Sch}). 
%More precisely, assuming that $a$ is only defined on $(0,\infty)$ and that $a$ is not integrable at $0+$, 
%We prove that there exists a unique proper unique weak solution to \eqref{eq:Sch} (now only defined for $t>0$), in the absence of any boundary condition.
 In Subsection \ref{sec:uniqueness-weak-self-similar}, we provide additional results in the stable case $\psi(\lambda)= c\lambda^\gamma$ for $\gamma \in (1,2]$. (that we call the stable case.)

\subsection{Continuous-state branching processes (CSBP)}
\label{subsec:csbp}

We collect here known results about CSBP. See e.g., \cite{DLG02} or Section 2.2.3. in \cite{L08}.

A CSBP $Z=(Z_t;t\ge 0)$ is a Feller process with values in $\R^+$ with the branching property, namely if $P_x$ denotes the law of $Z$ started at $x\ge 0$, then $P_x\star P_y =P_{x+y}$. It is well-known \cite{L67, CLUB09}, that CSBPs are in one-to-one correspondence with L\'evy processes with no negative jumps via several different bijections, including a random time-change known as Lamperti's transform. If $\psi$ is the Laplace exponent of such a L\'evy process $Y$ (assumed to be (sub)critical), then the CSBP $Z$ associated to $Y$ is called the \emph{CSBP with branching mechanism $\psi$} and the generator $A$ of $Z$ is given by
$$
Af(\lambda)= \lambda A^\psi f(\lambda) \qquad \lambda\ge 0,
$$
for any $f$ in the domain of the generator $A^\psi$ of $Y$. Further, 
\be\label{eq:laplace-csbp}
\E(\exp(-\lambda Y_t)) = \ \exp(-x u_t), \ \mbox{where}  \ \dot u = -\psi(u), \ \ u_0 = \lambda.
\ee
Note that by the branching property, $0$ is an absorbing state for $Z$. It is accessible iff $1/\psi$ is integrable at $\infty$ (Grey's condition \cite{G74}). We then denote by $T_0$ the first hitting time of 0 by $Z$.

By the branching property, the law of $Z_t$ is infinitely divisible. More specifically, if Grey's condition is fulfilled, then under $P_x$, $Z_t$ is equal to the sum $\sum_i Z_t^{(i)}$, where $(Z_t^{(i)})$ are the atoms of a Poisson point process with intensity measure $x N$, where $N$ is a $\sigma$-finite measure on c\`adl\`ag processes started at 0 and with non-negative values. We will call $N$ the \emph{entrance measure at 0 of the CSBP}.
In particular, for any $\lambda\ge 0$,
\begin{equation}
\label{eq:expo-formula}
E_x(\exp(-\lambda Z_t)) =\exp\left(-x N\left(1-e^{-\lambda Z_t}\right)\right),
\end{equation}
and so for any measurable functional of paths $G$ such that $|G|\le KT_0$ for some $K$,  %set $B$ of c\`adl\`ag paths such that $T_0\ge \varepsilon$ for some $\varepsilon >0$, 
\be
\label{eqn:N}
\lim_{x\downarrow 0}x^{-1}E_x(G) = N(G)<\infty.
\ee
Note in particular that $N (T_0>t)<\infty$ for all $t> 0$. 

\medskip

\subsection{Duality between a branching and a coalescing particle system}  
In this section, we fix $\tr$ an ultrametric binary tree with $n$ labelled leaves and depth $T$
(i.e., the distance from the root to each leaf is $T$). We let $N_t(\tr)$ denote the number of points in $\tr$ at time $t$, i.e., at distance $t$ from the root. 
We will now introduce two particle systems, a coalescing particle system initialized at the leaves of $\tr$ and a branching particle system initialized at the root of $\tr$.

Let us first introduce the coalescing particle system. We start by assigning a mark $\lambda_i\ge 0$ to the leaf of the tree $\tr$ labelled by $i$. Then we let the marks propagate from the leaves to the root according to the following rules:
(i) along each branch, the marking evolves according to the deterministic dynamics $\dot x=-\psi(x)$, and (ii) when two branches merge, we add
the corresponding two marks. We call $F(\tr,\lambda)$ the resulting mark at the root.

Now, we consider a system of branching particles with random mass running along the branches of the tree, from the root to the leaves, according
to the following rules: (i) we start with one particle at the root, (ii) at each branching point the incoming particle, carrying say mass $x$, duplicates into two copies of itself, one copy for each branch, each with mass $x$,
and (iii) along each branch, the mass of the particle on that branch evolves independently according to a CSBP with branching mechanism $\psi$.
See Fig. \ref{fig:csbp}. 

 \begin{figure}[h] \includegraphics[width=13cm]{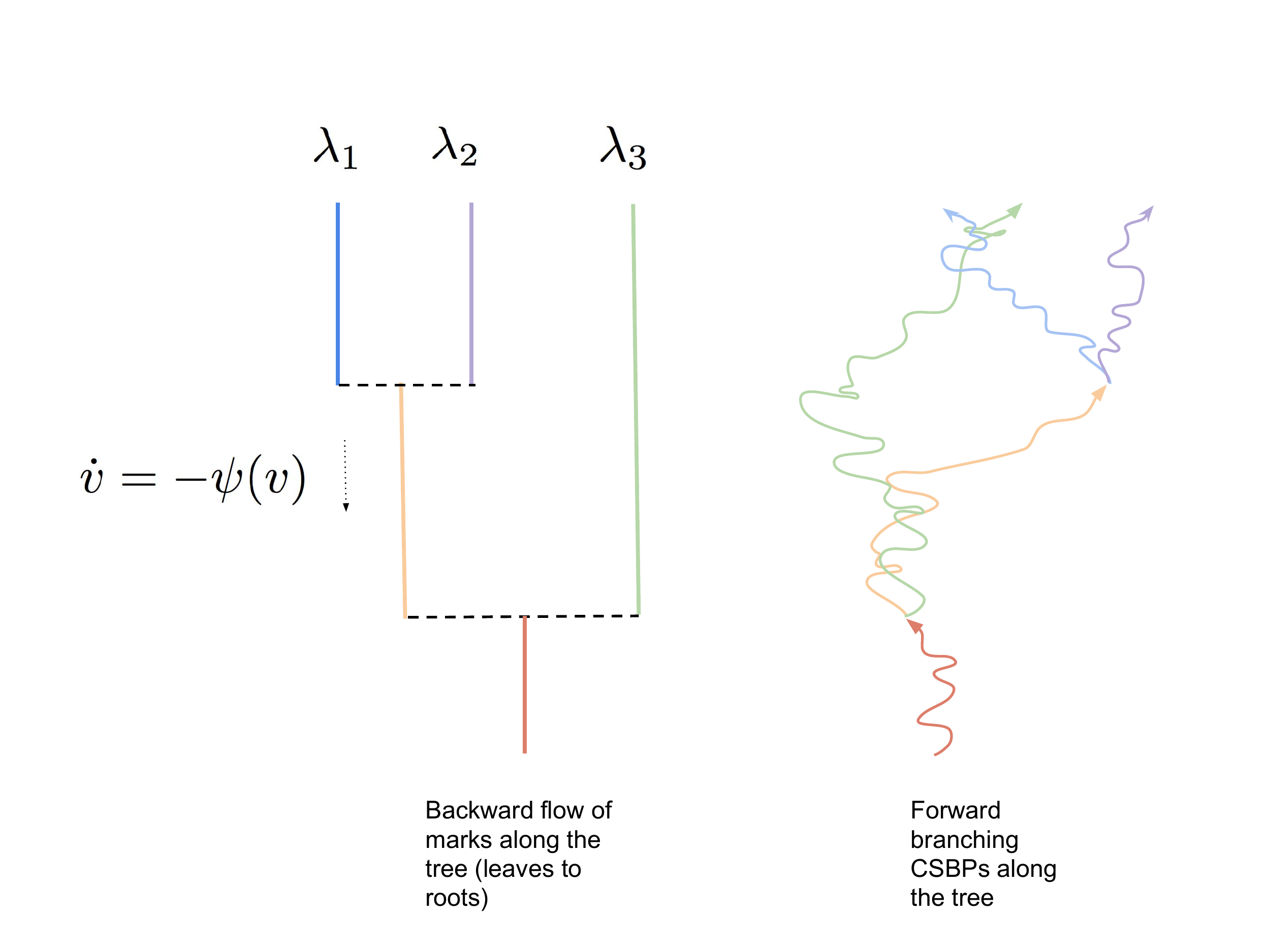} 
 \caption{Duality between branching CSBPs and the coalescing particle system} 
 \label{fig:csbp}\end{figure}

Let $Q_x^\tr$ denote the law of the branching particle system started with one particle with mass $x$ at time 0 and let $\mathcal Z^{\tr}_t = (Z_t^i)_{1\le i\le N_t(\tr)}$ denote the masses carried by the particles at time $t$. Finally, as a direct consequence of the branching property, it is not hard to see that $\mathcal Z^{\tr}$ is infinitely divisible, and as for a simple CSBP, under Grey's condition, $\mathcal Z^{\tr}$ under $Q_x^\tr$
can be decomposed into a Poisson sum of elementary processes starting from $0$ with intensity measure $xM^\tr$, where $M^\tr$ is the \emph{entrance measure at 0 of ${\mathcal Z}^\tr$}. In view of \eqref{eqn:N}, we have
 $$
 M^\tr(G) = \lim_{x\downarrow 0} x^{-1} Q_x^\tr(G)
 $$  
 for any measurable functional of paths $G$ such that $|G|\le KT_0$ for some $K$ (where here $T_0$ is the extinction time of the particle on the root edge of $\tr$, set to $+\infty$ if the particle splits before going extinct).

Now analogously to (\ref{eq:laplace-csbp}), there exists a nice characterization of the Laplace transform of $\mathcal Z^{\tr}_t $.
\begin{prop}  
\label{prop:reloumanu}
For any non-negative numbers $x$ and $(\lambda_i)_{1\le i\le n}$,
\be \E_x\left(\exp( -\lambda\cdot{\mathcal Z}^\tr_T  )\right) = \exp (-x F(\tr, \lambda)), \label{eq:laplace-general}\ee
with the notation ${\lambda\cdot\mathcal Z}^\tr_T  = \sum_{i=1}^{n} \lambda_i Z_t^i$.
Under Grey's condition, we additionally have
\be M^{\tr}(1 - \exp(-\lambda\cdot{\mathcal Z}_T^\tr )) \ = \ F(\tr, \lambda).
 \label{eq:laplace-general2}
\ee
\end{prop}
\begin{proof}
Using (\ref{eq:laplace-csbp}), it is straightforward to get \eqref{eq:laplace-general}  by induction on the number of nodes of the tree. Under Grey's condition, analogously to (\ref{eq:expo-formula}),
we have
\[
\E_x\left(\exp( -\lambda\cdot{\mathcal Z}^\tr_T  )\right) = \exp \left\{-x
M^{\tr}(1 - \exp(-\lambda\cdot{\mathcal Z}_t )) \right\},\]
which yields \eqref{eq:laplace-general2}.
\end{proof}

%Let $\delta\geq 0$. We say that a sequence of probability measures $(\mu_t; t\geq\delta)$ belongs to $\cE_\delta$ iff the function $u(t,\lambda)\ = \ \left<\exp(-\lambda x), \mu_t\right>$ is $C^{1,2}([\delta,\infty),\R^+)$. Let us now consider the PDE
%\begin{eqnarray}
%\forall t>\delta, \forall \lambda>0, \ \ \partial_t u (t,\lambda) \ = \ \frac{c}{2} \lambda \partial_{\lambda\lambda} u \ + \ \frac{1}{t} (u^2-u). \label{eq:laplace}
%\end{eqnarray}
 
 \subsection{\Sm equation in finite population}
 \label{sec:uniqueness-weak}
  Let $P_t$ denote the semigroup and $A$ the infinitesimal generator of a time-homogeneous Feller process $Z$ with values in $[0,\infty)$. We are interested in the case when $Af(\lambda) = \lambda A^\psi (\lambda)$ but do not immediately restrict our study to this case. 
 Let $\mathcal E_A$ be the space of continuous functions $f:[0,\infty)\to [0,1]$ such that $Af$ is a well-defined continuous, bounded function on $(0,\infty)$ and $f(Z_t) -\int_0^tAf(Z_s) \,ds$ is a martingale (or equivalently $t\mapsto P_tf(x)$ is differentiable at 0 with derivative $Af(x)$ for all $x\ge 0$). Further let $\mathcal E_A'$ be the space of two-variable continuous functions $f:[0,\infty)\times [0,\infty)\to [0,1]$ such that $f(t,\cdot)\in \mathcal E_A$ and $f(\cdot, x)\in \cC^1$.  

Let $a:[0,\infty)\to [0,\infty)$ be a continuous map and fix $T>0$. Let us consider a time-inhomogeneous system of branching particles, where each particle carries an individual mass evolving like the process $Z$ and each particle independently gives birth at rate $\tilde a(t)=a(T-t)$ (at time $t$) to a copy of itself (i.e., a branching particle with mass $x$ splits into two particles, each with mass $x$). Let $\mathcal Z_t$ denote the state of this system (e.g., its empirical measure) at time $t$, assumed to be c\`adl\`ag. We assume that the system starts at time 0 with one particle carrying mass $x$, and we then let $Q_x^{T}$ denote the law of $\mathcal Z$ up to time $T$. We also let $N_t$ denote the number of particles at time $t$ and by $(Z_t^i)_{1\le i\le N_t}$ the masses carried by these particles. 

\begin{rmk}
Let $\tr$ be an ultrametric tree with depth $T$. When $Af(\lambda) = \lambda A^\psi (\lambda)$, and if we condition on the genealogy of the process 
up to $T$ to be $\tr$, the process $\mathcal Z$ coincides with $\mathcal Z^\tr$
as defined in the previous section. 
\end{rmk}

 \begin{lem} 
 \label{lem:uniqueness-pde}
For any $g\in \mathcal E_A$ there is at most one solution $u\in \mathcal E_A'$ to the PDE (or IPDE) 
\begin{eqnarray}
 \partial_t u \ = \ A u \ + \ a(t) (u^2-u) \label{eq:pde},
\end{eqnarray}
with initial condition $u(0,x) =g(x)$. If $v$ is defined by 
$$
v(T, x)=Q_x^{T}\left(\prod_{i=1}^{N_T} g(Z_T^i)\right)\qquad T>0, x\ge 0
$$
is in $\mathcal E_A'$ then $u=v$.
 \end{lem}
   Before proving this lemma, we wish to state the relevant corollary regarding weak solutions to \eqref{eq:Sch}, taking $A$ in the previous lemma equal to the generator of the CSBP $Z$ with branching mechanism $\psi$. 
%    We need to introduce some additional notation. For a given binary tree $\tr$ embedded in continuous time and starting at time 0 with one particle, we write $N_t(\tr)$ for the number of particles at time $t$ in $\tr$ and we denote by $Q_x^\tr$ the law of the particle system started with one particle with mass $x$ at time 0, where particle masses evolve like independent copies of $Z$ through time and where \emph{the genealogy of particles is given by $\tr$}. 
 \begin{teo} 
 \label{thm:uniqueness-weak}
 Assume that $\psi$ is the Laplace exponent of a (sub)critical spectrally positive L\'evy process.
  \begin{enumerate}
 % such that $1/\psi$ is integrable at $\infty$ (including the cases when $\psi(x) = c x^\gamma$ for $\gamma\in(1,2]$).
 \item[(i)] There exists a unique weak solution $(\mu_t;t\ge 0)$ to the \Sm equation \eqref{eq:Sch} with initial distribution $\nu$
 and inverse population $\delta>0$.
 \item[(ii)] For any $T\ge 0$, $\mu_T$ is equal to the law of $F({\bf T}, (W_i); 1\le i \le N_T (\bf T))$, where $\bf T$ is the time-inhomogeneous tree started at 0 with one particle, stopped at time $T$ and with birth rate $\tilde a(t)=1/(T-t+\delta)$, the $(W_i)$ are iid with law $\nu$.
 \end{enumerate}
 \end{teo}
\begin{proof}
Let us apply Lemma \ref{lem:uniqueness-pde} with $A$ defined by $Af(\lambda) = \lambda A^\psi (\lambda)$ the generator of the CSBP $Z$ with branching mechanism $\psi$ and $a(t) = \frac{1}{t+\delta}$. Then Equation \eqref{eq:pde} is the same equation as \eqref{eq:Laplace-pde}, with initial condition $g(\lambda) = \int_{\R^+}e^{-\lambda x}\, \nu(dx)$, which we can write 
$$
g(\lambda) = \E(\exp(-\lambda W)),
$$
where $W$ denotes a rv with law $\nu$. Note that $g$ takes values in $[0,1]$. Let us check that $g\in \mathcal E_A$. Recall that for any $f$ of the form $f(\lambda) \ = \ \int_{\R^+}  \exp(-\lambda x) \mu(dx)$, $f$ is in the domain of $A$, and further $A f(\lambda) =  \lambda \int_{\R^+} \psi(x) \exp(-\lambda x) \mu(dx)$. This shows that 
$$
Ag(\lambda)=\lambda \int_{\R^+} \psi(x) \exp(-\lambda x) \nu(dx),  
$$
so that $Ag$ is a well-defined continuous function on $(0,\infty)$  (since $\psi$ increases at most polynomially at $\infty$). In addition, it is well-known (see e.g. \cite{CLUB09}) that 
$$
e^{-x Z_t}- \psi(x) \int_0^t Z_s\,e^{-x Z_s} \,ds
$$
is a martingale for any $x\ge 0$, so by integrating $x$ wrt the probability measure $\nu$, we get that $g(Z_t) - \int_0^t Ag(Z_s) \,ds$ is also a martingale. By dominated convergence, $Ag$ vanishes at $\infty$ and so is bounded and we conclude that $g\in \mathcal E_A$. So by Lemma \ref{lem:uniqueness-pde} there is at most one solution $v\in\mathcal E_A'$ to \eqref{eq:Laplace-pde} given by
$$
v(T, \lambda)=Q_\lambda^{T}\left(\prod_{i=1}^{N_T} g(Z_T^i)\right). 
$$
Now let $(\mu_t;t\ge 0)$ be a weak solution to the \Sm equation \eqref{eq:Sch} and set 
$$
u(t,\lambda)\ = \  \int_{\R^+} e^{-\lambda x}\, \mu_t(dx) \qquad \lambda, t\ge 0.
$$
Recall that $u$ satisfies \eqref{eq:Laplace-pde} with initial condition $u(0,\lambda) = g(\lambda)$. The exact same reasoning used to prove that $g\in \mathcal E_A$ shows that $u(t, \cdot)\in \mathcal E_A$ for all $t$, and because $u$ satisfies \eqref{eq:Laplace-pde}, 
$$
\partial_t u (t,\lambda) - a(t) (u^2-u)(t,\lambda) = \lambda A^\psi u (t,\lambda) = \lambda \int_{\R^+} \psi(x) \exp(-\lambda x) \mu_t(dx),
$$
which is continuous in $t$. Since $a$ is continuous, we get that $u(\cdot, \lambda)$ is of class $\cC^1$ and so $u\in\mathcal E_A'$. This shows that $u=v$, so that
$$
u(T,\lambda)=Q_\lambda^{T}\left(\prod_{i=1}^{N_T} \E(\exp(-Z_T^iW_i)|Z_T^i)\right)=\E\left[Q_\lambda^{T}\left(\exp\left(-\sum_{i=1}^{N_T}W_iZ_T^i\right)\right)\right],
$$
where $\E$ is the expectation taken wrt the $(W_i)$, which are independent copies of $W$ (and we have applied Fubini--Tonelli Theorem). 
From (\ref{eq:laplace-general}), we get
$$
u(T, \lambda)= \E\left[Q_\lambda^{T}\left(\exp\left(-\sum_{i=1}^{N_T}W_iZ_T^i\right)\right)\right]= \EE_T\left[\exp\left(-\lambda F({\bf T}, (W_i); 1\le i \le N_T (\bf T))\right)\right],
$$
where now $\E_T$ is the expectation taken wrt the Yule tree $\bf T$ with branching parameter $\tilde a(t)=1/(T-t+
\delta)$ stopped at time $T$ and the iid rvs $(W_i)$.
It follows that
$$
\int_{\R^+}e^{-\lambda x}\mu_T(dx) = u(T,\lambda) = \EE\left[\exp\left(-\lambda F({\bf T}, (W_i); 1\le i \le N_T (\bf T))\right)\right],
$$
and by the injectivity of the Laplace transform, $\mu_T$ is the law of $F({\bf T}, (W_i); 1\le i \le N_T (\bf T))$. 

For the sake of completeness, in Appendix \ref{Appendix1} we check that 
$\mu_T$ \emph{defined as} $F({\bf T}, (W_i); 1\le i \le N_T (\bf T))$ indeed is solution to \eqref{eq:Sch}. 
(As a matter of fact, the existence of a weak solution will be proved in the MK-V section (see Theorem \ref{teo:mckean}) and thus checking 
that $\mu$ is indeed a weak solution is not formally needed.)
\end{proof}
 
\begin{proof}[Proof of Lemma \ref{lem:uniqueness-pde}]
Recall that $P_t$ denotes the semigroup of $Z$. We will use the notation $E_x$ to denote the expectation associated with its law when started from $x$. 
We extend the semigroup and generators by defining for any $f\in \mathcal E_A'$,
$$
\bar P_{s} f(t,x) =  E_x(f(t+s, Z_{s}))\quad  \mbox{ and }\quad \bar Af(t,x)= Af(t,x)+ \partial_t f(t, x),
$$
so that in particular, 
$$
\lim_{\varepsilon\downarrow 0} \frac1\varepsilon\left(\bar P_{\varepsilon} f(t,x) - f(t, x)\right) = \bar Af(t,x).
$$
Let $u\in \mathcal E_A'$ be a solution to \eqref{eq:pde} with initial condition $g$. Fix $T>0$ and recall the system of branching particles defined before the statement of the lemma. Because the dynamics of the system are time-inhomogeneous, we will need to denote by $Q^T_{t,x}$ the law of $\mathcal Z$ when started with one single particle with mass $x$ at time $t$. In particular, $Q^T_x=Q^T_{0,x}$.
 Let $\mathcal F_t$ denote the $\sigma$-field generated by $\mathcal Z_t$. Fix $T>0$ and 
 set $\mathcal Z^u$ the $(\mathcal F_t)$-adapted process given by 
$$
\forall t\in [0,T], \ \ \ \mathcal Z^u_t =\prod_{i=1}^{N_t} \tilde u(t, Z_t^i), \ \ 
\mbox{where   }
\tilde u(t,x):= u(T-t, x),
$$
$N_t$ is the number of particles present at time $t$ and $Z_t^i$ is the mass of particle $i$ (note from the definition of $\mathcal Z_t^u$ that it does not depend on the labelling chosen). 

%Now let $t, s$ such that $0\le t \le t+s\le T$. 
%Notice that, by the branching property,
%$$
%Q^T(\mathcal Z^h_{t+s}\mid \mathcal F_t) = \prod_{i=1}^{N_t}Q^T_{t,Z_t^i}(\mathcal Z^h_{t+s}).
%$$
We aim at proving that $Q^T_{x}(\mathcal Z^u_{t})$ is constant as a function of $t$. %, which will show by the previous equality that $\mathcal Z^h$ is a $(\mathcal F_t)$-martingale.
Let $t, \varepsilon$ such that $0\le t \le t+\varepsilon\le T$.
Conditional on $\mathcal F_t$, denote by $\tau_i$ time when the $i$-th particle splits, $i=1,\ldots, N_t$. Then by the branching property,
\begin{multline*}
Q^T(\mathcal Z^u_{t+\varepsilon}\mid \mathcal F_t) = \PP(\tau_i >t+\varepsilon, \forall i)\prod_{i=1}^{N_t}E_{Z_t^i}(\tilde u(t+\varepsilon,Z_{\vareps}))\, \\+ \sum_{j=1}^{N_t}\PP(\tau_i >t+\varepsilon, \forall i\not=j)\int_{t}^{t+\varepsilon}\PP(\tau_j\in dv)\, \,E_{Z_t^j}(Q^T_{v,Z_{v}}(\mathcal Z^u_{t+\varepsilon})^2)\,\prod_{i\not =j}E_{Z_t^i}(\tilde u(t+\varepsilon, Z_{\varepsilon})) + C_\varepsilon,
\end{multline*}
where $C_\varepsilon \le \PP(B_\varepsilon\ge 2)$, with 
$$
B_\varepsilon :=\#\{i\le N_t: t\le \tau_i \le t+\varepsilon\}.
$$
Now $B_\varepsilon$ is a binomial rv with parameters $N_t$ and $u_\varepsilon:=1-e^{-\int_{t}^{t+\varepsilon}\tilde a(u)\,du}$, so 
$$
\PP(B_\varepsilon \ge 2) \le \frac{N_t(N_t-1)}{2}\,u_\varepsilon^2 \le N_t^2 M^2 \varepsilon^2,
$$
where $M:=\sup_{t\in [0,T]} a(t)$.
Re-arranging, we get
\begin{multline*}
\frac 1\varepsilon(Q^T(\mathcal Z^u_{t+\varepsilon}\mid \mathcal F_t) -  \mathcal Z^u_{t}) =  \frac1\varepsilon\left(\prod_{i=1}^{N_t}E_{Z_t^i}(\tilde h(t+\varepsilon, Z_\varepsilon))-\mathcal Z^u_{t}\right)\\
+ \frac1\varepsilon(1-e^{-N_t\int_t^{t+\varepsilon} \tilde a(u)\, du})\prod_{i=1}^{N_t}E_{Z_t^i}(\tilde  u(t+\varepsilon, Z_\varepsilon))\\
+  \frac{e^{-(N_t-1)\int_t^{t+\varepsilon} \tilde a(u)\, du}}\varepsilon\sum_{j=1}^{N_t}\int_{t}^{t+\varepsilon}\tilde a(v) dv\,e^{-\int_t^v\tilde a(u)\,du} \,E_{Z_t^j}(Q^T_{v,Z_{v}}(\mathcal Z^u_{t+\varepsilon})^2)\,\prod_{i\not =j}E_{Z_t^i}(\tilde u(t+\varepsilon,Z_{\varepsilon}))\\
 + \frac{1}\varepsilon C_\varepsilon.
\end{multline*}
Because $u\in \mathcal E_A'$ and $a$ is continuous, the right-hand side of the last equality converges as $\varepsilon \downarrow 0$ to 
\begin{align*}
\sum_{i=1}^{N_t}(A\tilde u(t, Z_t^i))\prod_{j\not = i}\tilde u(t, Z_t^j)- \tilde a(t)\,N_t\mathcal Z_t^u+\tilde a(t)\,\sum_{j=1}^{N_t}\tilde u(t, Z_t^j)\,\mathcal Z_t^u \\ = \sum_{i=1}^{N_t}\left(\prod_{j\not = i}\tilde u(t, Z_t^j)\right)\,\left[A\tilde u(t, Z_t^i))-\tilde a(t)\,\tilde u(t, Z_t^i)(1-\tilde u(t, Z_t^i)) \right].
\end{align*}
Now the last quantity is zero because for any $x$
\begin{multline*}
A\tilde u(t, x)-\tilde a(t)\,\tilde u(t, x)(1-\tilde u(t, x)) \\= A u(T-t, x)- \partial_t u(T-t, x)- a(T-t)\,u(T-t, x)(1- u(T-t, x)),
\end{multline*}
which is zero by \eqref{eq:pde}.
So we have proved
$$
\lim_{\varepsilon \downarrow 0} \frac 1\varepsilon(Q^T(\mathcal Z^u_{t+\varepsilon}\mid \mathcal F_t) -  \mathcal Z^u_{t})=0.
$$
We would now like to take expectations inside the limit. 
Since $u$ and so $\mathcal Z^u$ take values in $[0,1]$, we first have
\begin{eqnarray*}
\left|\frac 1\varepsilon(Q^T(\mathcal Z^u_{t+\varepsilon}\mid \mathcal F_t) -  \mathcal Z^u_{t})\right| &\le&  \left|\frac1\varepsilon\left(\prod_{i=1}^{N_t}\bar P_{\varepsilon} \tilde u(t,Z_t^i)-\mathcal Z^u_{t}\right)\right| +2MN_t+ N_t^2 M^2 \varepsilon.
\end{eqnarray*}
Now because $\bar A\tilde u = a(t)( \tilde u-\tilde u^2)$, $\bar A\tilde u$ takes values in $[0,M]$, and since
$$
\frac{\bar P_{\varepsilon} \tilde u(t,x)-\tilde u(t,x)}\varepsilon =\frac{1}{\varepsilon}\int_0^{\varepsilon} \bar P_s\bar A\tilde u(t,x)\, ds, 
$$ 
we get 
$$
0\le \frac{\bar P_{\varepsilon} \tilde u(t,x)-\tilde u(t,x)}\varepsilon \le M. 
$$ 
So we can write
$$
\frac1\varepsilon\left(\prod_{i=1}^{N_t}\bar P_{\varepsilon} \tilde u(t,Z_t^i)-\mathcal Z^u_{t}\right)=\frac{H(\varepsilon)-H(0)}\varepsilon
$$
where 
$$
H(\varepsilon) = \prod_{i=1}^{N_t}(x_i+\varepsilon y_i),
$$
with $x_i = \tilde u(t,Z_t^i)$ and $y_i= \frac{\bar P_{\varepsilon} \tilde u(t,Z_t^i)-\tilde u(t,Z_t^i)}\varepsilon$, so that $0\le x_i\le 1$ and $0\le y_i\le M$. This shows that for any $z\in[0,\varepsilon]$
$$
0\le H'(z) \le H'(\varepsilon) = \sum_{i=1}^{N_t} y_i \prod_{j\not=i}(x_j+\varepsilon y_j)\le N_t M (1+\varepsilon M)^{N_t-1}.
$$
Then by the Mean Value Theorem
$$
\left|
\frac1\varepsilon\left(\prod_{i=1}^{N_t}\bar P_\varepsilon \tilde u(t,Z_t^i)-\mathcal Z^u_{t}\right)\right|=\left|\frac{H(\varepsilon)-H(0)}\varepsilon\right|\le N_t M (1+\varepsilon M)^{N_t-1}.
$$
Finally we get
\begin{eqnarray*}
\left|\frac 1\varepsilon(Q^T(\mathcal Z^u_{t+\varepsilon}\mid \mathcal F_t) -  \mathcal Z^u_{t})\right| &\le&  N_t M (1+\varepsilon M)^{N_t-1}+2MN_t+ N_t^2 M^2 \varepsilon=:S_t(\varepsilon).
\end{eqnarray*}
Since under $Q^T_x$, $N_t$ is dominated by the number of lineages at time $t$ in a Yule process with birth rate $M$ started at 1, it is geometrically distributed and so there is $\varepsilon_0$ such that for any $\varepsilon \in [0, \varepsilon_0]$, $S_t(\varepsilon)\le S_t(\varepsilon_0)$ and $\EE(S_t(\varepsilon_0))<\infty$. Then the Dominated Convergence Theorem ensures that 
$$
\lim_{\varepsilon \downarrow 0} \frac 1\varepsilon(Q^T_x(\mathcal Z^u_{t+\varepsilon}) -  Q^T_x(\mathcal Z^u_{t}))=0.
$$
This proves that $Q^T_x(\mathcal Z^u_{t})$ is constant as a function of $t$, so that 
$$
u(T, x)=Q^T_x(\mathcal Z^u_{0}) =Q^T_x(\mathcal Z^u_{T}) = Q^T_x\left(\prod_{i=1}^{N_T} u(0, Z_T^i)\right)=Q^T_x\left(\prod_{i=1}^{N_T} g(Z_T^i)\right)= v(T,x),
$$
which yields the announced result. 
\end{proof}
\begin{rmk} 
\label{rmk:domain}
By the branching property,
\begin{eqnarray*}
v(T+\varepsilon, x) -v(T,x) &=& -v(T,x)+ Q^{T+\varepsilon}_x\left(\prod_{i=1}^{N_{T+\varepsilon}} g(Z_{T+\varepsilon}^i), N_\varepsilon =1\right)+ a(T)\varepsilon\,v(T,x)^2+o(\varepsilon)\\
	&=& -v(T,x)+(1-a(T)\varepsilon) \,E_x (v(T, Z_\varepsilon))+ a(T)\varepsilon\,v(T,x)^2+o(\varepsilon)\\
	&=& E_x (v(T, Z_\varepsilon)) - v(T,x) + a(T)\varepsilon\, (v(T,x)^2-v(T,x))+o(\varepsilon).
\end{eqnarray*}
So for any $t,x\ge 0$,
$$
\lim_{\varepsilon \downarrow 0}\frac{(v(t+\varepsilon,x) - v(t,x))-(P_\varepsilon v(t,x) - v(t,x))}{\varepsilon} = a(t)\, (v(t,x)^2-v(t,x)).
$$
If we could prove that $v(\cdot, x)$ is of class $\cC^1$ or that $v(t,\cdot)\in \mathcal E_A$, then the RHS would equal $\partial_t v - Av$ and $v$ would indeed be solution to \eqref{eq:pde}.
%First note that for any $f, h\in \mathcal E_A$, for any $t\ge 0$, $fh$ and $P_tf$ also belong to $\mathcal E_A$. In particular, since $Af$ and $Ah$ are bounded, $A(fh) = f\,Ah + h\,Af$ ET NON and $A(P_tf) =P_t(Af)$ are also bounded. Since $g\in \mathcal E_A$ and $\mathcal E_A$ is, as we just saw, closed under the action of the semigroup $P_t$ and of finite products, we get that for any $T\ge 0$,
%$$ x\mapsto Q^T_x\left(\left.\prod_{i=1}^{N_T} g(Z_T^i) \right|(N_t;0\le t\le T)\right)$$
%belongs to $\mathcal E_A$. Taking expectations, $u(T, \cdot)\in \mathcal E_A$ by dominated convergence. After dividing by $\varepsilon$, we know from the preceding paragraph that the rhs converges as $\varepsilon \to 0$, to  $Au +a(T)(u^2 - u)$. This shows that $u(\cdot, x)$ is differentiable at $T$ and that $u$ satisfies  \eqref{eq:pde} (so that $u\in \mathcal E_A'$).
   \end{rmk}

\subsection{Proper solutions for the $\infty$-pop. \Sm equation}
\label{sec:uniqueness-weak-infinite}

In addition to the assumptions of Theorem \ref{thm:uniqueness-weak}, we now assume Grey's condition. 
Under this assumption, $0$ is accessible and since $Y$
is assumed to be (sub)critical
$$
 P_x (T_0<\infty) =1 \qquad x\ge 0,
$$
where we recall that $T_0=\inf\{t\ge 0: Z_t=0\}$.

As in the proof of Theorem \ref{thm:uniqueness-weak}, we start with a general lemma, which is the $\infty$-pop. analog of Lemma \ref{lem:uniqueness-pde}. We make the same general assumptions with the  notable difference that
we only assume that $a$ is only defined on $(0,T)$. 
We denote by $T_{\mathscr M}$ the time of mass extinction of the branching particle system with birth rate $\tilde a$ defined on $[0,T)$, i.e., the first time when all particles carry zero mass. 
Under $Q_x^T$, we denote by ${\mathscr M}$ the event $\{T_{\mathscr M}<T\}$.
 \begin{lem} 
 \label{lem:uniqueness-pde-2}
 Assume that $\int_{(0,T)}a(t) \,dt=\infty$.
Then there exists at most unique solution $u\in \mathcal E_A'$ to the PDE (or IPDE) 
\begin{eqnarray}
 \partial_t u \ = \ A u \ + \ a(t) (u^2-u) \label{eq:pde}
\end{eqnarray}
defined for $t>0$, such that $\limsup_{t\downarrow 0}\sup_{y\in [x,\infty)}u(t,y)<1$ for all $x>0$ and $u(t,0)=1$ for all $t> 0$. In addition this solution is given for any $T>0$ by 
$$
u(T, x)=Q^T_x({\mathscr M}).
$$
 \end{lem} 
 Before proving this lemma, we wish to state the relevant corollary regarding proper (weak) solutions to \eqref{eq:Sch} when $A = \lambda A^\psi$.
 Recall $M^{\tr}$ is the entrance measure at 0 of the branching particle system conditional on the genealogy $\tr$, where particle masses evolve like the CSBP $Z$ with branching mechanism $\psi$ (and at each branching event in $\tr$, the particle with mass $x$ undergoing division, splits into two particles, each with mass $x$). 

  \begin{teo}
 \label{thm:uniqueness-weak-infinite}
Let $\psi$ be the Laplace exponent of a spectrally positive and (sub)critical L\'evy process such that $1/\psi$ is integrable at $\infty$. Then 
\begin{enumerate}
\item[(i)] There exists a unique proper weak solution $(\mu_t;t>0)$ to the \Sm equation \eqref{eq:Sch}. 
\item[(ii)] For any $T> 0$, $\mu_T$ is the law of $M^{\bf T}(\mathscr{M}^c)$, where $\bf T$ is the time-inhomogeneous binary tree started at 0 with one particle, stopped at time $T$ and with birth rate $\tilde a(t)=1/(T-t)$.
\end{enumerate}
 \end{teo}
\begin{proof}
The proof follows the same lines as the proof of Theorem \ref{thm:uniqueness-weak}, but here in the absence of initial condition. If $(\mu_t;t>0)$ is a weak solution to the \Sm equation \eqref{eq:Sch}, then its Laplace transform $u(t,\lambda)$ is solution of \eqref{eq:pde}.
Further, the fact that the solution is proper implies that
\[\limsup_{t\downarrow 0}\sup_{y\in [x,\infty)}u(t,y) = \limsup_{t\downarrow 0} u(t,x)  \ < \ 1\] 
and as a consequence of the previous lemma $u(T,x) = Q^T_x({\mathscr M})$.
The same application of the branching property as in the proof of Theorem \ref{thm:uniqueness-weak} shows that
$$
 Q^T_x({\mathscr M}) = \EE\left[\exp\left(-x M^{\bf T}(\mathscr{M}^c)\right)\right], 
$$ 
where the expectation is taken wrt to the binary tree $\bf T$ with branching rate $\tilde a(t)=1/(T-t)$. More specifically, we can write conditionally on ${\bf T} = \tr$ 
$$
\mathbbm{1}_{\mathscr M}=\exp(-\sum_i \chi_i),
$$
where the sum is taken over the atoms of the Poisson process of branching particles with intensity $xM^\tr$ and $\chi_i=0$ if the $i$-th particle has zero mass in its descendance at time $T$ and $+\infty$ otherwise. The result is obtained by an application of the exponential formula and taking expectation wrt $\bf T$.

As a consequence, we get that $\mu_T$ is the law of $M^{\bf T}(\mathscr{M}^c)$.
Finally, it remains to show the existence of a proper solution. One option consists in checking that 
$\cL(M^{\bf T}(\mathscr{M}^c))$ is solution. Alternatively, we will provide a construction 
through the MK-V approach in the next section. 
\end{proof}

\begin{proof}[Proof of Lemma \ref{lem:uniqueness-pde-2}]
Let us first prove the uniqueness part of the statement. Let $u$ be a solution to \eqref{eq:pde} with the requested properties. Following the proof of Lemma \ref{lem:uniqueness-pde}, $Q^T_x(\mathcal Z^u_{t})$ is constant as a function of $t\in [0, T-\varepsilon]$ for any $\varepsilon\in (0,T)$, so that
\be\label{eq:desintegration}
u(T, x)=Q^T_x(\mathcal Z^u_{0}) =Q^T_x(\mathcal Z^u_{T-\varepsilon}) = Q^T_x\left(\prod_{i=1}^{N_{T-\varepsilon}} u(\varepsilon, Z_{T-\varepsilon}^i)\right).
\ee
If we set $X_\varepsilon:=\prod_{i=1}^{N_{T-\varepsilon}} u(\varepsilon, Z_{T-\varepsilon}^i)$, we can write
$$
u(T, x)=Q^T_x (X_\varepsilon \mathbbm{1}_{T_{\mathscr M}< T}) + Q^T_x (X_\varepsilon\mathbbm{1}_{T_{\mathscr M}\ge T}).
$$
On $\mathscr M$, because $u(\varepsilon ,0)=1$, $X_\varepsilon =1$ for any $\varepsilon$ such that $T_{\mathscr M}<T-\varepsilon$. By dominated convergence, 
$$
\lim_{\varepsilon \downarrow 0} Q^T_x (X_\varepsilon \mathbbm{1}_{\mathscr M}) =Q^T_x (\mathscr M^c).
$$
Then it only remains to show that 
$$
\lim_{\varepsilon \downarrow 0} Q^T_x (X_\varepsilon \mathbbm{1}_{\mathscr M^c}) =0.
$$
First observe that because $a$ is not integrable in the neighborhood of $0+$, the birth rate $\tilde a$ is not integrable in the neighborhood of $T-$, so that a.s. $\lim_{\varepsilon \downarrow 0} N_{T-\varepsilon}=+\infty$. On $\{T_{\mathscr M}\ge T\}$, there is at least one particle born before $T$ which carries positive mass up until time $T$. The probability that the mass of this particle vanishes exactly at time $T$ is zero. As a consequence, there is $\eta,\epsilon>0$ such that this focal particle has mass larger than $\eta$ on $[T-\epsilon, T]$. Let $t_n$ the times at which the focal particle gives birth, where $(t_n)$ is an increasing sequence converging to $T$. Let $\eta_n\ge \eta$ be the mass carried by the focal particle at time $t_n$. Then conditional on $(t_n)$ and $(\eta_n)$, let $\alpha_n$ denote the probability that the particle born at time $t_n$ with mass $\eta_n$ carries mass always larger than $\eta/2$ between $t_n$ and $T$. Since $Z$ is Feller, the sequence $(\alpha_n)$ is bounded away from 0 and by independence of these particles conditional on $(\eta_n)$, infinitely many of them carry mass larger than $\eta/2$ on $[t_n, T]$. As a consequence, a.s. on $\{T_{\mathscr M}\ge T\}$, there is $\eta>0$ such that 
$$
X_\varepsilon \le \prod_{i=1}^{N^\eta_{T-\varepsilon}} u(\varepsilon, Z^{\eta,i}_{T-\varepsilon})
$$
where $N^\eta_{t}$ is the number of particles at time $t$ which carry more than $\eta/2$ on $[t, T]$ and  $(Z^{\eta,i}_{t})$ are their masses, which satisfy
$$
\lim_{\varepsilon \downarrow 0} N^\eta_{T-\varepsilon}=+\infty \quad \mbox{ and } \quad  Z^{\eta,i}_{T-\varepsilon}\ge \eta/2.
$$ 
Since $\limsup_{\varepsilon\downarrow 0}\sup_{[\eta/2,\infty)}u(\varepsilon,x)<1$, $\lim_{\varepsilon\downarrow 0} X_\varepsilon=0$ a.s. on ${\mathscr M}^c$. The result follows from dominated convergence.

The fact that such defined $u$ satisfies \eqref{eq:pde} is due to the same reasoning as in Remark \ref{rmk:domain}, except that here $u(\cdot, x)$ is of class $\cC^1$ (since $T_0$, and so $T_{\mathscr M}$, has a continuous density), so we can conclude that $u(t,\cdot)\in \mathcal E_A$ and that indeed $\partial_t u = Au +a(t)\,(u^2-u)$. 
%The reason why $u(t,\cdot)\in \mathcal E_A$ is due to the fact that $v$ defined by
%$$ v(x, (N_t;0\le t\le T))= Q^T_x({\mathscr M}| (N_t;0\le t\le T))$$
%is in $\mathcal E_A$ (with $Av=0$).
\end{proof}

\subsection{The self-similar case}
\label{sec:uniqueness-weak-self-similar}

Here,  we assume that $A$ is the generator of the stable CSBP $Z$ with Laplace exponent $\psi(x)=cx^\gamma$, for $\gamma\in (1,2]$
and $c>0$. 
For any real number $r$, we denote by $Z^{(r)}$ the CSBP with branching mechanism $\psi(\lambda)-r\lambda$, which is the Feller process with generator $A_r$ defined by
$$
A_r f(x)=Af(x) +rx f'(x)\qquad x\ge0.
$$ 
Set 
$$
\beta = \frac{1}{\gamma -1}.
$$
Here we denote by $Q_x^{(\beta)}$
the law of a branching particle system started with a single particle with mass $x$, where particles branch \emph{at rate 1} and masses follow independent copies, \emph{not of the original CSBP (with Laplace exponent $c x^\gamma$), but of the CSBP $Z^{(\beta)}$}.
Similarly, for any \emph{infinite} binary tree $\tr$ embedded in continuous time, $Q_x^{(\beta),\tr}$ now denotes the law of the particle system started with one particle with mass $x$ at time 0, where particle masses evolve like independent copies of $Z^{(\beta)}$ and where the genealogy of particles is given by $\tr$. Consistently, the entrance measure at 0 of the branching particle system with genealogy $\tr$ is 
 $$
 M^{(\beta),\tr} = \lim_{x\downarrow 0} x^{-1} Q_x^{(\beta),\tr}.
 $$  
Here $\mathscr M$ denotes the event of total mass extinction, i.e., the event $\mathscr{M} =\{T_{\mathscr{M}}<\infty\}$ that \emph{after some finite time} all particles have mass $0$. 
\begin{teo}
\label{thm:self-similar-weak}
%Under the assumptions mentioned above, the unique solution $u$ to \eqref{eq:pde} advertized in Lemma \ref{lem:uniqueness-pde-2} can also be expressed as 
%$$
%u(t,x)= h(x/t^\beta), \ \ \mbox{where \  \ } h(x)=Q^{(\beta)}_x(\mathscr{M}).
%$$
%In particular, for any $x\ge 0$, $h(x) \leq \exp(-x(\beta/c)^\beta)$ and $h$ is the unique solution in $\mathcal E_{A_\beta}$ to
%\begin{eqnarray}
%A_\beta h \ + h^2 - h  \ = \ 0 \nonumber \\
%h(0)  = 1, \ \    \limsup_{x\to\infty} h(x) <1\label{eq:ODE-branching}.
%\end{eqnarray}
%%$$\forall x\geq0,  \ \ h(x) \leq \exp(-x(\beta/c)^\beta), \ \ \mbox{ and }  \ \ h'(0)= -M(\mathscr{M}^c).$$
If $(\mu_t;t\ge 0)$ is the unique proper solution of the \Sm equation \eqref{eq:Sch} advertized in Theorem \ref{thm:uniqueness-weak-infinite}, then 
\begin{enumerate}
\item[(i)] $\mu_t$ is the law of $t^{-\beta}\Upsilon$, where
\be \label{eq:def-upsilon}
\Upsilon =  M^{\bf T, (\beta)}(\mathscr{M}^c),
\ee
where $\bf T$ denotes the Yule tree, i.e., the pure-birth tree with unit birth rate. 
\item[(ii)] Let $h(x) = \exp(-x \Upsilon)$. Then 
$h(x) \leq \exp(-x(\beta/c)^\beta)$ and $h$ is the unique solution in $\mathcal E_{A_\beta}$ to
\begin{eqnarray}
A_\beta h \ + h^2 - h  \ = \ 0 \nonumber \\
h(0)  = 1, \ \    \limsup_{x\to\infty} h(x) <1\label{eq:ODE-branching}.
\end{eqnarray}
\end{enumerate}

\end{teo}
\begin{proof}
We first recall some known facts about $Z^{(r)}$ (see for example \cite{L08}). For any $x\ge 0$, the law of $Z^{(r)}$ started at $x$ is denoted $P^{(r)}_x$. It is well-known that $0$ is absorbing for $Z^{(r)}$ and that if $T_0$ denotes the absorbing time of $Z^{(r)}$ at 0, then
$$
P^{(r)}_x(T_0<t) = e^{-x\varphi_r(t)}\qquad x,t\ge 0,
$$
where $\varphi_r$ is the inverse of 
$$
\phi_r(\lambda) = \int_\lambda^\infty \frac{dx}{\psi(x)-rx} = -\frac\beta r \ln\left( 1-\frac{r}{c\lambda^{1/\beta}}\right)
$$
when $r\not=0$, and if $r=0$,
$$
\phi_0(\lambda) =\frac{\beta/c}{\lambda^{1/\beta}}.
$$
This yields
$$
\varphi_r(t) = \left( \frac{r/c}{1-e^{-rt/\beta}}\right)^{\beta}
$$
when $r\not=0$, and if $r=0$,
$$
\varphi_0(t) = \left(\frac{\beta/c}{t}\right)^{\beta}.
$$
In particular, the probability of extinction (of a non-branching particle) is
$$
P^{(r)}_x(T_0<\infty) =  \exp(-x(r/c)^\beta) 
$$
when $r>0$ and is 1 if $r\le 0$.
More specifically, 
$$
E^{(r)}_x(e^{-\lambda Z^{(r)}_t} )= e^{-x\varphi_r(t+\phi_r(\lambda))},
$$
where
\begin{equation}
\label{eqn:heavy}
\varphi_r(t+\phi_r(\lambda) )= \left( \frac{r/c}{1-e^{-rt/\beta}\left(1-\frac{r}{c\lambda^{1/\beta}}\right)}\right)^{\beta}
\end{equation}
when $r\not=0$.

Now we wish to compare the two branching particle systems. We will refer to the $(0,T)$-system as the time-inhomogeneous branching particle system with inhomogeneous branching rate $\tilde a(t) = a(T-t) = 1/(T-t)$ blowing up at time $T$ and masses evolving as independent copies of $Z^{(0)}$. 
We will refer to the $(\beta,\infty)$-system as the time-homogeneous particle system where particles branch at rate 1 and masses evolve as independent copies of $Z^{(\beta)}$. For either system, the genealogy of particles can be represented by the infinite binary tree, and for each node $v$ of the infinite binary tree, we record the corresponding branching time $U_0 (v)$ (resp. $U_\beta (v)$) and the mass of the corresponding particle just before it splits $M_0 (v)$ (resp. $M_\beta (v)$) in the $(0,T)$-system (resp. the $(\beta,\infty)$-system). 
Recall from Lemma \ref{lem:uniqueness-pde-2} that the probability that all masses go extinct in the $(0,T)$-system is $u (T,x)$. We denote by $h(x)$ this probability in the $(\beta,\infty)$-system. We claim that $u(T,x) = h(x/T^\beta) $.

First, it is straightforward to check that the first branching time $U_0$ of the first particle in the $(0,T)$-system is uniformly distributed in $(0,T)$, so that using \eqref{eqn:heavy} with $r=\beta$, we get
\begin{eqnarray*}
E_{T^\beta x} (e^{-\lambda Z^{(0)}_{U_0} /(T-U_0)^\beta} )&=& \frac 1 T \int_0^T ds\, \exp\left\{-T^\beta x\, \varphi_0(s+ \phi_0(\lambda/(T-s)^\beta))\right\}\\
	&=& \int_0^\infty du\, e^{-u}\,\exp\left\{-T^\beta x\, \varphi_0(T(1-e^{-u})+ \phi_0(\lambda/T^\beta e^{-\beta u}))\right\}\\
	&=& \int_0^\infty du\, e^{-u}\,\exp\left\{-T^\beta x\left(
	\frac{\beta/c}{
	T(1-e^{-u})+ \frac{\beta/c}{\lambda^{1/\beta}/T e^{- u}})
	}
	\right)^\beta
	\right\}\\
	&=&  \int_0^\infty du\, e^{-u}\,\exp\left\{-x\left(
	\frac{\beta/c}{
	1-e^{-u}+ e^{- u}\frac{\beta/c}{\lambda^{1/\beta} })
	}
	\right)^\beta
	\right\}\\
	&=& \int_0^\infty du\, e^{-u}\,e^{-x\varphi_\beta (u+\phi_\beta(\lambda))}\\
	&=& E_x(e^{-\lambda Z^{(\beta)}_{U_\beta}}),
\end{eqnarray*}
where $U_\beta$ is an independent exponential variable with parameter 1. By an immediate induction, we see that if the $(0,T)$-system starts with mass $T^\beta x$ and the $(\beta,\infty)$-system starts with mass $x$, then the sequence of rescaled masses $(w^{(0)}(v)/(T-U^{(0)}(v))^\beta)_v$ indexed by the binary tree is equally distributed as the sequence $(w^{(\beta)}(v))_v$. Now by a similar argument as the one used in the proof of Lemma \ref{lem:uniqueness-pde-2}, it can be seen that in both cases, there is extinction of mass iff all masses are zero except in a finite number of nodes $v$ of the genealogy. This shows that $u(T,T^\beta x) = h(x)$, so that $u(T,x)=h(x/T^\beta)$, as claimed earlier.
Finally, 
the same application of the branching property as in the proof of Theorem \ref{thm:uniqueness-weak-infinite} shows (\ref{eq:def-upsilon}).
\medskip

Now let us show the properties of $h$ stated in (ii) of the theorem. Let $\mathcal Z_t$ denote the state (e.g., the empirical measure) at time $t$ of the $(\beta,\infty)$-system. Let $(P_t)$ denote the semigroup of $Z^{(\beta)}$, that we now simply denote $Z$ (more generally, we will omit the $\beta$
superscript until the end of the proof). Since the first branching time $\tau$ of the initial particle is independent of the dynamics of its mass, by the branching property
$$
h(x)=  E_x(h(Z_t))\,\PP(\tau >t) + \int_0^t ds\, \PP(\tau \in ds) \, E_x(h(Z_s)^2).
$$ 
Re-arranging, we get for any $t>0$
$$
\frac{1}t(P_th(x) - h(x)) =  \frac{1-e^{-t}}{t}\,P_th(x) - \frac 1 t \int_0^t ds\, e^{-s} \, P_s(h^2)(x)
$$
Since $h$ is bounded continuous and $Z$ is a Feller process, the right-hand side converges as $t\downarrow 0$ to $h(x)-h(x)^2$. Then $h$ is in the domain of $A_\beta$ and we have
$$
A_\beta h(x) = h(x)-h(x)^2\qquad x\ge 0.
$$
Note that $h(x)\in [0,1]$ and $h(0) =1$. Also notice that $\mathscr M \subset \mathscr E$, where $\mathscr E$ is the event that the mass of a single (randomly chosen, say) lineage goes to 0. Now $Q_x(\mathscr E)= P^{(\beta)}_x(T_0<\infty) =  \exp(-x(\beta/c)^\beta) 
$, which yields 
 $$
 h(x) = Q_x( \mathscr{M})\le Q_x( \mathscr{E}) =  \exp(-x(\beta/c)^\beta).
 $$
As a consequence $\lim_{x\to\infty} h(x) =0$ and so $h$ is a solution to \eqref{eq:ODE-branching}, which yields the existence part of the statement.

For the uniqueness part, it is sufficient to note that if $h$ is solution of the ODE, then $u(t,x)=h(x/t^{\beta})$ is solution of the IPDE (\ref{eq:pde}). Since the solution of this equation is unique, the result follows.
\end{proof}

\section{Finite population McKean--Vlasov equation }
\label{sect:MK-V}
In this section, we only assume that $\psi\in\Hc$, which notably encompasses the case studied in the previous section, i.e., when $\psi$
is the Laplace transform of a spectrally positive (sub)critical L\'evy process.

For any $u>0$, $\theta_u$ will be the time shift operator by $u$ so that
$\theta_u\circ f(t) \ =  f(t+u)$ for any generic function $f$ of time.

We fix $\delta>0$ and $\nu\in M_P(\R^+)$. 
(Recall that we think of $\delta$ as the inverse population size.) We consider the McKean--Vlasov equation \eqref{eq:mckean}.
 As already mentioned in the introduction, one may think of $(x_t; t\geq0)$ as the evolution of the mass of a typical cluster in the population described by the \Sm equation (\ref{eq:Sch}). We start by giving a more formal definition in terms of a fixed point problem (see Proposition \ref{lem:unicity}).

\medskip

Let us consider the Skorohod space $D(\R^+,\R^+)$ (i.e., the space of c\`adl\`ag functions equipped with the Skorohod topology on every finite interval $[0,T]$). 
For every probability measure  $m$ on $D(\R^+,\R^+)$, define $\phi(m)$ the law of the process
$$ d y_t \ = \ -\psi(y_t) dt \ + \ \Delta J_t^{(\delta)} v_t, \  \ \cL(y_0)=\nu$$
where $(v_t; t\geq0)$ is a family of independent random variables with $v_t$ being distributed as $z_t$ -- the process with law $m$ evaluated at time $t$ -- and $J^{(\delta)}$
is an inhomogeneous Poisson process with rate $1/(t+\delta)$.
(More precisely, conditioned on the jump times $\{s_i\}$ of $J^{(\delta)}$, $\{v_{s_i}\}_i$ is a sequence of independent rv's with respective law $\cL(z_{t_i})$.)

We will say that $\left(x_t; t\geq0\right)$ is solution of the McKean--Vlasov equation 
(\ref{eq:mckean}) iff the law of
the process $x$ is a fixed point for the map $\phi$.

\begin{prop}[Uniqueness to MK-V]\label{lem:unicity}
The operator 
$\phi$ has a unique fixed point. As a consequence, there is at most one solution 
to the MK-V equation (\ref{eq:mckean}).
%in ${\mathcal E}$.
\end{prop}
\begin{proof}
We give a contraction argument analogous to Theorem 1.1. in \cite{Sn99}.
 For every  pair of measures $m^1,m^2$ on $D(\R^+,\R^+)$, and every $T\geq0$, we define the Wasserstein distance
$$ D_T\left( m_1,m_2 \right)\ = \ \inf\left\{ \E(\sup_{s\in[0,T]} |y_s^1 - y_s^2|) \ : \ {\mathcal L}(y^1) = m^1, \ {\mathcal L}(y^2)=m^2   \right\} $$
where the infimum is taken over every possible coupling between $y^1,y^2$ under the constraint  ${\mathcal L}(y^1) = m^1$ and ${\mathcal L}(y^2)=m^2$. 
Consider $z^1, z^2$ be two processes in $D(\R^+,\R^+)$ with respective laws $m^1$ and $m^2$ and define
\begin{eqnarray}
d x^i_t \ = \ - \psi(x^i_t) dt \ + \ \Delta J^{(\delta)}_t v^i_{t}, \ \  \cL(x_0^i) \ = \ \nu,
\end{eqnarray}
where $x_0^1= x_0^2$ and $(v^1_{s},v^2_s)$ are independent random variables with
${\mathcal L}(v^i_s) \ = \ {\mathcal L}(z_s^i), i=1,2$ and $(v^1_s,v^2_s)$ are coupled in a minimal way, i.e., 
for every $s\geq0$
$$
\E(|v^1_s - v^2_s |) \ = \ \inf\left\{\ \E(|a-b|) \ : \ {\mathcal L}(a) \ = \ z_s^1,     {\mathcal L}(b) = z_s^2 \right\}.
$$
(One can show that the minimum is attained by considering an approximating subsequence and using a standard tightness argument.)
Write $\Delta x_t = x_t^2 - x_t^1$ and note that $\Delta x_0 =0$. We have 
$$ \Delta x_t \ = \ \Delta x_0 \ - \ \int_{0}^t \left(\psi(x^2_s)-\psi(x^1_s) \right)ds \ + \ \sum_{s_i\leq t \ : \   \Delta J^{(\delta)}_{s_i}=1 } (v_{s_i}^2 - v_{s_i}^1 ).   $$
Since $\psi$ is positive and non-decreasing in $x$, the part of the dynamics $\left(\psi(x^2_s)-\psi(x^1_s) \right)ds$ can only reduce the distance between $x^1$ and $x^2$,  it is not hard to see
that
$$\sup_{s\leq t} |\Delta x_s| \leq \sum_{s_i\leq t : \Delta J^{(\delta)}_{s_i}=1} |v^2_{s_i} - v^1_{s_i}| $$
and thus
\begin{eqnarray*}
\E(\sup_{s\leq t} |\Delta x_s|) 
& \leq & \int_0^t \frac{1}{\delta+s}  \E\left( |v_s^2 - v_s^1  | \right) ds \\
& \leq & \frac{1}{\delta}\int_0^t   D_s(m^1,m^2) ds 
\end{eqnarray*}
where the last inequality follows from the choice of our coupling $(v_s^1,v_s^2)$.
This implies that
\begin{eqnarray*}
D_t(\phi(m_1),\phi(m_2)) 
& \leq & \frac{1}{\delta} \ \int_0^t D_s(m_1,m_2) ds
\end{eqnarray*}
By a simple induction, this implies that
\begin{eqnarray*}
D_t(\phi^{k}(m_1),\phi^{k}(m_2)) 
& \leq &  \frac{t^k}{\delta^k k !} D_t(m_1,m_2). 
\end{eqnarray*}
Thus if $m_1$ and $m_2$ are two fixed points for $\phi$, letting $k\to\infty$, yields that $m_1=m_2$.
\end{proof}

\begin{lem}\label{lem:relation}
Let $(x_t; t\geq0)$ be a solution of (\ref{eq:mckean}), so that that $(\mu_t:= \cL(x_t); t\geq 0)$ be is a (MK--V) solution of 
the \Sm equation (\ref{eq:Sch}).
Then  $\left(\mu_t; t\geq 0\right)$ is also a weak solution
to the \Sm equation (\ref{eq:Sch}) with  initial condition $\nu$ and inverse population size $\delta$. 
\end{lem}
\begin{proof}
By definition of the process $x$, for any test function $f$, a direct application of It\^o's formula yields
\[\E\left(f\left(x_t\right)\right) \ = \ \E\left(f\left(x_0\right)\right)  -  \int_{0}^t \E(\psi(x_s) f'(x_s) ) ds \ + \int_0^t \frac{1}{\delta+s} \int \E\left(f(x_s+u) - f(x_s)\right) \mu_s(du) \]
which can be rewritten as (\ref{eq:weak}).
\end{proof}

\subsection{The Brownian CPP}

Let us denote by $\cP$ a Poisson point process with intensity measure 
$dl\times \frac{dt}{t^2}$ on $\R^+_*\times\R^+_*$. %The CPP can be naturally recovered from Brownian motion: if we think of the $l$ coordinate as the local time of Brownian motion at level $0$, then for every $(l,t)\in\cP$, the coordinate $t$ is the depth of the downward excursion starting from level $0$ at local time $l$. The Brownian CPP  naturally defines a random tree 
We call Brownian Coalescent Point Process the ultrametric tree $\cT$ associated with $\cP$
$$\cT \ = \ \{(l,s) \in \R^+\times \R^+ \ : \  \exists  t \geq s \ \mbox{s.t. } (l,t) \in \cP  \} \cup \{(0,t); t\in\R^+\}  $$ 
equipped with the distance
\[ d_{\cT}\left((l',s'), (l,s)\right) \ =
\left\{ \begin{array}{cc} \ 2\max\{\bar t \ : \ (\bar l,\bar t)\in\cP, \ \mbox{s.t.} \ \bar l\in[l\wedge l,l'\vee l] \}  - ({s+s'})& \mbox{if $l\neq l'$},  \\ 
|s-s'| & \mbox{otherwise.} \end{array}\right.\] 
$\{(0,t); t\in\R^+\}\subset \cT$ will be referred to as the eternal branch of the tree.
See Fig. \ref{cpp} for a pictorial representation of the tree $\cT$  above level $\delta>0$.

 \begin{figure}[h] \includegraphics[width=18cm]{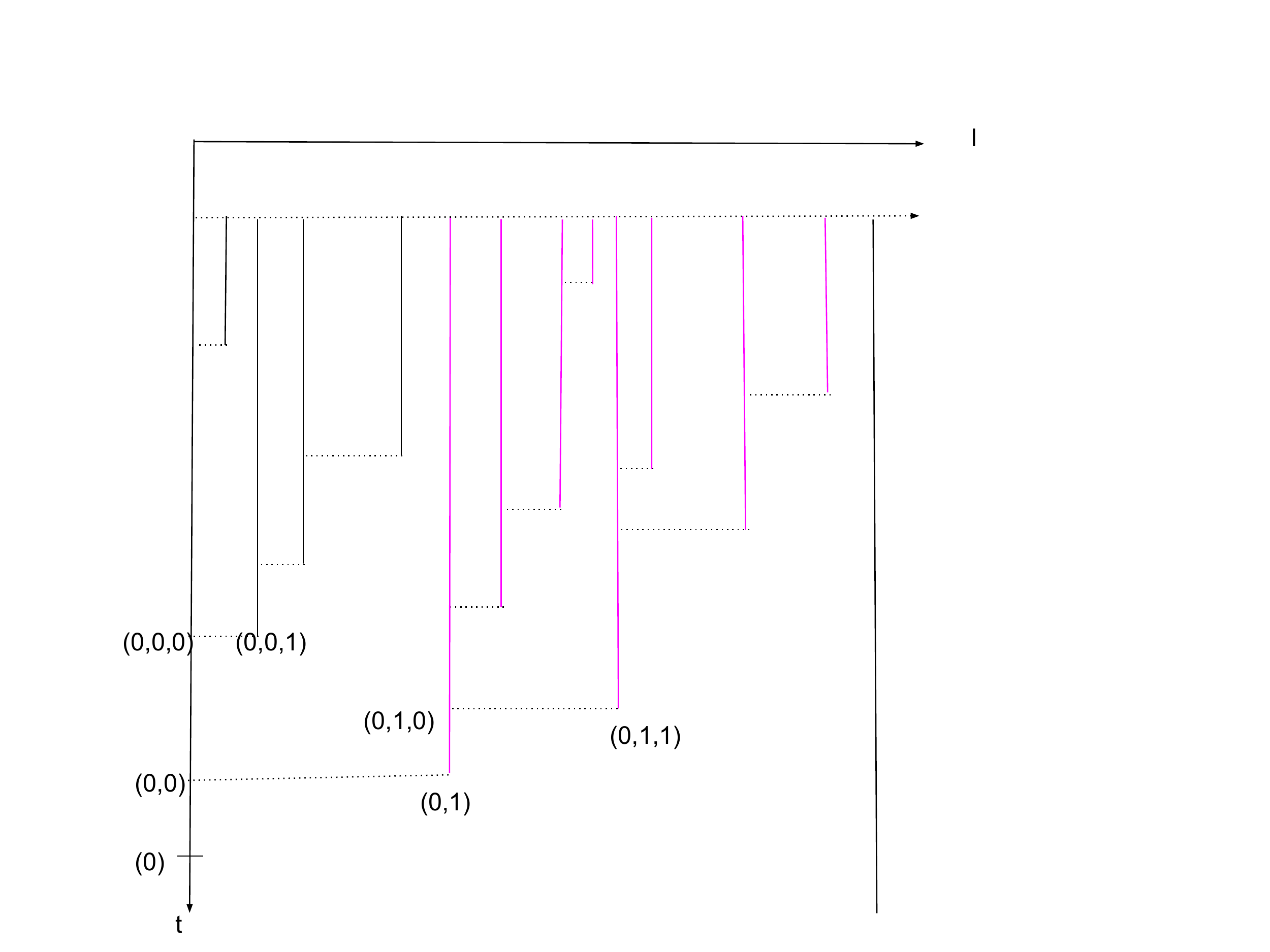} 
 \caption{Brownian CPP above a fixed level $\delta$. Points of $\cP$
 are related to the axis $\{t=0\}$ by a vertical branch (plain lines).
 Dotted lines correspond to the merging of two branches of the tree. {\it In the proof of Theorem \ref{teo:mckean}.}
 The left-most branch ``alive'' at time $T$ (the time coordinate of the point labelled $(0)$) is the black branch on the right hand side of the purple subtree. Upon coalescence of 
 the subtree rooted at $(0,0)$ and $(0,1)$, the two subtrees (black and purple) are equally distributed, and as a consequence, the mark at $(0,0)$
 and $(0,1)$ are equally distributed. {\it In the proof of Lemma \ref{cor:cv-to-self-similar}}. If $(0,T)$
 corresponds to the point labelled $(0)$, then $\{(l,\delta)\in D^{(\delta)}_{\cT}(0,T)\}$ 
 corresponds to the point of the black and purple subtrees with time coordinate $\delta$.} 
 \label{cpp}\end{figure}

\subsection{Marking the CPP and construction of a solution to MK-V}
\label{sect:cpp-construction}

% Let us now fix $\delta>0$, and equip each point $(x,\delta)\in \cT$ with an independent random variable distributed according to some arbitrary law $\nu$ on $[0,\infty]$. (Note that $\nu$ is a possibly  a point mass at $\infty$.) 

In this section, we describe a marking of the tree which will provide a solution to MK-V (Theorem \ref{teo:mckean}).

Fix $\delta>0$ and let $\{\zeta^{(\delta)}_i\}_i$ be a (possibly random) sequence in $\R^+$. 
Let $\{l_i,\delta\}_i$ be the elements in $\cT$ with time coordinate $\delta$, and assume 
that the $l_i$'s are listed in increasing order.
For each point $(l,s)$ of the tree $\cT$ with time coordinate $s\geq \delta$, we assign a mark $m_{l}(s)$ such that
(1) $m_{l_i}(\delta)=\zeta_i^{(\delta)}$ for every $i\in\N$, and (2) the marks with higher time coordinates are deduced (deterministically conditioned on $\cT$
and the initial marking) from the differential relation
\begin{eqnarray}
\forall t\geq \delta, \ d m_l^{(\delta)} (t) \ = \  -\psi(m_l^{(\delta)} (t)) dt \ + \ \sum_{(l',t)\in\cP: l'>l  } \sigma_{(l,l',t) } m^{(\delta)}_{l'}(t), \     \label{eq:sim}\\ 
\ \mbox{where   } \   \sigma_{(l,l',t) } \ = \ 1 \mbox{ if $\sup\{\bar t \  : \ \bar l \in( l,l'], (\bar l, \bar t) \in \cP \} \ = \ t$}, \  0 \ \mbox{otherwise.} \nonumber 
\end{eqnarray}
In words, we start by marking each point at level $\delta$ from left to right with the sequence $\zeta_i^{(\delta)}$; then the marks evolve according to the ODE $\dot x \ = \ -\psi(x) $ along each branch, and when two branches merge, we simply add up the values of the marks. The previous procedure defines a marking of the tree $\cT$
above time horizon $\delta$. 

\begin{rmk}
Let us assume here that $\psi$
is again the Laplace exponent 
of (sub)critical spectrally positive L\'evy process.

The previous marking is completely identical to the one 
involved in the definition of the function $F$ in (\ref{eq:laplace-general}).
This is obviously not coincidental. 

Let $(\mu_t; t\geq0)$ be the weak solution of the \Sm equation. Recall from Theorem \ref{thm:uniqueness-weak} that $\mu_T$ is the law of $F({\bf T}, (W_i); 1\le i \le N_T (\bf T))$, where $\bf T$ is the time-inhomogeneous tree started at 0 with one particle, stopped at time $T$ and with birth rate $\tilde a(t)=1/(T-t+\delta)$, the $(W_i)$ are iid with law $\nu$.

Now, it is not hard to see that the tree originated from $(0,T)$
and stopped at time horizon $\delta$ in the CPP is identical in law with ${\bf T}$.
From Theorem \ref{thm:uniqueness-weak}, this implies that $\left(\cL(\theta_\delta\circ m_0^{(\delta)}(t)); t\geq0\right)$ 
is a weak solution of (\ref{eq:Sch}). We will actually show more, namely that 
$\left(\theta_\delta\circ m_0^{(\delta)}(t); t\geq0\right)$ is solution to MK-V (and with no restriction on $\psi$).
See Theorem \ref{teo:mckean}.

This provides a natural interpretation for the expression of $\mu_t$
in terms of the MK-V equation.
\end{rmk}

The marks $(m_l^{(\delta)}(t); t\geq\delta, (l,t)\in{\cT})$ will be referred to as the {\it partial} marking of $\cT$ above level $\delta$ with initial condition 
$\{\zeta_i^{(\delta)}\}_i$.  When the initial marks $\{\zeta_i^{(\delta)}\}_i$ are distributed as independent copies with law $\nu$,
 the partial marking will be referred to as the partial marking above level $\delta$ with initial condition $\nu$.

Finally, a {\it full} marking of $\cT$ will refer to marks $(m_l(t);  (l,t)\in{\cT})$ defined on the whole tree $\cT$ such that for every $\delta>0$,
$(m_l(t); t\geq\delta, (l,t)\in{\cT})$ is a partial marking above level $\delta$ (with initial marking $\{m_{l_i}(\delta),  (l_i,\delta)\in\cT\}$).

\begin{teo}\label{teo:mckean}
Let us consider $m^{(\delta)}$ the partial marking above level $\delta$
with initial law $\nu$. Then
($\theta_\delta\circ m_0^{(\delta)}(t); t\geq 0)$ is solution to the MK-V equation (\ref{eq:mckean}) with initial condition $\nu$
and inverse population $\delta$.
\end{teo}

\begin{proof} 

It is enough to show that $\left(m_0^{(\delta)}(t); t\geq\delta\right)$ solves the McKean--Vlasov
\be \forall t\geq \delta, \ \ \bar x_t - \bar x_\delta \  = \ \underbrace{-\int_\delta^t \psi( \bar x_s) ds}_{\mbox{term I}} \ + \ \underbrace{\sum_{\delta \leq s_i\leq t \ : \  \Delta \bar J^{(0)}_{s_i}=1}  \bar v_{s_i}}_{\mbox{term II}}, \  \ \mbox{and}  \  \ \cL(\bar x_{\delta}) \ =  \ \nu\label{e:dynamics}\ee
where $\bar J^{(0)}$ is identical in law to a Poisson Point process with rate $1/t$ and 
where conditional on $\bar J^{(0)}$,
$\{\bar v_{s_i}\}$ is a collection of independent random variables with $\cL(v_s)\ = \ \cL(\bar x_s)$.

First, the initial condition is obviously satisfied.

Secondly, we note that in the absence of a coalescence event, $m^{(\delta)}_0(t)$
decreases at rate $-\psi(m_0(t))$ which exactly corresponds to term I in the dynamics (\ref{e:dynamics}).
For term II, we note that $m_0^{(\delta)}(t)$ experiences a jump upon a coalescence event (see the second term on the RHS of (\ref{eq:sim})). Recall that the Brownian CPP is defined as a Poisson Point process with 
intensity rate $dl\times dt/t^2$, and by definition of $m^{(\delta)}_0$ such a coalescence event occurs whenever the left-most branch ``alive'' at time $t$ with  a strictly positive $l$--coordinate   dies out (see Fig. \ref{cpp}). This occurs at a rate
$$ \frac{1}{t^2} /\int_{t}^\infty \frac{ds}{s^2} \ = \ \frac{1}{t} \ \  \mbox{at time $t$} $$
which exactly corresponds to the rate of $\bar J^{(0)}$ in (\ref{e:dynamics}).
Finally, by translation invariance and the independence structure in the Brownian CPP, the branch coalescing with the eternal branch $\{0\}\times\R^+_*$ carries a mark that is identical in law 
to $m^{(\delta)}_0(t)$, and independent of $m^{(\delta)}_0(t)$. See again Fig. \ref{cpp}.
\end{proof}

As a corollary of Proposition \ref{lem:unicity} and Theorem \ref{teo:mckean},  we get the following existence and uniqueness result.

\begin{cor}\label{cor:existence-uniqueness}
For $\delta>0$ and $\nu\in M_P(\R^+)$,
there exists a unique solution to the MK-V  equation (\ref{eq:mckean}).
\end{cor}

\subsection{Scaling} 

 For every $\tau>0$, define
the scaling map
$$F_\tau : (l,t) \to (\tau l, \tau t)$$ 
Fix $\gamma>1$. We set 
$$
\beta :=\frac{1}{\gamma -1}
$$
and we define for any $\nu\in M_F(\R^+)$ ${\mathcal S}^{\tau,\gamma}(\nu)$ as the push-forward of the measure $\nu$ by the map
$$ x\mapsto \tau^{-\beta}x.
$$
\begin{prop}[scaling]\label{cor:scaling}
For every $\tau>0$
\begin{enumerate}
\item[(i)] $\tilde \cT = F_\tau(\cT)$ is identical in law with $\cT$.
\item[(ii)] Assume that $\psi(x) = c x^\gamma$ for $c>0$ and $\gamma>0$.
Let $m^{(\delta)}$ be a partial marking above level $\delta$
with initial measure $\nu$. Define
$$(l,t)\in\tilde \cT, \ t\geq\delta \tau, \ \ \  \ \tilde m^{(\delta \tau)}_{l}(t) : \ = \ \frac{1}{\tau^{\beta}} m^{(\delta)}_{l/\tau}(t/\tau) $$ 
Then $\tilde m^{(\delta \tau)}$ is a partial marling of $\tilde \cT$ above level $\tau \delta$ with initial measure $\cS^{\tau,\gamma}(\nu)$.
\end{enumerate}
\end{prop}
\begin{proof}
(i) is a direct consequence of the definition of the Brownian CPP. (ii) is a consequence of the observation that
(a) for every $(l,\delta\tau)\in F_\tau(\cT)$, we have $\cL\left(\tilde m_l^{(\tau \delta)}(\tau \delta)\right) = \cS^{\tau,\gamma}(\nu)$, and
(b) along each branch of the tree $\tilde \cT$, the marking evolve according to the dynamics 
$$dx_t  \ = \ -c x_t^\gamma  dt$$ 
(because of the pre-factor $\frac{1}{\tau^{\beta}}$ in the definition of $\tilde m$) and at a coalescence point, marks add up.
(So that $\tilde m$ defines a partial marking on $\tilde \cT$ with initial marking $\cS^{\tau,\gamma}(\nu)$.)
%This shows that the markings $\bar m^{(\delta\tau)}$ and $\tilde m^{(\delta\tau)}$ are identical in law and  (\ref{id-scaling0})
%follows by construction of our coupling.
\end{proof}

\section{$\infty$-population McKean-Vlasov equation}
\label{sect:infty-pop}
In this section, we assume that $\psi \in \Hc$ and that Grey's condition holds. Then we can define the homeomorphism $q:(0,\infty)\to (0, V)$ as
$$
q:x\mapsto \int^\infty_x \frac{1}{\psi(s)} ds,
$$
with $V=q(0+)\in (0,+\infty]$. Define $\phi:(0,\infty)\to (0,\infty)$ as
$$
\phi(t)=\left\{
\begin{array}{cl}
q^{-1}(t) &\mbox{ if } t<V\\
0 &\mbox{ if } t\ge V.
\end{array}
\right.
$$
Then for any $x_0\in (0,+\infty]$, the ODE
 \[ \dot u = -\psi(u), \ \ u(0) = x_0 \]
has a unique solution on $\R^+$ given by 
$$
u(t) = \phi(t+q(x_0)),
$$
with $q(x_0)=0$ if $x_0=\infty$. Notice that the flow $x_0\mapsto \phi(t+q(x_0))$ is continuous and keep in mind that $\phi$ is the unique solution to 
\[ \dot u = -\psi(u), \ \ u_0 = \infty.  \]
%has a unique (finite) solution given by 
%\be \label{eq:inverse-eq} u_t \ =\ \inf\{u > 0 \ : \ \int^\infty_u \frac{1}{\psi(s)} ds < t\}<\infty.\ee
In this section, we will also make the extra assumption that $\psi$ is also convex. (so that $V=\infty$.) Note that if $\psi$
is the Laplace exponent of a spectrally positive L\'evy process the latter condition holds.

We will say that $(x_t; t>0)$ is an $\infty$--pop. solution of (\ref{eq:mckean})
if it is a solution of (\ref{eq:mckean}) for $\delta=0$ (without prescribing the initial condition at $t=0$).
More precisely, $x$ will be an $\infty$-pop solution iff for every $\tau>0$,
conditional on $x_\tau$, the process $(x_t; t \geq\tau)$
is identical in law to the solution of
\be d\bar x_t \ = \ -\psi(\bar x_t) dt \ + \ \Delta J^{(0)}_t \bar v(t); t\geq \tau;  \ \ \bar x_\tau = x_\tau \label{eq:mckeqn22}\ee
where $(\bar v_t)_{t\geq0}$ is a family of independent rv's with $\cL(\bar v_t) = \cL(\bar x_t)$.

Recall the definition of dust and proper solutions for the $\infty$-pop. \Sm equation. (See Definition \ref{def:weak-eq}.)
There is a natural extension of this definition to the MK-V equation. We will say that 
$(x_t; t>0)$ is a dust solution if it satisfies the $\infty$-pop. MK-V equation 
with $\delta=0$ and $\lim_{t\to0} x_t = 0$ in probability.
%, whereas $\P(x_t>0)>0$ for every $t>0$. 
Solutions are said to be proper otherwise.
By a direct application of It\^o's formula (analogous to Lemma \ref{lem:relation}), 
if $x$ is a dust (resp., proper) solution then
$(\cL(x_t); t>0)$ is a dust (resp., proper) solution of the \Sm equation. 
We state the two main results of this section.

\begin{teo}[Proper solutions]\label{teo:existence:global:solution}
\begin{itemize}
\item[(i)] There exists at least one proper solution to the $\infty$-pop. MK-V equation
and any such solution satisfies
\be
 \forall t>0, \ x_t \geq \phi(t) \ \ \mbox{a.s.} \label{growth-marking-0}\phantom{kzfgy ouregviquerzltv} \mbox{\rm(Growth condition)} 
 \ee
\item[(ii)] In the stable case  
\[\psi(x)=c x^\gamma \ \ \ \mbox{with $\gamma>1$},\]
there exists a unique proper solution $(x_t^{(0)}; t>0)$. 
%and satisfies the growth condition
%$x^{(0)}_t \geq  (\frac{2}{c(\gamma-1)t})^{\beta}$  for every $t>0$. 
Further,
\begin{enumerate}
\item[(Self-similarity)] For every $t>0$, $\cL(x_t^{(0)}) \ = \  \cL(\Upsilon/t^{\beta})$ where $\Upsilon:= x_1^{(0)}$.
\item[(Integrability)] For every $t>0$, \ $\E(x^{(0)}_t)<\infty$.
\item[(Measurability)] $\left(x^{(0)}_t; t>0\right)$ can be constructed on the same space as the Brownian CPP, and under this coupling,  it is measurable with respect to the $\sigma$-field
generated by the CPP.
\end{enumerate}
\end{itemize}
\end{teo}

%\begin{rmk}
%Note that when $\gamma\in(1,2]$ our definition of $\Upsilon$
%coincides with the one of Section \ref{}.\marginpar{ref missing} This is due to the fact that the law of the proper MK-V solution
%coincides with the proper solution of the \Sm equation. 
%\end{rmk}

\begin{teo}[Dust solutions]\label{teo:dust}
In the stable case $\psi(x) = c x^\gamma$ ($c>1$ and $\gamma>1$), there exist infinitely many dust solutions to the $\infty$-pop. MK-V
equation. Since the law of a dust solution to MK-V is also a weak dust solution, there exist infinitely many dust solutions to 
the \Sm equation.
\end{teo}

The rest of the section is mainly dedicated to the proof of those two theorems.
In Section \ref{ex:infinite}, we prove the existence of a proper solution and derive some of its properties (growth condition, measurability, self-similarity...). In Section \ref{sect:growth-condition} we show that 
any proper solution satisfies the growth condition (\ref{growth-marking-0}). 
This is shown by introducing what we call the marking associated to an $\infty$-pop. solution. 
In Section \ref{sect:uniqueness-sol}, we use this full marking to
prove uniqueness of a proper solution in the stable case, and we show all the required properties
of the solution. (This will show part (ii) of Theorem \ref{teo:existence:global:solution}).

Finally, in Section \ref{sect:applications}, 
we show that the long-term behavior of finite pop. solutions 
can be described in terms of the $\infty$-pop. Finally, we close 
this section with the proof of Theorem \ref{teo:dust}.

\subsection{Existence and construction of an $\infty$-pop. proper solution}\label{set:existence-of-infinite-pop}
\label{ex:infinite}
We start with an $\infty$-pop. analog of Theorem \ref{teo:mckean}
that will be exploited repeatedly throughout this section. 
\begin{teo}\label{teo:mckean2}
Let $m$ be a full marking of the CPP. Assume that for every $t>0$,  $\{m_{l}(t) : (l,t)\in\cT\}$
(assuming that the points are ranked according to the values of $l$ in increasing order) is a sequence of i.i.d. random variables distributed as $m_0(t)$. Then $(m_0(t); t>0)$
is an infinite solution to MK-V.
\end{teo}
\begin{proof}
By Theorem \ref{teo:mckean}, for every $t>0$, $(\theta_t\circ m_0(u); u \geq 0)$ is solution of 
\[\forall u\geq 0, \  d y_u \ = \ -\psi(y_u) du \ + \ \Delta J^t_u v_{u},  \ \ \mbox{with} \ \cL(y_0) = m_0(t), \]
where $\{v_u\}_{u>t}$ is an infinite collection of r.v. with $\cL(y_u)= \cL(v_u)$ for every $u\geq t$.  Equivalently, $(m_0(u); u \geq t)$ is identical in law
to
\[\forall u\geq t, \  d z_u \ = \ -\psi(z_u) du \ + \ \Delta J^0_u v_{u},  \ \ \mbox{with} \ \cL(z_t) = m_0(t), \]
where $\{v_u\}_{u>t}$ is an infinite collection of r.v. distributed with $\cL(z_u)= \cL(v_u)$ for every $u\geq 0$. Since this holds for any $t>0$, this $m_0$
is an $\infty$-pop. solution to MK-V.
\end{proof}

In order to construct a non-trivial $\infty$-pop. solution out of the CPP, we 
will consider $(m_l^{(\delta),+}(t); (l,t)\in\cT, \ t\geq\delta)$ the partial marking starting from level $\delta>0$ 
with initial measure  $\nu^{+}(dx) \ = \ \delta_{\infty}(dx)$,
i.e, we start with the initial condition $+\infty$ at level $\{t=\delta\}$.

\begin{prop}\label{lem:marking}
Let us consider a positive non-increasing sequence $(\delta_n)$ going to $0$. 
\begin{enumerate}
\item[(i)] For almost every realization of the CPP $\cT$,  for every $(l,t)\in\cT$, the sequence $(m_l^{(\delta_n),+}(t))$ is non-increasing
and if we define  $m_l^{+}(t)$ its limit, then 
\[0< \phi(t) \le m_l^{+}(t) <\infty.\]
% (\frac{1}{c(\gamma-1)t})^{\beta}.$$

\item[(ii)]
$m^+$ is a full marking of $\cT$ which is measurable with respect to the $\sigma$-field
generated by $\cT$ and does not depend on the sequence $(\delta_n)$.
%Further, the limiting marks do not depend on the choice of the sequence $(\delta_n)$.
\item[(iii)] This marking is maximal in the sense that for every full marking $m^-$  defined on $\cT$, for every $(l,t)\in\cT, \ m_l^+(t)\geq m_l^-(t)$.
\item[(iv)] $(m^+_0(t); t\geq0)$
is an $\infty$-pop. proper solution to MK-V.
\item[(v)] In the stable case, $\psi(x) = c x^\gamma$ (for $c>0$, $\gamma>1$),
for every $u>0$, 
$$
\cL\left( u^{\beta} m_0^+(u) \right)\ = \ \cL\left( m_0^{+}(1) \right).$$
\end{enumerate}
\end{prop}
\begin{proof}
We start with a monotonicity property of our marking of the Brownian CPP that we will use repeatedly 
throughout this proof.
Let us first consider two partial markings $m,\bar m$ of the CPP above a given level $\delta>0$. It is clear 
that if for every $(l,\delta)\in\cT$ we have $m_l(\delta)\geq \bar m_l(\delta)$, then for every $(l,t)\in\cT$
with $t\geq \delta$, we must have
$$ m_l(t) \geq \bar m_l(t). $$
Now, since $(\delta_n)$ is non-increasing, it follows that for every $(l_n,\delta_n)\in\cT$
$$ m_{l_n}^{(\delta_{n+1}),+}(\delta_n) < m_{l_n}^{(\delta_{n}),+}(\delta_n)=\infty $$
which implies that the sequence of marks $(m^{(\delta_n),+}_l(t))$ (for any $(l,t)\in\cT$) is non-increasing
and converges to a mark $m_l(t)<\infty$. (Note that since there are only finitely many coalescences in each compact time interval of $(0,\infty)$, Grey's condition ensures that the limiting marking is finite.)

Next, in the absence of coalescence along the vertical branch 
$[(0,t), (0,\delta_n)]$ in $\cT$ (for $\delta_n<t$), we have
$$ m_{0}^{(\delta_n),+}(t) \ = \  \phi( t-\delta_n).$$
Thus, ignoring all the coalescence events along the vertical branch  $[(0,t), (0,\delta_n)]$ ensures that  
$m_{0}^{(\delta_{n}),+}(t)$ is bounded from below by the RHS of the latter identity. Since coalescence
events can only add extra mass to the eternal branch (and by uniqueness of the solution to the ODE), the inequality in (i) follows after taking the limit $n\to\infty$.

\medskip

Let us now show (ii) and (iii). By continuity of the flow, it is not hard to check that $m^+$ defines a full marking of the tree $\cT$ in the sense prescribed in the beginning of Section \ref{sect:cpp-construction}.  (In other words, the property ``marks evolve according to $\dot x_t = -\psi(x_t)$ along branches and marks add up upon coalescence'' passes to the limit.) For details, we refer the reader 
to the proof of Lemma \ref{cor:cv-to-self-similar} below where we develop a similar argument in more detail.

Let us show that $m^+$ is independent of the choice of the sequence $(\delta_n)$.
Let $\bar \delta_n$ be another non-increasing sequence going to $0$, and let $\bar m^{+}$
be the limit of $m^{(\bar \delta_n),+}$ as $n\to\infty$. Going to a subsequence of $(\bar \delta_n)$
if necessary, one can always assume that $\bar \delta_n \leq  \delta_n$ for every $n$. Under this assumption, 
one gets $\bar  m^{(\bar \delta_n),+}_{l_n}(\delta_n) \leq m^{(\delta_n),+}_{l_n}(\delta_n)$
for every $(l_n,\delta_n)\in\cT$
and by similar monotonicity arguments as above this ensures that  $\bar m^+ \leq m^+$. Finally, assuming that  $\bar \delta_n \geq  \delta_n$ 
yields the reverse inequality. This shows that the limiting marking $m^{+}$
does not depend on the choice of the sequence $(\delta_n)$.

Further, by construction, for any $\delta_n\leq t$ and any $(l,t)\in\cT$, we must have $m^{-}_l(t) \leq m_{l}^{(\delta_n),+}(t)$ since 
at time $\delta_n$, the initial condition of the the marking $m^{(\delta_n),+}$ dominates the one of $m^-$. This completes the proof of the first part of (ii) and (iii), i.e., that $m^+$ is a maximal full marking of the CPP.  Finally, since
the initial condition of $m_0^{(\delta_n),+}$ is deterministic and 
the marking  $m^{(\delta_n),+}$ below level $\delta_n$
is determined deterministically from the tree $\cT$ above level $\delta$,
we deduce that the  $(m_0^+(t); t>0)$ is measurable with respect to $\sigma(\cT)$. 

\medskip

Let us now proceed with the proof of (iv). By the branching structure of the CPP,
at every $t>0$, all the marks at level $t$ are i.i.d. with law $\cL(m^+_0(t))$. 
By Theorem \ref{teo:mckean2}, $m_0^+$
is an $\infty$-pop. solution to MK-V. By the inequality in (i), the solution is proper (and goes to $\infty$
as $t\to 0$).

\medskip

Let us now show 
the scaling identity (v).

%Let us now show the scaling relation
%\be\label{id:scaling0} \cL(t^{\beta} m_0^\pm(t)) \ = \ \cL(m_0^\pm(1)).  \ee
%Let $\tau>0$, and recall the definition of $\nu^\tau$ of Corollary \ref{cor:scaling}. \
%One can directly check that 
%$$(\nu^{(\delta),\pm})^\tau = \nu^{(\delta \tau),\pm}$$
From Proposition \ref{cor:scaling} and the invariance of the initial condition $\delta_{\infty}(dx)$ under rescaling (i.e., $\cS^{\tau,\gamma}(\delta_\infty)= \delta_\infty$ for every $\tau>0$), 
we get for every $t\geq \delta_n \tau$
\be m_0^{(\delta_n\tau),+}(t) \ =_{\cL} \ \frac{1}{\tau^{\beta}} m^{(\delta_n),+}_{0}(t/\tau)\label{id:scaling}.\ee
Let $\tilde m^+$ be the limit of $m^{(\delta_n\tau),+}$. Since 
$m^{+}$ does not depend on the choice of the sequence $(\delta_n)$ (in particular if we replace $(\delta_n)$ by $(\tau \delta_n)$),
$$ \  m_0^+(t)  \ =  \tilde m^+_0(t) \  \ =_{\cL} \frac{1}{\tau^{\beta}}m^+_0(t/\tau), $$
where the latter identity follows from (\ref{id:scaling}).
This completes the proof of the scaling identity after taking $\tau=t$. 
\end{proof}

\subsection{Full marking associated to an $\infty$-pop. solution}
\label{sect:growth-condition}
In this subsection, we construct a full marking of the CPP
using an $\infty$-pop. solution to MK-V.
% As a fiOur main goal os the show Secondly, we show the growth condition  (\ref{growth-marking-0})
%Our goal is two-fold. First, using the is construction, we will provide a general criterium for the uniqueness of
%a proper  solution to the MK-V equation. Indeed, Proposition \ref{cor:how-to-prove-u} below allows to rephrase the question of uniqueness of a proper solution (for MK-V) in terms of a uniqueness result on the full marking
%of the CPP. . See Lemma \ref{lem:grwoth-condition} below. 

\begin{prop}\label{prop:marking-associated}
Let $(x_t; t>0)$ be an $\infty$-pop. solution to MK-V. There exists a unique full marking $(\cT, m)$ such that 
\be\cL(\{m_0(t);t>0\}) \ = \ \cL(\{x_t;t>0\}). \label{eq:rel-mkv}\ee
and at every level $s$, $\{m_{l}(s) \ : \ (l,s)\in\cT\}$ is a sequence 
with i.i.d. rv's with law $\cL(x_s)$. This marking $m$ will be referred to as 
the marking associated to the solution $x$.

\end{prop}

%In other words, In order to prove  Proposition \ref{cor:how-to-prove-u}, we will use the insight of Theorem \ref{teo:mckean}.
%When $\delta>0$, we showed a direct correspondence between 
%the solution to MK-V and the partial marking above level $\delta$ -- see Theorem \ref{teo:mckean}.  In the case when $\delta=0$,
%we now explain the correspondence between $\infty$-pop. solutions and 
%full markings of the CPP.  

%More precisely, we argue that the uniqueness
%of an $\infty$-pop. solution to MK-V boils down to proving uniqueness
%of a full marking of the CPP.

\begin{proof}
We start by proving existence. The proof goes along similar lines as Theorem \ref{teo:mckean2}. Let $(\delta_n)$ a  sequence decreasing to $0$. For every $n$, we can consider the partial marking $m^{(\delta_n)}$ above level $\delta_n$
with initial marking $\cL(x_{\delta_n})$.  By Theorem \ref{teo:mckean},
$(\theta_{\delta_n}\circ m^{(\delta_n)}_0(t); t\geq0)$ is identical in law with the solution
of the MK-V equation with inverse population $\delta_n$ and initial measure $\cL(x_{\delta_n})$.
Equivalently, this amounts to saying that $(m^{(\delta_n)}_0(t); t\geq\delta_n)$ is identical in law to the solution of
\[ t\geq \delta_n, \ dy_t = -\psi(y_t) dt \ + \ \Delta J^0_t v_t, \ y_{\delta_n} = x_{\delta_n}, \ \ \mbox{with $\cL(v_t) = \cL(y_t)$},  \]
and by uniqueness of the solution to the latter equation, $(m^{(\delta_n)}_0(t); t\geq\delta_n)$ is identical in law to $(x_t; t\geq\delta_n)$.
Further, the independence of the marks at level $\delta_n$ easily implies (by the branching structure in the CPP) that for every $s\geq \delta_n$, the set of marks $\{m^{(\delta_n)}_{l}(s) \ : \ (l,s)\in\cT\}$ (labelled by increasing values of $l$)
is a sequence of i.i.d. marks with law $\cL(x_s)$.

Let $\cT^{\delta_n}$ be the  subset of the tree $\cT$ consisting of all the points above time level $\delta_n$. By the previous argument, the sequence
$\{(\cT^{\delta_n}, m^{(\delta_n)})\}_n$ is consistent in the sense that if we consider the marked tree $(\cT^{\delta_{n+1}}, m^{(\delta_{n+1})})$
below level $\delta_n$, the resulting object is distributed as $(\cT^{\delta_{n}}, m^{(\delta_{n})})$.

By the Kolmogorov extension theorem, we can then construct a unique full marking $(\cT,m)$ such that the restriction above level $\delta_n$
coincides with $(\cT^{\delta_n}, m^{(\delta_n)})$, and thus $m$ is such that  
\[\cL(\{m_0(t);t>0\}) \ = \ \cL(\{x_t;t>0\}). \label{eq:rel-mkv}\]
and at every level $s$, $\{m_{l}(s) \ : \ (l,s)\in\cT\}$ is a sequence 
with i.i.d. marks with law $\cL(x_s)$.

For uniqueness, it is enough to note that the second property (i.e., the distribution of the marks at level $s$) implies that any two full markings satisfying this property must be identical in law above level $s$ for every $s>0$. 
\end{proof}

%Proposition \ref{cor:how-to-prove-u} is then a direct corollary of the growth condition in Proposition \ref{lem:marking}(i).

\begin{lem}\label{lem:grwoth-condition}
Let $(x_t; t>0)$ be a proper solution and let $m$ be the associated full marking. Then $(m_0(t); t>0)$ (and thus $(x_t,; t>0)$) satisfies the growth condition
(\ref{growth-marking-0}).
\end{lem}

\begin{proof}
We will first need the following preliminary step that will be also used later in this manuscript.

{\bf Step 1.} Let us first consider a general rooted ultra-metric tree ${\bf t}$
whose leaves are denoted by $l_1,\cdots,l_n$
and such that the distance between the root and the leaves is given by $\tau$.
Consider a marking of the leaves $M_{l_1},\cdots,M_{l_n}\in\R^+$,
and let us consider the marking $M_0$ of the root obtained by propagating the marks 
according to the dynamics $\dot x  = -\psi(x)$ along the branches, and by adding the marks upon coalescence
(as in the Brownian CPP). Recall that $t\mapsto \phi(t+q(x_0))$ is the unique solution to $\dot u = -\psi(u)$ with initial condition $x_0$ at time 0. 
We claim that 
\be   \phi(\tau + q(\sum_{i=1}^n M_i)) \  < \ \tau   \leq M_0  \leq \sum_{i=1}^{n} \phi(\tau + q(M_i)). \label{eq:degnerate-trees} \ee
On the one hand, the RHS of the inequality corresponds to the extreme case of the star tree 
i.e, when all the branches coalesce simultaneously at the root of the tree (in this case the marks evolve independently along each branch,
and then add up at the root). On the other hand,
the LHS corresponds to the degenerate situation where all the leaves coalesce instantaneously (in which case, the marks add up to 
$\sum_{i=1}^n M_i$ and then evolve along a single branch).

Since $\psi\in \Hc$ and is convex, it is also
super-linear in the sense 
that for $x,y>0$
\[\psi(x+y)\geq \psi(x)+\psi(y).\]
Recall that marks evolve according to the dynamics $\dot x = -\psi(x)$ along each branch, so that the latter super-linearity assumption implies that
the marking decreases faster if we collapse two branches into a single branch, i.e.,
if we consider
\[\dot  r_t \ = \ -\psi(r_t), \ r_0=a_0+a_1, \ \ \mbox{and } \forall i=1,2, \ \ \dot  r^i_t \ = \ -\psi(r_t^i), \ r_0=a_i.\]
then $r_t \leq r_t^1 + r_t^2$. 
(\ref{eq:degnerate-trees})
can easily be deduced from there and a simple induction on the number of nodes of the ultrametric tree.

\medskip

{\bf Step 2.} 
Let $D^{(\delta_k)}_{\cT}(l,t)$ be the set of descendants of $(l,t)$ in the tree $\cT$ 
with time coordinate $\delta_k$. See Fig \ref{cpp}.  $|D^{(\delta_k)}_{\cT}(l,t)|$ denotes the cardinality of the set $D^{(\delta_k)}_{\cT}(l,t)$.

Let $m$ be the full marking constructed from the solution $x$. By construction, at a given level $\delta_k<t$, the set of marks 
$\{m_{l}(\delta_k)\}_{(l,\delta_k)\in D^{(\delta_k)}_{\cT}(0,t)}$
is identical in law 
with $\{y^k_n\}_{n \leq |D^{(\delta_k)}_{\cT}(0,t)|}$
where $\{y_n^k\}_{n}$ is an sequence of independent random variables with law $\cL(x_{\delta_k})$,
and independent of $|D^{(\delta_k)}_{\cT}({0,t})|$. Since $|D^{(\delta_k)}_{\cT}(0,t)|\to\infty$ a.s. as $t\to 0$, we have
\[ M^k(0,t) \ := \  \sum_{(l,\delta_k)\in D^{(\delta_k)}_{\cT}(0,t)} m_l(\delta_k) \ \to \infty \ \ \mbox{in probability.} \]
Here, we used the fact that $\cL(m_l(\delta_k))=\cL(x_{\delta_k})$
and that the solution is proper. Indeed, since the solution is proper, there exist $\varepsilon >0$ and a sequence $\delta_k'>0$ converging to 0 such that $x_{\delta_k'}\ge \varepsilon$ with probability larger than $\varepsilon$. Substituting $\delta_k$ with $\delta_k'$ yields the result.

Finally, by Step 1, we have
\be m_l(t) \geq \phi(t-\delta_k + q(M^k(0,t))), \label{eq:mlt}\ee
and the result follows by letting $k\to\infty$.
\end{proof}

\begin{rmk}
Note that we only used the LHS of (\ref{eq:degnerate-trees}) in the proof of the lemma.
The RHS will be useful later.
\end{rmk}

%As a direct corollary of the previous lemma, we get the following uniqueness criterium.
%\begin{prop}[Uniqueness criterium for proper solutions]\label{cor:how-to-prove-u}
%Assume that $\psi$ is such that for any two full (and coupled) markings $m^1$
%and $m^2$ 
%of the (same) CPP  satisfying the growth condition 
%\be \forall t>0, \   m_0(t) \ \geq \ \inf\{s>0 \ : \ \int_s^\infty 1/\psi(u) du <t \}, \ee
%then $(m^1_0(t); t>0) = (m^2_0(t); t>0)$. (In other words, there exists a unique full marking of the eternal branch satisfying
%the previous growth condition.)
%Then there exists a unique proper solution to the corresponding MK-V equation.
%\end{prop}
%\begin{proof}
%This is shown by contradiction. Let us assume that there exists two distinct solutions $x^1$
%and $x^2$. We can extend the same construction spelled out at the beginning of this section from one to two markings, i.e.,  we could couple
%two markings $(m^1,m^2)$ on the same CPP so that for every $i=1,2$, $\cL(x^i) = \cL(m^i_0)$. Those two solutions would generate two distinct markings
%of the eternal branch $m^1_0,m^2_0$ with $m^1_0, m^2_0$ satisfying the growth condition  (\ref{growth-marking-0}).
%\end{proof}
%
%
%

\subsection{Uniqueness of a proper solution in the stable case}
\label{sect:uniqueness-sol}
In this section, we restrict our study  to the stable case $\psi(x) = c x^\gamma$ with 
$\gamma>1$. Then $q(x) = \beta x^{1-\gamma}/c$ and $\phi(t) = (ct/\beta)^{-\beta}$. As a consequence, the ODE
$\dot z_t = \ -\psi(z_t)$ with initial condition $z_0$ at time 0 has the solution
\be \forall t\geq0, \ \  z_t  \ = \ \left( c(\gamma-1)t + z_0^{1-\gamma}\right)^{-\beta}
% (\frac{1}{c(\gamma-1)t + \frac{1}{z_0^{\gamma-1}}})^{\beta}\ 
\label{eq:descent-from-infty}.
\ee

%In the following $m^-$ will denote a full marking of the CPP
%associated to $x$.
%
%
%satisfying the growth condition  
%\be m^{-}_0(t) \geq \inf\{s>0 \ : \  \int_s^\infty \frac{1}{\psi(u)} du <t \} \  =  \left(\frac{1}{c(\gamma-1)t}\right)^{\beta} \label{growth-marking} \ee
%Note that Proposition \ref{lem:marking}(ii) ensures that $m^+_l(t)\geq m^-_l(t)$ for every $(l,t)\in\cT$.
%According to Proposition \ref{cor:how-to-prove-u}, the problem of uniqueness for a proper solution 
%boils down to the following result.

\begin{prop}\label{lem:m_-=m_+}
Let $(x_t; t>0)$ be a proper solution to MK-V. Let $m^-$
be the full marking of the CPP $\cT$ associated to $x$ (as defined in Proposition \ref{prop:marking-associated}).
Let $m^+$ be maximal full marking of $\cT$ (as defined in  Proposition \ref{lem:marking}).
Then $(m_0^+(t); t>0) = (m_0^-(t); t>0)$ a.s.. As a consequence,
$x$ is identical in law to $m_0^+$.
\end{prop}

Let $T>0$. By right continuity of $m_0^\pm$
it is enough to show that  
\be
m^+_0(T)=m_0^-(T) \ \  \mbox{a.s..} \label{eq:cqfd}
\ee
 The rest of the proof is dedicated to this result. 
We note that in the following, we will use repeatedly the fact that $m^\pm$ satisfy the growth condition
\be\label{growth-condition}
\forall t>0,  \ \ m_{0}^{\pm}(t) \ \geq \   (c(\gamma-1)t)^{-\beta},
\ee
(as a consequence of Lemma \ref{lem:grwoth-condition} and Proposition \ref{lem:marking}).

Let us consider the set of descendants of $(0,T)$ belonging to the PPP $\cP \subset\cT$ (i.e. the branching points in the descendence of $(0,T)$) and let us denote this set by $\cD_{0,T}$.
$\D_{0,T}$ is a set of points endowed with a natural (a.s.) binary tree structure -- See Fig. \ref{cpp}. We can index every point in $\cD_{0,T}$ by a point in ${\bf t} = \cup_{n\in\N^*} \{0\} \otimes \{0,1\}^{n}$, i.e., we construct a bijection $G$
from ${\bf t}$ to $\cD_{0,T}$ in such a way that 
\begin{itemize}
\item $G(0) \ = \ (0,T)$
\item If $\kappa \in{\bf t}$, $G(\kappa,0)$ (resp.,  $G(\kappa,1)$) is 
the left-child (resp., right-child) of $G(\kappa)$ in $\D_{0,T}$. 
\end{itemize}
The binary (planar) tree ${\bf t}$ is naturally equipped with a triplet $(g_{\kappa}^+, g^-_{\kappa}, d_{\kappa})_{\kappa\in{\bf t}}$
where $g_{\kappa}^\pm$ is the mark  $m^\pm_{\bar l}(\bar t-)$ where $(\bar l, \bar t)=G(\kappa)$  and $d_{\kappa}$ (the depth of the point $\kappa$)
is the time coordinate of the point $G(\kappa)$ in $\cP$.  Note that for $\kappa\neq 0 $, the point $G(\kappa)$ is a branching point of the CPP, and corresponds to a discontinuity point in the marking. In our notation,
the marks $g_{\kappa}^\pm$ are considered right before the occurrence of the discontinuity (i.e., before we add up the marks at the branching
point). 

We now fix $\kappa \in{\bf t}$. Our first goal is to show the ``passage formula'' (\ref{eq:passage-formula}) below, which will be achieved through Lemmas \ref{lem:subadd} and
\ref{lem:subadd2}. This formula will allow us to go from an arbitrary mark $\kappa$
to the marks of its children. The desired formula (\ref{eq:cqfd})
will be achieved from there by an induction on the nodes of the tree.

\medskip

First, from the definition of the marking $m^\pm$, we can deduce the marking
at $\kappa$ from the marking of its two children $\left((\kappa,0),(\kappa,1)\right)$, namely, if we consider the dynamics
\be \dot z^\pm_u \ = \ -\psi(z^\pm_u) ,  \ \ \ z^\pm_0 = \sum_{i=0}^1 g _{(\kappa,i)}^\pm  \label{eq:dyn-z} \ee
then $g^\pm_\kappa \ = \ z^\pm_{d_\kappa - d_{\kappa,0}} \ = \ z^\pm_{d_\kappa - d_{\kappa,1}}$ (since $d_{\kappa,0}=d_{\kappa,1}$).
In other words, we sum up the marks carried by the two children of $G(\kappa)$
and let the marking evolve according to the dynamics $\dot x =- \psi(x)$ along the branch connecting $\kappa$
to its children $(\kappa,0)$ and $(\kappa,1)$.

Alternatively, we will consider the dynamics
\be \label{eq:dynamics-node-to-node}  \mbox{for $i=0,1$ } \  \ \ \dot z^{\pm,i}_u \ = \ -\psi(z^{\pm,i}_u), \ \ z^{\pm,i}_0 \ = \ g^\pm_{(\kappa,i)}\ee
In words, instead of merging the two branches at $G(\kappa,0)=G(\kappa,1)$, we treat the two branches 
as if the merging had not occurred. 
\begin{lem}\label{lem:subadd}
Let $z^\pm$ and $z^{\pm,i}$ be defined as above, then for every $u\geq0$,
\be \Delta z_u \ \leq \ \sum_{i=0}^1 \ \Delta z_u^i, \ \ \mbox{where  } \ \Delta z_u \ := \ z^{+}_u - z^{-}_u, \ \ \mbox{and for $i=0,1$, } \ \Delta z^i_u \ := \ z^{+,i}_u - z^{-,i}_u. \ee
\end{lem}
\begin{proof}
Define
$$ f_t(x,y) \ := \ \left(c(\gamma-1)t + (x+y)^{1-\gamma}\right)^{-\beta} 
\ - \ \left(c(\gamma-1)t + x^{1-\gamma}\right)^{-\beta} 
\ - \ \left(c(\gamma-1)t + y^{1-\gamma}\right)^{-\beta} 
%\left(\frac{1}{c(\gamma-1)t + (\frac{1}{x+y})^{\gamma-1}}\right)^{\frac{1}{\gamma-1}} 
%\ - \ \left(\frac{1}{c(\gamma-1)t + (\frac{1}{x})^{\gamma-1}}\right)^{\frac{1}{\gamma-1}} 
%\ - \ \left(\frac{1}{c(\gamma-1)t + (\frac{1}{y})^{\gamma-1}}\right)^{\frac{1}{\gamma-1}}
$$
so that according to (\ref{eq:descent-from-infty})
$$ z_t^\pm - \sum_{i=0}^1 z_{t}^{\pm,i} \ =\  f_t(z_{0}^{\pm,0}, z_{0}^{\pm,1}).$$
We need to prove that $f_t(z_{0}^{+,0}, z_{0}^{+,1})\leq f_t(z_{0}^{-,0}, z_{0}^{-,1})$.
Since $z^{+,i}_0\geq z^{-,i}_0$ for $i=0,1$, 
the problem boils down to proving that the coordinates of the gradient of $f_t$
are non-positive for $x,y>0$. We have
$$\partial_x f_t(x,y) \ = \ g_t(1/(x+y)) - g_t(1/x), \ \mbox{where }  
 \ g_t(u) = \left(c(\gamma-1)t + u^{\gamma-1}\right)^{-\frac{\gamma}{\gamma-1}} u^{\gamma},
 $$
 and one can easily check that $g$ is increasing in $u$ (for $\gamma>1$), thus showing that $\partial_x f_t(x,y)\leq 0$. An analogous argument shows that 
 the gradient is non-positive along the $y$ coordinate. This completes the proof of the lemma.
%
%Define $\upsilon_u \ := \ \Delta z_u - \sum_{i=0}^1\Delta z_u^i$ and let $\tau:=\inf\{t: \upsilon_t >0\}$. Since $\upsilon_0=0$,  if $\tau<\infty$
%then we should have $\upsilon'_\tau\geq0$. However, using the fact that $\psi(z^\pm_\tau) \ = \ \psi(\sum_{i=0}^1 z_\tau^{\pm,i})$, we get
%$$ \upsilon'_\tau  \ = \ -\left( h(z_\tau^{+,0}, z_\tau^{+,1}) - h(z_\tau^{-,0}, z_\tau^{-,1})\right)   $$
%where $h(x,y)=\psi(x+y)-\psi(x)-\psi(y)$. One can directly check that for $x,y>0$ the coordinates of the gradient of $h$
%are non-negative, and since $z_\tau^{+,i}\geq z_\tau^{-,i}$, we get that $\upsilon'_\tau\leq 0$. This yields a contradiction
%and completes the proof of the lemma.
\end{proof}

\begin{lem}\label{lem:subadd2} For any $\kappa\in{\bf t}$, and $i=0,1$
\be\label{eq:b1-b2}\Delta z^{i}_{d_\kappa - d_{(\kappa,i)}} \leq \Delta g_{(\kappa,i)} \left( \frac{d_{(\kappa,i)}}{d_{\kappa}} \right)^{\frac{\gamma}{\gamma-1}
\left(1+F(\frac{d_{\kappa,i,*}}{d_{\kappa}} )\right)}\ee
where $*=0,1$ (the value of $d_{\kappa,i,*}$does not change the value of  $*$) and 
$$\forall x\in[0,1],  \ \ F(x) \ =  x \frac{(1 - 2^{1-\gamma})  x}{1-(1 - 2^{1-\gamma})  x} \ \geq \ 0 $$
\end{lem}
\begin{proof}

{\bf Step 1.} 
Recall from (\ref{eq:dynamics-node-to-node}) that $\Delta z_0^i = \Delta g_{\kappa,i}$ where $\Delta g_{\kappa,i}= g_{\kappa,i}^{+}-  g_{\kappa,i}^{-}$ and 
\be d \Delta z_u^i \ = \ -(\psi(z_u^{+,i}) - \psi(z_u^{-,i}) )du. \label{eq:solve-gamma}\ee
 The strategy will consists in bounding from above the RHS of the differential relation and then use 
Gronwald's lemma. First, by convexity of $\psi$ and since $z^{+,i}_u \geq z^{-,i}_u$
\begin{eqnarray*}
\psi(z_t^{+,i}) - \psi(z_t^{-,i})  
& \geq & \psi'(z_u^{-,i}) \Delta z^{i}_t \ = \ c\gamma (z_u^{-,i})^{\gamma-1} \Delta z^{i}_t. 
\end{eqnarray*}
We will now bound $z_u^{-,i}$ on the interval $[0, d_{\kappa} - d_{(\kappa,i)}]$. 
Let $v$ be such that 
$$ z_0^{-,i} \ = \ g^{-}_{\kappa,i} \ = \  (\frac{v}{c(\gamma-1) d_{\kappa,i}})^{\beta} $$
Note that since the depth of $g_{\kappa,i}$ is given by $d_{\kappa,i}$ and since
we have 
$m^\pm_l(u)\geq (\frac{1}{c(\gamma-1)u)})^{\beta}$ (see (\ref{growth-condition})), we have $v\geq1$. Solving
the equation (\ref{eq:dynamics-node-to-node}), we find that for any $u\in[0,d_{\kappa}-d_{\kappa,i}]$
\begin{eqnarray*} z_u^{-,i} & = & \left( c(\gamma-1)u + c(\gamma-1) \frac{d_{\kappa,i}}{v}\right)^{-\beta} \\
& = & \left(c(\gamma-1)(u+d_{\kappa,i})\right)^{-\beta}\left( 1+ \frac{d_{\kappa,i}(1-1/v)}{(u+d_{\kappa,i})-d_{\kappa,i}(1-\frac{1}{v})} \right)^{\beta} \\
& \geq & \left(c(\gamma-1)(u+d_{\kappa,i})\right)^{-\beta} \left(1+ \frac{1-\frac{1}{v}}{\frac{d_{\kappa}}{d_{\kappa,i}} - (1-\frac{1}{v})}\right)^{\beta}
\end{eqnarray*}
where the LHS and RHS of the inequality are equal  when $u=d_{\kappa}-d_{\kappa,i}$. Putting everything together, we get
\be \forall u\in[0, d_{\kappa}- d_{\kappa,i}], \ \  \psi(z_u^{+,i}) - \psi(z_u^{-,i}) \geq  \frac{\Delta z_t^i}{u+ d_{\kappa,i}}  \ \frac{\gamma}{\gamma-1}
\left(1+ \frac{1-\frac{1}{v}}{\frac{d_{\kappa}}{d_{\kappa,i}} - (1-\frac{1}{v})}\right)
\ee 

\medskip

{\bf Step 2.} Since the RHS of the latter inequality increases in $v$, we produce a lower bound for $v$.
By construction, the mark $g^{-}_{\kappa,i}$ is obtained by considering the dynamics
\be dy_t \ = \ -\psi(y_t) dt, \ \ y_0 \ = \ \sum_{j=0}^1 g_{\kappa,i,j}^{-} \label{eq:ytu}\ee
evaluated at time $d_{\kappa,i} - d_{\kappa,i,*}$ and where $*=0$ or $1$.
Since (\ref{growth-condition}) implies that $g_{\kappa,i,j}^{-} \geq (c(\gamma-1)d_{\kappa,i,*})^{-\beta}$, we get 
\begin{eqnarray*} 
g^{-}_{\kappa,i} & \geq  & 
\left(c(\gamma-1)(d_{\kappa,i} - d_{\kappa,i,*} ) + \frac{1}{2^{\gamma-1}}c(\gamma-1)d_{\kappa,i,*} \right)^{-\beta}  \\
& = & \left(c(\gamma-1) d_{\kappa,i}\right)^{-\beta}\left( 1-(1-\frac{1}{2^{\gamma-1}}) \frac{d_{\kappa.i,*}}{d_{\kappa,i}}\right)^{-\beta}
\end{eqnarray*}
where the RHS and LHS of the inequality are equal when the initial condition of $(y_t;t\geq0)$ is taken to be $2 (c(\gamma-1)d_{\kappa,i,*})^{-\beta}$
(where the factor $2$ comes from the sum in the initial condition of (\ref{eq:ytu})). This implies that 
\be v\geq  \frac{1}{1-(1-\frac{1}{2^{\gamma-1}}) \frac{d_{\kappa.i,*}}{d_{\kappa,i}}} \ee

\medskip

{\bf Step 3.} Combining the two previous steps yield that 
\be \nonumber \forall u\in[0, d_{\kappa}- d_{\kappa,i}], \ \  \psi(z_u^{+,i}) - \psi(z_u^{-,i}) \geq  \frac{\Delta z_t^i}{u+ d_{\kappa,i}}  \ \frac{\gamma}{\gamma-1}
\left(1+ G( \frac{d_{\kappa,i}}{d_{\kappa}}, \frac{d_{\kappa,i,*}}{d_{\kappa,i}})\right).
\ee 
where $G(x,y) \ = \ x \frac{(1 - 2^{1-\gamma})  y}{1-(1 - 2^{1-\gamma})  y}$. Further since $G$ is increasing in $x$ and $y$ (since $\gamma>1$), 
and $\frac{d_{\kappa,i}}{d_{\kappa}}, \frac{d_{\kappa,i,*}}{d_{\kappa,i}} \leq \frac{d_{\kappa,i,*}}{d_{\kappa}}$
we have
\be \nonumber \forall u\in[0, d_{\kappa}- d_{\kappa,i}], \ \  \psi(z_u^{+,i}) - \psi(z_u^{-,i}) \geq  \frac{\Delta z_t^i}{u+ d_{\kappa,i}}  \ \frac{\gamma}{\gamma-1}
\left(1+ F( \frac{d_{\kappa,i,*}}{d_{\kappa}})\right).
\ee 
Next, if we solve 
\be \nonumber d\bar y_u \ = \   \frac{\bar y_u}{u+ d_{\kappa,i}}  \ \frac{\gamma}{\gamma-1}
\left(1+ F( \frac{d_{\kappa,i,*}}{d_{\kappa}})\right) du,  \  \bar y_0 = g_{\kappa,i} 
\ee
one finds the RHS of (\ref{eq:b1-b2}). The proof of the lemma is achieved by a direct application of Gronwald's lemma.
\end{proof}
Recall that $\Delta g_\kappa = \Delta z_{d_{\kappa}- d_{\kappa,*}}$ and $\Delta z = z^+ - z^-$ with $z^+$ and $z^-$ follow the dynamics defined in (\ref{eq:dyn-z}).
From Lemma \ref{lem:subadd}, we have $\Delta g_\kappa  \leq \sum_{i=0}^1  \Delta z^{i}_{d_\kappa - d_{(\kappa,i)}}$. By the previous lemma, this implies that 
$$ \Delta g_{\kappa} \leq \sum_{i=0}^1 \Delta g_{(\kappa,i)} \left( \frac{d_{(\kappa,i)}}{d_{\kappa}} \right)^{\frac{\gamma}{\gamma-1}
\left(1+F(\frac{d_{\kappa,i,*}}{d_{\kappa}})\right)} $$
or equivalently that 
\be d_\kappa^{\beta}\Delta g_{\kappa} \leq \sum_{i=0}^1  d_{(\kappa,i)}^{\beta} \Delta g_{(\kappa,i)}   \left( \frac{d_{(\kappa,i)}}{d_{\kappa}} \right)^{
\left(1+ \frac{\gamma}{\gamma-1} F(\frac{d_{\kappa,i,*}}{d_{\kappa}})\right)} \label{eq:passage-formula} \ee
where $F$ is defined in the previous lemma.

Let us now proceed with the rest of the proof. For $n\geq1$,
let ${\bf t}_n = \{0\}\otimes\{0,1\}^n$ that can be thought of 
as the vertices in ${\bf t}$ which are at a distance $n$ from the root. For any $v\in{\bf t}_n$,
let $v(i)\in{\bf t}_i$ the vertex obtained by only considering the first $i$ coordinates of $v$
(i.e., $v(i)$ is the ancestor of $v$ at a distance $i$ from the root).
After a simple induction, we can generalize the previous inequality to the following one
$$\forall n\geq1, \ \ T^{\beta} \Delta g_{0} \leq \sum_{v\in{\bf t}_{2n}} d_v^{\beta}\Delta g_{v} \ \Pi_{i=0}^{2n-1} \left( \frac{d_{v(i+1)}}{d_{v(i)}} \right)^{1+ \frac{\gamma}{\gamma-1} F(\frac{d_{v(i+2)}}{d_{v(i)}} )} $$
(using the fact that $d_{v(0)}=d_0=T$) which implies that (using the fact that $F(x)\geq0$ on $[0,1]^2$)
\beqn
\forall n\geq1, \ \ T^{\beta}  \Delta g_{0}  & \leq & \sum_{v\in{\bf t}_{2n}} 
\ d_v^{\beta} \Delta g_{v}  
\Pi_{i=0}^{n-1}  \left( \frac{d_{v(2i+1)}}{d_{v(2i)}} \right)^{
\left(1+\frac{\gamma}{\gamma-1} F(\frac{d_{v(2i+2)}}{d_{v(2i)}} )\right)}  \frac{d_{v(2i+2)}}{d_{v(2i+1)}} \nonumber \\
& = & \sum_{v\in{\bf t}_{2n}} 
d_v^{\beta} \Delta g_{v} \
\Pi_{i=0}^{n-1}  \ \eps( \frac{d_{v(2i+1)}}{d_{v(2i)}},  \frac{d_{v(2i+2)}}{d_{v(2i)}}  ) \nonumber \\
& \leq  & 2 \sum_{v\in{\bf t}_{2n}} 
d_v^{\beta} g^+_{v} \
\Pi_{i=0}^{n-1}  \ \eps( \frac{d_{v(2i+1)}}{d_{v(2i)}},  \frac{d_{v(2i+2)}}{d_{v(2i)}}  ) \label{ineq:full-tree}
\eeqn
where $\eps(x,y) = y x^{\gamma/(\gamma-1)F(y)}$. We will now take advantage of the self-similarity 
in the Brownian CPP. 

\begin{lem} \label{lem:scaling-in-CPP}Let $v\in{\bf t}_n$. Then 
\begin{enumerate}
\item $\{d_{v(i+1)}/d_{v(i)}\}_{i=0}^{n-1}$ is a sequence of i.i.d. random variables uniformly
distributed on $[0,1]$. 
\item $d_v^{\beta} g^+_{v} = T^{\beta} g^+_{0} = T^{\beta} m^+_0(T)$ in law.
\item $\{d_{v(i+1)}/d_{v(i)}\}_{i=0}^{n-1}$ and $d_{v}^{\beta}  g^+_{v}$ are independent.
\end{enumerate}
\end{lem}
\begin{proof}
W.l.o.g. we take  $v=0_n$ and $T=1$, where $0_n$ is the vector of length $n$ filled up with $0$. 
Let $\tau > 0$ be random or deterministic.
Define the scaling operator $F_{1/\tau}(l,t) \ =\ (l/\tau, t/\tau)$. Finally,
$$ l_{\tau} \ := \ \inf\{l>0 \ : \ (l,t)\in\cP, \ \  t\geq \tau\}, \ \ \cP_{\tau}:= \{(l,t)\in\cP \ : \  l\leq l_{\tau}\},$$
$\{l_\tau\}\times\R_*^+$which is the left-most branch in the tree $\cT$ alive at time $\tau$.
In particular, we note that $m_0(\tau)$ is measurable with respect to $\cP_\tau$ since the vertical 
branches of the CPP $l$-coordinates such 
that $l\geq l_\tau$ will not coalesce with the branch $\{0\}\times\R^+$ before time $\tau$.

From the scale invariance properties 
of the CPP, it is not hard to show that  (i)
$\{d_{0_n(i+1)}/d_{0_n(i)}\}_{i=0}^{n-1}$ is a sequence of uniform random variables on $[0,1]$, and that if we take $\tau=d_{0_n}$
then (ii)
$F_{1/d_\tau}(\cP_{d_{\tau}})$ is identical in law with $\cP_1$, and (iii) that  $\{d_{0_n(i+1)}/d_{0_n(i)}\}_{i=0}^{n-1}$
and $F_{1/d_\tau}(\cP_{d_{\tau}})$ are independent. 
Further, by reasoning along the exact same lines of Proposition \ref{lem:marking}(v),  one can show that 
$ d_{\tau}^{\beta} m^+_{\cdot}(\cdot d_{\tau})$ coincides with the marking $m^+$ on the rescaled CPP $F_{1/d_\tau}(\cP_{d_\tau})$.
This completes the proof of the lemma. 
\end{proof}
Passing to the expectation on both side of (\ref{ineq:full-tree}) and using the previous lemma, we get that  
\beqn
\forall n\geq1, \ \ \E\left(T^{\beta}  \Delta g_{0} \right) & \leq & 
2 \sum_{v\in{\bf t}_{2n}} 
\E\left(d_v^{\beta} g^+_{v} \right) 
\Pi_{i=0}^{n-1}  \E\left( \ \eps( \frac{d_{v(2i+1)}}{d_{v(2i)}},  \frac{d_{v(2i+2)}}{d_{v(2i)}}  ) \right) \nonumber \\
& = & 2 \times 4^n \times T^{\gamma-1}  \E\left(g^+_0\right)  \ \E\left(\eps(U_1, U_1 U_2)\right))^{n} \label{ineq:almost-there}
\eeqn
where $U_1$ and $U_2$ are two uniform random variables on $[0,1]$. From the definition
of $\eps$, we have
$$\eps(U_1,U_1 U_2) < U_1 U_2 \ \mbox{a.s.}$$ 
Since $\E(U_1 U_2)=1/4$, we have
$\E(\eps(U_1, U_1 U_2))<1/4$ and thus that the $\limsup$
of the RHS of (\ref{ineq:almost-there}) goes to $0$ as $n\to\infty$.
Finally, the proof of (\ref{eq:cqfd}) (and thus of Proposition \ref{lem:m_-=m_+}) is completed by the following lemma.

\begin{lem}\label{lem:finite-first-moment}
$\E(g_0^+) \ = \ \E\left(m^+_0(T) \right)<\infty.$
\end{lem}
\begin{proof} 

Let $n\geq1$. For every $v\in{\bf t}_n$
we consider the dynamics
\be  \dot r_t^v \ = \ - \psi(r_t^v), \ \ \ r_0^v = g_v. \label{dynamics:r}\ee
From the RHS of (\ref{eq:degnerate-trees}), we get that 
$$\E\left(m_0^+(T)\right) \ \leq 2^n \E\left( r^{0_n}_{d_0-d_{0_n}} \right) \ \leq  2^n \E(\bar r_{T-d_{0_n}}) $$
where $\bar r$ follows the dynamics $d \bar r  = -\psi(\bar r) dt$ with initial condition $+\infty$. Solving for $\bar r$
we get
$$\bar r_{T-d_{0_n}} = \left( c(\gamma-1) (d_0-d_{0_n}) \right)^{-\beta} \ = \   \left( c(\gamma-1) T (1-\Pi_{i=0}^{n-1} \frac{d_{0_n(i+1)}}{d_{0_n(i)}}) \right)^{-\beta} $$
and from Lemma \ref{lem:scaling-in-CPP}, it remains to show that the following integral is finite for $n$ large enough
$$I_n' \ = \ \int_{[0,1]^n} \left(1- \Pi_{i=1}^n u_i \right)^{-\beta} d u_1\cdots d u_n. $$
Since the singularity of the integral is at $(1,\cdots,1)$ we can consider the integral
$$I_n  = \ \int_{[\frac{1}{2},1]^n} \left(1- \Pi_{i=1}^n u_i \right)^{-\beta} d u_1\cdots d u_n. $$
Let us now make the change of variable $\forall i\in[n],  \ w_i \ = \ \Pi_{j=i}^n u_j$ so that
\beqnn
I_n & = & \int_{w_1 \leq \cdots \leq w_n, \forall i\in[n], w_i \in[\frac{1}{2^i},1]} \left(1- w_1 \right)^{-\beta} \frac{1}{\Pi_{i=2}^n w_i}d w_1\cdots d w_n \\
& \leq & \frac{1}{2^{n-1}}  \int_{{w_1 \leq \cdots \leq w_n, \forall i\in[n], w_i \in[\frac{1}{2^i},1]}} \left(1- w_1 \right)^{-\beta} d w_1\cdots d w_n \\
& = & \frac{1}{2^{n-1}}  \frac{1}{\beta-1} 
\int_{w_2 \leq \cdots \leq w_n, \forall i\in\{2,\cdots,n\}, w_i \in[\frac{1}{2^i},1]} 
\left(1- w_2 \right)^{1-\beta} d w_2\cdots d w_n + K_{n,\gamma}
 \eeqnn
 where $K_{n,\gamma}$ is a finite constant. Iterating the same calculation, we get
$$I_n \leq C_{n,\gamma} \int_{w_n\in[\frac{1}{2},1]}  \left(1- w_n \right)^{n-1-\beta} d w_n + \bar C_{n,\gamma}$$
where $C_{\gamma,n}, \bar C_{\gamma,n}<\infty$. Taking $n > \beta$ makes the integral $I_n$ finite. This completes the proof of the lemma.
\end{proof}

\begin{proof}[Proof of Theorem \ref{teo:existence:global:solution}]
The existence of a proper MK-V solution is provided by 
Proposition \ref{lem:marking}. We proved the uniqueness of the solution in the stable case
in Section \ref{sect:uniqueness-sol}. The scaling and measurability properties follow directly from 
Proposition \ref{lem:marking}. The integrability property follows from Lemma \ref{lem:finite-first-moment}.
\end{proof}

\subsection{Asymptotic behavior of finite population models}\label{sect:applications}
In this section, we use Theorem \ref{teo:existence:global:solution} to deduce some asymptotical results on the MK-V equation (\ref{eq:Sch}).

\begin{teo}\label{teo:solution1}
Again, we assume that $\psi(x)=c x^\gamma$ with $c>0$ and $\gamma>1$.
Let $\nu\in M_P(\R^+)$, with $\nu(\{0\})<1$. 
For every $\delta>0$, let
$(x^{(\delta)}_t; t\geq0)$ be the MK-V solution with inverse population
$\delta$ and initial measure $\nu$.
Finally, let $(x_t^{(0)}; t>0)$ be the unique proper solution to MK-V. 
\begin{itemize}
\item(Convergence to the $\infty$. pop. solution) For every $t>0$ 
\[  \lim_{\delta\downarrow 0} x^{(\delta)}_t \ = \ x_t^{(0)} \ \ \mbox{in law.}\]
\item(Long time behavior) For every $\delta>0$,
$$ \ \lim_{t\uparrow \infty} x_t^{(\delta)} t^{\beta} \ = \  \Upsilon =  \ x_1^{(0)} \ \ \mbox{in law.}$$ 
\end{itemize}
\end{teo}

\begin{rmk}
Let $\gamma\in(1,2]$ and let
$(\mu_t^{(0)}, t\geq0)$ be the unique proper weak solution
to the \Sm equation  (\ref{eq:Sch}). 
$(\cL(x^{(0)}_t); t>0)$ coincides with the measure valued process $\mu^{(0)}$.
(Using the fact that there is a unique proper \Sm solution, and that $(\cL(x^{(0)}_t); t>0)$ is a proper weak to the \Sm solution.)
As a direct corollary of Theorem \ref{teo:solution1}, we obtain the following PDE result: 
$S^{t,\gamma}(\mu_t^{(\delta)}) \Longrightarrow  \mu_1^{0}$,
where the convergence is meant in the weak topology.\end{rmk}

The proof of the previous theorem relies on the following lemmas, which is a corollary of   
the work carried out in the previous section.

\begin{lem}\label{cor:cv-to-self-similar}
Let $\nu\in M_P(\R^+)$ and consider a sequence 
$\{\nu^{(\delta)}\}_\delta$ in $M_P(\R_+)$ with $\nu^{(\delta)} \geq \nu$, where the domination is meant 
in the stochastic sense.

Let $\left(X_t^{(\delta)}; t\geq0\right)$ be a solution to (\ref{eq:mckean})
with $\cL(X_0^{(\delta)})=\nu^{(\delta)}$. Assume that $\nu(\{0\})<1$. Then
for every fixed $t > 0$, $\{X_t^{(\delta)}\}_{\delta}$ converges in law to the proper solution $x_t^{(0)}$ as $\delta\to 0$.
\end{lem}
\begin{proof}

Let $\{\delta_n\}$ be a sequence of positive numbers decreasing to $0$. 
According to Theorem \ref{teo:mckean}, $\cL(X^{(\delta_n)}) = \cL(\theta_{\delta_n}\circ m^{(\delta_n)}_0)$
where $m^{(\delta_n)}$ is the partial marking above level $\delta_n$
with initial condition $\nu^{(\delta_n)}$.
The strategy of the proof will consist in showing that the sequence of partial marking converges (up to a subsequence and in a sense specified below) 
to a full marking ${\bf m}$ (Step 1). Then, we show that ${\bf m}_0$ must be 
the (unique) proper solution of MK-V due to the condition $\nu^{(\delta)}\geq \delta$ (Step 2).
\medskip

%We start by investigating the convergence of the partial marking $m^{(\delta_n)}$
%to a global marking $\bar m$.
{\bf Step 1.} For $j<n$, define $a^n_{ij}$ to be the $i$th mark of $m^{(\delta_n)}$ at level $\delta_j$ (where marks are ranked from left to right).
By the branching structure of the CPP, the marks $\{a_{ij}^n\}_i$ are i.i.d..

We first claim that for every fixed $i,j\in\N$, the sequence $\{a_{ij}^n\}_n$ is tight.
In order to see that, we first consider the case where $\nu^{(\delta_n)}(dx) =\delta_\infty(dx)$ for every $n$. This exactly corresponds to
the marking $m^{(\delta_n),+}$ introduced in Proposition \ref{lem:marking}, for which we showed that 
$\{a_{ij}^n\}_n$ converges. Since in general $\{a_{ij}^n\}_n$ is dominated by the previous case,
it follows that $\{a_{ij}^n\}_n$ is tight. 

Now, $\{(a^n_{ij}; i,j\in\N) \}_n$ seen as a random infinite array 
(equipped with the product topology) is also tight. 
Going to a subsequence
if necessary, there exists a limiting array $(a_{ij}^\infty; i,j\in\N)$ such that
$$\{\left(\cT, (a_{ij}^n; i,j\in\N)\right)\}_n \Longrightarrow \left(\cT, (a_{ij}^\infty; i,j\in\N))\right)  \ \ \mbox{as $n\to\infty$},$$ 
where the convergence is meant in the product topology. Note that since for every fixed $j,n$, $\{a_{ij}^n\}_i$ is a sequence of i.i.d. random variables, the same holds for the sequence $\{a_{ij}^\infty\}_i$
for every $j$ (the marks at level $\delta_j$ are independent).

Using the Skorohod representation theorem, one can assume w.l.o.g. that the convergence holds a.s.. 
Let us now fix $j$ and let us consider ${\bf m}^j$ the marking 
above level $\delta_j$ with initial marks $(a_{ij}^\infty; i\in\N)$ (where the initial marks are assigned from left to right). 
By continuity of the flow, for almost every realization of the CPP and the markings,  
for every $t>\delta_j$ and $(l,t)\in\cT$,  $\{m^{(\delta_n)}_{l}(t)\}_{n>j}$ converges a.s. to  ${\bf m}^{j}_{l}(t)$
under our coupling.
(In other words, the convergence of the initial conditions induce the convergence of the partial marking with the limiting mark.) 

Let us now take $j>j'$ so that $\delta_j<\delta_{j'}$. The previous result at time $t=\delta_{j'}$, together with the fact $\{(a_{ij'}^n; i\in\N) \}_n$
converges to $\{(a_{ij'}^\infty; i\in\N) \}_n$
readily implies that the marks of ${\bf m}^{j}$ at level $\delta_{j'}$ coincides with $(a_{ij'}^\infty; i\in\N)$ -- the initial marking of ${\bf m}^j$ at level $\delta_{j'}$. Equivalently,
this guarantees that
the sequence of markings $\{{\bf m}^j\}_j$  is consistent in the sense that for $j>j'$, the marking ${\bf m}^j$
restricted to $\{t\geq \delta_j'\}$ coincides ${\bf m}^{j'}$. Thus, there exists a full marking ${\bf m}$ of the CPP such that 
$m$ coincides with ${\bf m}^j$ on $\{t\geq\delta_j\}$. 

Gathering the previous results, we showed that  (i) for every $s>0$, 
$\{{\bf m}_l(s) : \ (l,s)\in\cT\}$ is a sequence of i.i.d. rv's; and (ii) for every realization of the CPP and the markings, 
\be \forall t >0,   \ \ \{m^{(\delta_n)}_0(t)\}_n\to{\bf m}_0(t) \ \ \ \mbox{a.s.}. \label{eq;thett}\ee
As a consequence of (i), $({\bf m}_0(t); t > 0)$ is an $\infty$-pop. solution of MK-V  in virtue of Theorem \ref{teo:mckean2}.
\medskip

{\bf Step 2.} Next,  by reasoning along the same lines as Proposition \ref{prop:marking-associated} (see Step 2 therein),
the condition $\nu^{(\delta)} \geq \nu$ implies that this solution satisfies the growth condition (\ref{growth-marking-0}), and $({\bf m}_0(t); t > 0)$ must be proper. By the uniqueness result of  Theorem \ref{teo:existence:global:solution},
we get that $\cL\left({\bf m}_0(t)\right) \ = \ \cL( x_t^{(0)} )$. Finally, (\ref{eq;thett}) together with the fact that
 $\cL(X^{(\delta_n)}) = \cL(\theta_{\delta_n}\circ m^{(\delta_n)}_0)$ (see Theorem \ref{teo:mckean}) completes the proof the lemma.
\end{proof}

\begin{proof}[Proof of Theorem \ref{teo:solution1}]
The first item follows directly from Lemma \ref{cor:cv-to-self-similar} after taking $\nu^{(\delta)}=\nu$.
For the second item, as a direct consequence of
Proposition \ref{cor:scaling}, 
\be t^{\beta} x_{t}^{(\delta)}  \ = \ \bar x_1^{(\delta/t)}   \ \mbox{in law}\ee
where $x^{(\delta)}$ is the solution of (\ref{eq:mckean}) with initial measure $\nu$ and inverse population size $\delta$, and where $\bar x^{\delta/t}$
is the solution of (\ref{eq:mckean}) with initial measure $\cS^{1/t,\gamma}(\nu)$ and inverse population size $\delta/t$. Since for $t\geq1$,
$\cS^{1/t,\gamma}(\nu) \geq \nu$ (in the stochastic sense), the second item follows again by a direct application of Lemma \ref{cor:cv-to-self-similar}. 

\end{proof}

\subsection{Dust solutions}\label{sect:dust}

\begin{proof}[Proof of Theorem \ref{teo:dust}]
The proof is very similar to the one of Lemma \ref{cor:cv-to-self-similar}.
Let $X$ be a positive rv with finite mean and let $\{\delta_k\}$ be a 
positive sequence 
going to $0$. Let $m^{(\delta_k)}$ be the partial marking above level $\delta_k$
with initial law $\delta_k \cL(X)$. 

{\bf Step 1.} Up to a subsequence, one can construct a full marking ${\bf m}$ of the CPP and coupling between the $(\cT, m^{(\delta_k)})$'s
and $(\cT, {\bf m})$ such that 
for every $(l,t)\in \cT$
\[m^{(\delta_k)}_{l}(t) \to {\bf m}_{l}(t) \ \ \ \mbox{a.s.},\]
and such that ${\bf m}_0$ is an $\infty$-pop solution to  MK-V. The proof goes along the exact same lines
as Lemma \ref{cor:cv-to-self-similar} (Step 1). 

{\bf Step 2.} To prove that ${\bf m}_0$ is a dust solution, we show that ${\bf m}_0(t)\Longrightarrow 0$ as $t\to0$.
Recall the definition of $|{\cD}^{(\delta_k)}(0,t)|$, the number of descendants 
of $(0,t)$ at time $\delta_k$ in the CPP.
From (\ref{eq:degnerate-trees})
we get the following stochastic bounds
\[ \left(c(\gamma-1)t +\left(\sum_{i=1}^{|{D}^{(\delta_k)}(0,t)|}\delta_k X_i\right)^{1-\gamma} \right)^{-\beta}  
\leq \  m^{(\delta_k)}_0(t) \ \leq \  
\sum_{i=1}^{|{D}^{(\delta_k)}(0,t)|} \left(c(\gamma-1)t +\left(\delta_k X_i\right)^{1-\gamma} \right)^{-\beta}    \]
where the $\{X_i\}$ is an infinite sequence of rv's distributed as $X$ and independent of the CPP.
From the definition of the CPP, one can show that
\[\delta_k |{D}^{(\delta_k)}(0,t)| \ \Longrightarrow \ \cE(t)\] 
where $\cE(t)$ is an exponential rv with parameter $t$. From here, a direct application of the law of large number 
shows that the LHS (resp., RHS) of the latter inequality converges (in law) to 
\[\left(c(\gamma-1)t + (\E(X) \cE(t))^{1-\gamma}\right)^{-\beta} \ \ (\mbox{resp., } \E(X)\cE(t)).\]
As a consequence, 
\[    \left(c(\gamma-1)t + (\E(X) \cE(t))^{1-\gamma}\right)^{-\beta}  \ \leq \  {\bf m}_0(t) \ \leq\  \E(X) \cE(t). \]
The RHS shows that ${\bf m}_0(t)\Longrightarrow0$ in law as $t\to0$. Finally, the LHS of the inequality ensures that
$\P({\bf m}_0(t)>0)=1$ for $t>0$. This shows that ${\bf m}_0$ is a non-trivial dust solution of the $\infty$-pop. MK-V equation.

Following the same approach, we can construct $m^1,m^2,\cdots,m^k$, $k$ dust MK-V solutions using distinct positive rv's
$X^1,X^2,\cdots, X^k$ to initialize the underlying marking. Next, if we choose those random variables in such a way that
for every $j<k$,
\[ \E\left(m^j_0(t)\right) \ \leq \  \E(X^j) t \  < \ \E\left( \left(c(\gamma-1)t + (\E(X^{j+1}) \cE(t))^{1-\gamma}\right)^{-\beta}   \right)  \leq \E(m^{j+1}_0(t))\]
so that the law of the $m^k$'s must be distinct  at time $t$. This shows that one can construct infinitely many dust solutions to MK-V.
\end{proof}

\section{Main convergence results} 

\subsection{Notation}
\label{sect:notations-discrete}

We will typically consider  a sequence of nested coalescents indexed by $n$, in such way that
the initial number of species increase linearly with $n$ (see Theorem \ref{teo-local-1}).

Let $s_t^n$ be the number of blocks in the species coalescent (for the model indexed by $n$); Further $\tilde s_t^n \ = \ s^n_{t/n}$.
We order the blocks of the species coalescent at any given time by their least element. 
$B_t^n(i)$ will denote the $i^{th}$ block at time $t$, and $V_t^n(i)$ the number of element in the block; $\tilde B_t^n \ = \ B_{t/n}^n$; $\tilde V_t^n \ = \ V_{t/n}^n$.

Define $R^n$ the scaling operator acting on measure-valued process such 
for every $(\nu_t; t\geq0)$ valued in $M_F(\R^+)$, the process $(R^n(\nu)_t; t\geq0)$ is the only measure valued process such that for every 
bounded and continuous function $f$ 
$$ \forall t>0, \ \ \int_{\R^+}  f(x) R^n(\nu)_t(dx) \ = \ \int_{\R^+} f(x/n) \nu_{t/n}(dx).   $$
In words, space and time are both rescaled by $1/n$.

For $i\leq s_t^n$, let us denote by $\Pi_t^n(i)$ the number of gene lineages 
in the species block with index $i$; further $\tilde \Pi_t^n \ = \ \frac{1}{n}\Pi^n_{t/n}$. We define 
$$\mud_t^n \ = \ \sum_{i=1}^{s_t^n} \delta_{\Pi^n_t}, \ \mbox{and   } \tilde \mud_t^n = R^n \circ \mud_t^n \ = \ \frac{1}{\ts_t^n} \sum_{i=1}^{\ts_t^n} \delta_{\tilde \Pi^n_t(i)} $$

\medskip

 $\cC_0(\R^+)$ will denote the set of continuous functions vanishing $+\infty$;
$\cC_b^\infty(\R^+)$ will denote the set of infinitely many differentiable functions 
with bounded derivatives. The Stone-Weierstrass Theorem ensures that 
$\cC_b^\infty(\R^+)$ is dense in  $\cC_0(\R^+)$, and is also dense in the set of test function (i.e.,
the $\cC^1$ functions $f$ so that $f$ and $f'\psi$ remain bounded).

$(M_F(\R^+),v)$ will refer to the set of Radon measure on $\R^+$
endowed with the vague topology (i.e., the smallest topology making the map $\mud \to \left<\mud,f\right>$
continuous for every function $f\in \cC_0(\R^+)$);
whereas $(M_F(\R^+),w)$ will refer to $M_F(\R^+)$
equipped with the weak topology (i.e., the smallest topology making the map $\mud \to \left<\mud,f\right>$
continuous for every function $f\in \cC_b(\R^+)$ -- the set of continuous bounded functions).

\subsection{Statement of the main results}
%In this section, we show a strong relation between the \Sm equation (\ref{eq:Sch}) when $\gamma=2$
%and the small time behavior of the nested coalescent.
%In Theorem \ref{teo-local-1}, we consider a sequence of nested coalescents $\{(s_t^n,\Pi^n_t)\}_n$ indexed by $n$ in such way that
%the initial number of species increase linearly with $n$: $s_0^n/n \to r\in(0,\infty)$. In the rest of the paper, we will show (1) that the properly rescaled empirical measure 
%converges to the solution of the \Sm (\ref{eq:Sch}) equation with initial population $r/2$ (see Theorem \ref{teo-local-1} below), and (2)
%that  the $\infty$-pop. \Sm equation (\ref{eq:Sch})  
%is related to the small time behavior of the nested-Kingman model coming down from $+\infty$ (see Theorem \ref{prop:sequential} below).
%Let us now spell our results in more details.

\begin{teo}\label{teo-local-1} 
Consider a sequence of nested Kingman coalescents $\{\Pi^n, s^n\}_n$.
Let $\{X_i^n\}_{i,n}$ be an infinite array of independent rvs such that 
$$\forall i,j,n, \  \ \cL(X_i^n) = \cL(X_j^n), \ \ \ \limsup_n \E\left( (X_1^n)^2 \right) \  < \ \infty, $$
and independent of the species coalescent $\left(\ts_t^n; t\geq0\right)$. Assume  that 
$$\tPi_0 \ = \ \left( X_1^n,\cdots,X_{\ts_0^n}^n \right) $$
Assume further that
\begin{enumerate}
\item $s_0^n/n$ converges to $r\in(0,\infty)$ in $L^{2+\eps}$ for some $\eps>0$.
\item There exists $\nu\in M_P(\R^+)$ such that
$\tmu_0 \ \Longrightarrow \ \nu$ in $(M_F(\R^+),w)$.
\end{enumerate}
Then  
\be \left((\tmu_t, \ts^n_t); \ t\geq0\right) \ \Longrightarrow \  \left( (\mu_t, \frac{2}{\frac{2}{r}+t} ); \ t\geq0 \right) \label{eq:teo-local-1}\ee
where $\left(\mu_t; t\geq0\right)$ is the unique weak solution of (\ref{eq:Sch}) with initial condition $\nu$ and inverse population $\delta=\frac{2}{r}$. The convergence of $\{\tmu\}_n$ is meant in the Skorohod toplogy on 
$D([0,T], (M_F(\R^+),w))$ for every finite interval $[0,T]$. 
\end{teo}

In the previous theorem, we considered a sequence of nested Kingman coalescents (indexed by $n)$ with finite initial populations going to $\infty$ with $n$.
Next, we aim at investigating the case of a (single) nested Kingman coalescent where the size of the population at time $t=0$ is infinite.

\begin{defin}\label{def:infty-pop-K}
We say that $\left((\Pi_t, s_t); t>0\right)$ is an $\infty$-pop. nested coalescent iff 
\begin{enumerate}
\item[(i)] For every $t>0$,
conditioned on $(\Pi_t, s_t)$, the shifted process 
$(\theta_t\circ(\Pi_u, s_u); u\geq0)$ is a nested coalescent with 
initial condition $(\Pi_t, s_t)$.
\item[(ii)]  For every $t>0$, $\Pi_t(i)\geq 1$ for every $i\in[s_t]$.
\item[(iii)] $s_t \to \infty$ a.s. as $t\downarrow0$.
\end{enumerate}
\end{defin}
Note that property (i) and (iii) immediately imply that $(s_t; t>0)$
is distributed as the block counting process of a standard Kingman coalescent coming down from $\infty$.
In particular, $\frac{t}{2}s_t \to 1$ a.s. as $t\to 0$.

%\begin{lem}\label{lem:conv1}
%For every $t>0$, 
%there exists  a pair of random variables $(\S_t, \beta_t)$ such that
%\begin{itemize}
%\item $\S_t$ is distributed as the number of blocks 
%of a standard Kingman coalescent coming down from infinity and evaluated at time $t$
%\item $\beta_t$ is a random vector of integers 
%of size $\S_t$
%\end{itemize}
%such that up a subsequence, the shifted empirical measure 
%$$\left(\theta_t\circ\left(\frac{1}{s_u^m} \sum_{i=1}^{s_u^m} \Pi_t^m\right); u\geq 0\right)$$
%converges to the empirical measure of nested Kingman starting 
%with $\S_t$ species and an initial genetic decomposition $\beta_t$.
%\end{lem}
\begin{lem}\label{lem:existence-of-coalescent}
There exists an $\infty$-pop. nested Kingman coalescent. 
\end{lem} 
\begin{proof}

The idea is very similar to the existence of a proper solution to the $\infty$-pop. MK-V equation,
as described in Section \ref{set:existence-of-infinite-pop}. We omit the details and only give 
a brief outline of the construction.

Let $\{\delta_n\}$ a sequence of positive number decreasing to $0$ and 
let $(s_t; t>0)$ be a species coalescent (i.e., a standard coalescent with rate $1$) 
coming down from infinity.
For every $n$, define  $\left(\Pi^{(\delta_n),+}_t, \theta_{\delta_n}\circ s_t ; t>0\right)$ be
the nested coalescent 
starting with $s_{\delta_n}$
species
and infinitely many genes per species. 
Finally, for every $t>\delta_n$,
define $\Pi^{n,+}_t \ = \ \theta_{-\delta_n}\circ\Pi^{(\delta_n),+}_{t}$.
It is easy to find a coupling such that for every $t>0$, the sequence $\{\Pi^{n,+}_t\}_n$
is decreasing (i.e, each of the coordinates decreases in $n$) and converges to a limit $\Pi^+_t$ a.s., whereas
$\theta_{\delta_n}\circ s_t$ obviously converges to  $s_t$. 
Finally, one can easily check that $(\Pi^+_t, s_t)$
is an $\infty$-pop. nested Kingman coalescent at $\infty$.
\end{proof}

\begin{teo}\label{prop:sequential}
Let $(\mud_t; t>0)$ be the empirical measure of an $\infty$-pop. nested Kingman coalescent. Then 
$$\{\left((R^{n}\circ \mud_t, s_t); t>0\right)\}_{n} \Longrightarrow \left(\left(\mu^{(0)}_t, \frac{2}{t}\right); t>0\right)$$
where $\mu^{(0)}$
is the unique proper $\infty$-pop. solution 
of the \Sm equation (as described in Theorem \ref{thm:uniqueness-weak-infinite}) for $\psi(x)=\frac{c}{2} x^2$
and the convergence is meant in
$D([\tau,T],(M_F(\R^+),w))$ for any pair such that $0<\tau<T<\infty$.

In particular, for every $t>0$, $\mu_t^{(0)}$ is the law of $\frac{1}{t} \Upsilon$
where $\Upsilon$ is the r.v. defined in Theorem \ref{thm:self-similar-weak}. 
%Further,
%\be \left< R^n\circ \mud_t, x \right> \ \Longrightarrow_{n} \ \frac{1}{t} \int_0^\infty x h(x)dx  \ = \ \frac{1}{t} N(\mathscr{M}^c)\ee
%where $N$ is the Feller excursion measure and $\mathscr{M}$ is the event of mass extinction. (See Proposition \ref{prop:h} for a precise definition).
%\marginpar{to be cleaned}
\end{teo}

\begin{rmk}
\label{rmk:conj}
In the proof of Lemma \ref{lem:existence-of-coalescent}, we outlined the construction of 
the ``maximal'' $\infty$-pop. nested coalescent by starting from the $\infty$-initial condition at level $\delta_n$
(and  letting $\delta_n\to0$).
A ``minimal'' $\infty$-pop. nested Kingman coalescent would consist in setting the initial condition to $1$ for every species. 

Our next result suggests that
those two extremal nested coalescent are actually identical since their asymptotical empirical measures are undistinguishable. By a simple coupling argument, this should also imply that all
the nested coalescents coming down from $\infty$ are identical. In other words, we conjecture that there is a single entrance law 
for the nested Kingman coalescent.
\end{rmk}

As a corollary of the previous result, we will deduce the speed of coming down from infinity in the 
nested Kingman coalescent.

\begin{teo}[Speed of coming down from $\infty$]\label{teo:speed}
Let $\rho_t := s_{t} \left<\mud_t, x\right>$ be the number of gene lineages at time $t$. Then
\be \frac{1}{n^2} \rho_{t/n} \ \Longrightarrow_{n} \ \frac{2}{t} \int_0^\infty x \mu^{0}_t(dx) \ =  \frac{2}{t^2} \E\left( \Upsilon \right)<\infty.\ee
\end{teo}

\section{Some useful estimates} 

In this section, we establish some estimates which will be useful in due time. This section 
can be skipped at first reading.

%From now on, and until further notice, we will assume that $s_0^n=n$. In due time, we will show how this particular assumption can be easily extended to the case $(\ref{cond-r})$
%
%In this section, we develop some estimates which will be useful later on.
%The main result of this section is Proposition
%\ref{second-moment} which will be crucial for establishing the convergence 
%of the sequence of empirical measure $\{\tmu\}_n$.
%
%
%For the sake of clarity, we will assume in this section that $c=1$.
%The main result of this section (Proposition \ref{second-moment} ) remains unchanged for general value of $c>0$.
%We start with some estimates on the standard Kingman coalescent.

\begin{lem}[Large deviations]\label{lem:large-dev-2}
Let $z^n$ be the block counting process associated to a Kingman coalescent with rate $c>0$
such that $\{z^n_0\}$ is a deterministic sequence in $\R^+\cup\{+\infty\}$ such that $z_0^n/n \to r\in(0,+\infty]$.

There exists two functions $I_+$ and $I_-$ and two constants $K, \bar K$
\beqnn
\forall \gamma>0, \ \  \P\left( \frac{1}{n} z({t/n}) >   (1+\gamma)\frac{2r}{2+ rct}\right ) \leq K_+ \exp(-n I_+(\gamma)) \\
\forall \gamma\in(0,1),\ \  \P\left( \frac{1}{n} z({t/n}) <   (1-\gamma)\frac{2r}{2+ rct}\right ) \leq \bar K_- \exp(-n I_-(\gamma)) 
\eeqnn
such that $I_{\pm}(\gamma)>0$ for $\gamma>0$ and $\liminf_{\gamma\to\infty} I_+(\gamma)/\gamma >0$.
\end{lem}
\begin{proof}
The case $z_0^n=+\infty$ was treated in \cite{LS16}; see proof of Lemma 4.6 therein. The general case can be handled by a straightforward extension of our method.
\end{proof}

\begin{cor}\label{cor:second-moment-V}
Let $k\in\N$ and $r\in(0,\infty)$.
Let us assume that $\{s_0^n/n\}$ converges to a deterministic $r\in (0,\infty)$ in $L^{k+2+\eps}$ for some $\eps>0$. Then 
$$\limsup_n \ \E\left( (\frac{\ts_0^n}{n})^k \ \frac{1}{\ts_T^n}  \sum_{i=1}^{\ts_T^n} \  \tV_T^n(i)^2\right) \ < \ \infty.$$
\end{cor}
\begin{proof}
Take $\alpha= \frac{1}{2+rcT}$. Then
\begin{eqnarray*}
\E\left((\frac{\ts_0^n}{n})^k \frac{1}{\ts_T^n}   \sum_{i=1}^{\ts_T^n} \tV_T^n(i)^2  \ | \ \ts_0^n  \right)
& \leq & \frac{(\ts_0^n)^{k+2}}{n^k} \P\left(\frac{\ts_T^n}{\ts_0^n} < \alpha\  \ | \ \ts_0^n  \right) \ + \ \frac{({\ts_0^n})^{k+1}}{\alpha n^{k}} \E\left( \sum_{i=1}^{\ts_T^n} \frac{\tV_T^n(i)^2}{(\ts_0^n)^2} \ | \ \ts_0^n  \right)\\
&   =   &  \frac{(\ts_0^n)^{k+2}}{n^k}  \P\left(\frac{\ts_T^n}{\ts_0^n} < \alpha  \ | \ \ts_0^n  \right)  \ + \ \frac{(\ts_0^n)^{k+1}}{\alpha n^{k}} \left((1-\exp(-\frac{T}{n}))(1-\frac{1}{\ts_0^n}) + \frac{1}{\ts_0^n}\right),
\end{eqnarray*}
where the equality simply comes from the fact that the expectation on the RHS of the inequality is the probability that
two elements sampled in $\{1,\cdots,s_0^n\}$ (with replacement) are in the same block of a standard Kingman coalescent at time $T/n$.
(Note that our choice of $\alpha= \frac{1}{2} \frac{2}{2+rcT} $ is motivated by the previous large deviation estimates, so that 
when $s_0^n/n \sim r$ the first term on the RHS will be negligible.)

First, since $s_0^n/n$ converges in $L^{k+1}$, the second term of the RHS remains bounded in $L^1$.
Let us now deal with the second term and show that it remains bounded in $L^1$. Let us assume by contradiction that 
\be \limsup_n  \ \E\left(\frac{(\ts_0^n)^{k+2}}{n^k} 1_{\frac{\ts_T^n}{\ts_0^n} < \alpha}  \right) \ = \  \infty. \label{eq:contrad-v}\ee
Then, up to a subsequence, the latter expectation goes to $\infty$. Our aim is a extract a further bounded subsequence in $L^1$,
thus yielding a contradiction.
Let us now consider $p,q>1$ such that $1/p+1/q=1$ and $(2+k)p < 2+ k + \eps$. By Lemma \ref{lem:large-dev-2} and our choice of $\alpha$, 
we can take $\gamma$ small enough such that 
\[ \P\left( \ts_T^n < \alpha r(1+\gamma)n \ | \  s_0^n \ =  \ [r(1-\gamma)n] \right)\]  
goes to $0$ exponentially fast in $n$.
For this choice of $\gamma\in(0,1)$, one can extract a further subsequence
such 
\be\limsup_n \ n^{2q}\P(s_0^n/n \notin [r(1-\gamma),r(1+\gamma)]) <\infty.\label{eq:subsequence}\ee
For the rest of the proof, we will work under this subsequence.  Next,
\begin{eqnarray*}
\ \E\left(\frac{(\ts_0^n)^{k+2}}{n^k}  1_{\frac{\ts_T^n}{\ts_0^n} < \alpha}  \right) \ 
& \leq & \E\left(\frac{(\ts_0^n)^{k+2}}{n^k} , \frac{\ts_T^n}{\ts_0^n} < \alpha \ \mbox{and} \ s_0^n/n\in [r(1-\gamma),r(1+\gamma)] \right)  \\
& & \ + \  \E\left( \frac{(\ts_0^n)^{k+2}}{n^k}, s_0^n/n \notin [r(1-\gamma),r(1+\gamma)]  \right). 
\end{eqnarray*}
We first deal with the first term on the RHS of the inequality that we call (i).
\begin{eqnarray*}
 (i)& \leq & 
 r^{2+k}(1+\gamma)^{2+k} n^2 \P\left(  \ts_T^n < \alpha r(1+\gamma)n, \  s_0^n/n\in [r(1-\gamma),r(1+\gamma)] \right)   \\
& \leq & 
r^{2+k} (1+\gamma)^{2+k} n^2  \P\left( \ts_T^n < \alpha r(1+\gamma)n \ | \  s_0^n \ =  \ [r(1-\gamma)n] \right),
\end{eqnarray*}
where the RHS goes to $0$ by our choice of $\alpha$ and $\gamma$.
We now deal with the second term. By H\"older's inequality, we have
$$  \E\left( \frac{(\ts_0^n)^{k+2}}{n^k} , s_0^n \notin [r(1-\gamma),r(1+\gamma)]  \right)  \leq \E\left( (\frac{s_0^n}{n})^{(k+2)p}\right)^{\frac{1}{p}} \left( n^{2q}\P(  s_0^n \notin [r(1-\gamma),r(1+\gamma)]   )\right)^{1/q} $$
Since $\frac{s_0^n}{n}$ remains bounded in $L^{2+k+\eps}$,
and because of (\ref{eq:subsequence}),
this shows that $\limsup \E\left((\ts_0^n)^2 1_{\frac{\ts_T^n}{\ts_0^n} < \alpha}  \right)<\infty$, which is the desired contradiction.
This completes the proof of the lemma.
\end{proof}

\begin{lem}\label{lem:betan}
Assume that $\{s_0^n/n\}$ converges to a deterministic $r\in (0,\infty)$ in $L^{3+\eps}$ for some $\eps>0$.
Let $\left((\xi_{i}; i\in\N\right)$ be  
i.i.d. block counting processes of Kingman coalescent with rate $c>0$ coming down from infinity and independent of the species coalescent. Then
for every $0<\tau<T$:
$$ \limsup_{n\to\infty}\E\left( \sup_{[\tau,T]} \  \beta^n_{t} \right) \ < \ \infty \ \
 \mbox{where} \ \ \beta_t^n \ := \ \frac{1}{\ts_{t}^n} \ \sum_{i=1}^{\ts_{t}^n}  \left(\sum_{j\in \tB_t^n(i)} \frac{1}{n}\xi_{j}({t/n})\right)^2.  $$
and where $\tB_t^n$ was defined in Section \ref{sect:notations-discrete}.
\end{lem}
\begin{proof}

For every $k\in\N$, we define the event
$\tilde A_t^{k,n} \ \ = \ \{ \frac{2k n}{ct} \leq \max_{i\in[\ts_0^n]} \ \xi_{i}(t/n) \leq \frac{2(1+k)n}{ct}  \}$.
Take $k_0\in \N$ such that $k_0\frac{\tau}{T}>1$
\begin{eqnarray*}
\beta_t^n & \leq & \left( (\frac{2 k_0}{ct})^2  \ + \  \sum_{k=k_0}^\infty (\frac{2(k+1)}{ct})^2    1_{\tilde A_t^{k,n}}\right) \   \frac{1}{\ts_t^n}   \sum_{i=1}^{\ts_t^n} \tV_t^n(i)^2 \\
 & \leq & \left( (\frac{2 k_0}{ct})^2  \ + \  \sum_{k=k_0}^\infty (\frac{2(k+1)}{ct})^2    1_{\bar A_{t,T}^{k,n}}\right) \   \frac{1}{\ts_t^n}   \sum_{i=1}^{\ts_t^n} \tV_t^n(i)^2. 
\end{eqnarray*}
where $\bar A_{t,T}^{k,n} \ \ = \ \{\ \frac{2k n}{c T} \leq \max_{i\in[\ts^n_0]}\xi_{i}(t/n)\}$. Next,
for every $t\leq T$, let us denote by $\tilde C_{t,T}^n(i)$ be the indices of the blocks at time $t/n$ partionning the 
block $i$ at time $T/n$. (In particular, $\tilde C_{0,T}^n(i)=\tilde B_T^n(i)$.) Then
\begin{eqnarray}
\frac{1}{\ts_T^n}   \sum_{i=1}^{\ts_T^n} \tV_T^n(i)^2 
& = & \frac{1}{\ts_T^n}   \sum_{i=1}^{\ts_T^n} (\sum_{j \in \tilde C^n_{t,T}(i)} \tV_t^n(j))^2 \no \\   
& \geq & \frac{1}{\ts_T^n}   \sum_{i=1}^{\ts_T^n} \sum_{j \in \tilde C^n_{t,T}(i)} \tV_t^n(j)^2 \ = \   \frac{1}{\ts_T^n}   \sum_{k=1}^{\ts_t^n}  \tV_t^n(k)^2 \no \\
& \geq & \frac{1}{\ts_t^n}   \sum_{k=1}^{\ts_t^n}  \tV_t^n(k)^2 \label{eq:inequality1}
\end{eqnarray}
which implies that 
$\sup_{[\tau,T]} \frac{1}{\ts_t^n}   \sum_{i=1}^{\ts_t^n} \tV_t^n(i)^2   \ = \ \frac{1}{\ts_T^n}   \sum_{i=1}^{\ts_T^n} \tV_T^n(i)^2$. 
Further, since $\xi_{i}$
is non-increasing
$$\sup_{[\tau,T]} \beta_t^n \leq   \ \left( (\frac{2 k_0}{c\tau})^2  \ + \  \sum_{k=k_0}^\infty (\frac{2(k+1)}{c\tau})^2    1_{\bar A_{\tau,T}^{k,n}}\right) \   \frac{1}{\ts_T^n}   \sum_{i=1}^{\ts_T^n} \tV_T^n(i)^2.  $$
Since the $\xi_{i}$'s are independent of the coalescent, this yields
\begin{eqnarray*}
\E(\sup_{[\tau,T]}\beta_t^n \ |  \ \ts_0^n) & \leq & \left( (\frac{2 k_0}{c\tau})^2  \ + \  \sum_{k=k_0}^\infty (\frac{2(k+1)}{c\tau})^2     \P\left( \bar A_{\tau,T}^{k,n} \ | \ \ts_0^n\right)  \right) \   \E\left(\frac{1}{\ts_T^n}   \sum_{i=1}^{\ts_T^n} \tV_T^n(i)^2 \ | \ \ts_0^n \right)  \\
& \leq & \left( (\frac{2 k_0}{c\tau})^2  \ + \  \sum_{k=k_0}^\infty (\frac{2(k+1)}{c\tau})^2    {\ts_0^n} \ \P\left( \frac{1}{n}\xi(\tau/n) \geq   \frac{2k}{cT}  \right) \right) \   \E\left(\frac{1}{\ts_T^n}   \sum_{i=1}^{\ts_T^n} \tV_T^n(i)^2 \ | \ \ts_0^n\right) 
\end{eqnarray*}

Using Lemma \ref{lem:large-dev-2}
\begin{eqnarray*}
\E(\sup_{[\tau,T]}\beta_t^n \ | \ \ts_0^n ) & \leq & \left( (\frac{2 k_0}{c\tau})^2  \ + \  \frac{\ts_0^n}{n}\sum_{k=k_0}^\infty (\frac{2(k+1)}{c\tau})^2     n  K_+ \exp(-n I_{+}(k\frac{\tau}{T}-1)  \right) \   \E\left(\frac{1}{\ts_T^n}   \sum_{i=1}^{\ts_T^n} \tV_T^n(i)^2  \ | \ \ts_0^n\right)\\
\end{eqnarray*}
Recall that for every $k\geq k_0$, we have $k \frac{\tau}{T}>1$. 
Since $\liminf_{\gamma\to\infty} I_+(\gamma)/\gamma >0$, a straightforward application of  the dominated convergence theorem shows that the sum on the RHS of the inequality goes to $0$
as $n\to\infty$. Thus, there exists a constant $C$ (independent of $n$) such that
\begin{eqnarray*}
\E(\sup_{[\tau,T]}\beta_t^n \ | \ \ts_0^n )
& \leq &  \left( (\frac{2 k_0}{c\tau})^2  \ + \  C \frac{\ts_0^n}{n} \right) \   \E\left(\frac{1}{\ts_T^n}   \sum_{i=1}^{\ts_T^n} \tV_T^n(i)^2  \ | \ \ts_0^n\right)
\end{eqnarray*}
The fact that the RHS  of the inequality  is uniformly bounded in $L^{1}$  is handled by Corollary \ref{cor:second-moment-V}
(with $k=0,1$).
\end{proof}

\section{Convergence of the empirical measure}

In this section, we assume a sequence of nested coalescents $\{(\Pi^n_t, s_t^n); t\geq0\}$  indexed by $n$.
Assume that that there exists $r\in(0,\infty)$ such that 
\be\label{cond:conv-L1} \ts_0^n/n \ \to \ r  \ \mbox{in $L^1$}. \ee
We will also assume that
\be\label{cond:second-moment}\forall T>0, \ \ \ \limsup_n \E\left(\ \max_{[0,T]} \left<\tmu_t,x^2\right> \right)<\infty.\ee
As we shall see later, this condition will be satisfied under the initial conditions specified in Theorem \ref{teo-local-1},
and will also appear naturally in the infinite population nested Kingman coalescent.

\subsection{Generators}

We start with some definition. The process of genetic composition $\left(\tilde \Pi_t^n; t\geq0\right)$ defines a Markov process valued in the space 
$$ E \ := \ \ \cup_{k\in\N_*} \R^k_+. $$
We define $E_n$ to be the subspace of $E$
such that every coordinate of $\Pi\in E$ is such that $n\Pi(i)\in\N_*$.
%(In particular, we enforce the condition that each species contain at least one gene lineage at time $t=0$.)
%For a fixed value of $n$, we embed $\left(\tilde \Pi_t^n; t\geq0\right)$
%as a Markov chain valued in $E_n$.

For every $\Pi\in E$, we define 
$|\Pi|$ as the number of entries in $\Pi$. In particular,
we have $|\tilde \Pi_t^n| \ = \ts_t^n$. Finally,
when $|\Pi|>1$,
for every $i<j\leq |\Pi|$,
$\theta_{ij}(\Pi)$ is the only element $\Pi'\in E$ such that $\Pi'$ is obtained by coagulating coordinates $i$ and $j$. More precisely,
$\theta_{ij}(\Pi)$ is a vector of size $|\Pi|-1$ with coordinates
\begin{eqnarray*}\forall k < |\Pi|, \ \ \theta_{i,j}(\Pi)(k) & = &
\left\{\begin{array}{cc}
\Pi(i) + \Pi(j)& \mbox{if $k = i$}  \\
\Pi(k+1) & \mbox{if $k \geq j$}  \\
\Pi(k) &  \mbox{otherwise}
\end{array}\right.
\end{eqnarray*}
For instance,
$$\mbox{if} \ \ \Pi \ = \  \left(X^1, X^2, X^3, X^4\right), \ \ \mbox{then   } \ \theta_{1,3}(\Pi) \ = \  \left(X^1+X^3, X^2, X^4\right). $$
Finally, for every $\Pi\in E$, we define $\mud_{\Pi}=\frac{1}{|\Pi|} \sum_{i=1}^{|\Pi|} \delta_{\Pi(i)}$, the empirical measure associated 
to the genetic composition $\Pi$.

Let us now describe the generator of $\left( \tPi_{t}; t\geq0 \right)$
describing the evolution of the number of gene lineages per species. 
Define $e_i^k$ to be the vector of size $k$ filled with zeros except for the $i^{th}$ coordinate which is equal to $1$. Then
for every bounded function from $E$ to $\R$, and every $\Pi\in E_n$ we have
\begin{eqnarray}\label{eq:generator}
G^n h(\Pi) \ := \ \underbrace{\frac c n \sum_{i=1}^n \frac{n \Pi(i)(n \Pi(i)-1) }{2}\left(h(\Pi - \frac{1}{n}e_i^{|\Pi|}) - h(\Pi)\right)}_{\mbox{term I}}
+ \underbrace{\frac 1 n  1_{|\Pi|>1} \sum_{i<j\leq |\Pi|} \left(h\circ \theta_{ij}(\Pi) \ - \ h(\Pi) \right)}_{\mbox{term II}} \nonumber \\
\end{eqnarray}
where the first term corresponds to  a coalescence of two gene lineages (belonging to the same species), and the second term corresponds to a coalescence of two species lineages.
Finally, 
\be
\tmu_t \ = \ \frac{1}{|\tPi_{t}|} \sum_{i=1}^{|\tPi_{t}|} \delta_{\tPi_t(i)},
\ee
corresponds to the empirical measure of the block masses, where the mass of a block 
is measured in terms of its (renormalized) number of gene lineages.

Before going to the convergence of the empirical measure $\tmu$, we will need to establish a few technical lemmas
related to the generator of the process $(\tilde \Pi^n_t; t\geq0)$. For every $f\in C_b^\infty(\R^+)$, we define 
\be X^{n,f}_t \ :=  \ \left< \tmu_t, f\right> \label{def:X}\ee
Note that $X^{n,f}_t$ can be regarded as a function of $\tPi_t$. We call $h^f$ this function ($h^f(\Pi) \ = \ \left<g_\Pi, f\right>$), in such a way that
$X^{n,f}_t \ := \ h^f(\tPi_t)$, 

\begin{lem}[Generator approximation]\label{lem:moment11} 
Assume that conditions (\ref{cond:conv-L1}) and (\ref{cond:second-moment}) hold.
For every $f\in \cC_b^\infty(\R^+)$,  $\Pi\in E_n$,
define
$$\bar G h^f(\Pi) \ = \   - \left< \mud_{\Pi} , \frac{c x^2}{2} f' \right>  \ + \  \frac{r}{2+rs}\int_{\R^2} \mud_{\Pi}(dx) \mud_{\Pi}(dy) 
 \left(   f(x+y) -  f(x)   \right).  $$
 Then for every $t\geq 0$
 $$  \E\left( \int_0^t (G^n - \bar G) h^f(\tPi_s) ds \right)  \to 0  \ \ \mbox{as $n\to\infty$.}$$

 \end{lem}
 \begin{proof}
We first note that
$$  \E\left( \int_0^t  |\frac{|\tPi_s|}{2n} - \frac{r}{2+rs}|  \ \int_{\R^2} \mud_{\tPi_s}(dx) \mud_{\tPi_s}(dy) 
 \left(   f(x+y) -  f(x)  \ \right)  ds     \right) \ \leq \ 2 ||f||_\infty \E(\int_0^t |  \frac{|\tPi_s|}{2n} - \frac{r}{2+rs}|    ds ). $$
From Lemma \ref{lem:large-dev-2} and using the fact that $(|\tPi_t|/n; t\geq0)$ is non-increasing
$$ (|\tPi_t|/n; t\geq0) \to (\frac{2r}{2+rt}; t\geq0) \ \ \mbox{in probability} $$
(where the convergence is meant in the Skorohod topology on every interval $[0,T]$)
so that the integrand on the RHS of the latter inequality goes to $0$ in probability. Further,
by assumption $\{|\tPi_0|/n\}$ is uniformly integrable, and  since $|{\tPi_0|}\geq |\tPi_t|$
it easily follows (by uniform integrability in $(\Omega\times[0,t], \P\times dt)$) that 
$$ \E(\int_0^t |  \frac{|\tPi_s|}{2n} - \frac{r}{2+rs}   | ds ) \to 0$$
so that the LHS of the latter inequality vanishes.
From our assumptions, it is then sufficient  to show 
the existence of a constant $K$  such that  for every $\Pi\in E_n$ and $f\in \cC_b^\infty(\R^+)$
\begin{eqnarray} \left| G^n h^f(\Pi) \ - \left(  \  - \left< \mud_{\Pi} , \frac{c x^2}{2} f' \right>  \ + \  \frac{|\Pi|}{2n}\int_{\R^2} \mud_{\Pi}(dx) \mud_{\Pi}(dy) 
 \left(   f(x+y) -  f(x)  \ \right) \right) \right| \nonumber \\
 \leq \frac{K}{n} \left( {||f'||_\infty} \left< \mud_\Pi, x \right> \ + \ {||f''||_\infty} \left< \mud_\Pi, x^2 \right> \right) \ + \  K \times 1_{|\Pi|>1}\frac{1}{|\Pi|-1}{||f||_\infty} . \label{inequality:11}
 \end{eqnarray}
We start by approximating term I  of the generator $G^n$ (as defined in (\ref{eq:generator})). 
For every $\Pi\in E_n$, we have
\begin{eqnarray*}
I & = & \frac {c}{  n} \sum_{i=1}^{|\Pi|} \frac{n \Pi(i)(n \Pi(i)-1) }{2}\left(h^f(\Pi-\frac{1}{n} e_i^{|\Pi|} ) - h^f(\Pi)\right) \\
   & = & \frac {c}{ |\Pi| n} \sum_{i=1}^{|\Pi|} \frac{n \Pi(i)(n \Pi(i)-1) }{2}\left(f(\Pi(i)-\frac{1}{n}) - f(\Pi(i))\right) 
   \end{eqnarray*}
and by a simple Taylor expansion, it follows that there exists a constant $c_1$ such that 
   \begin{eqnarray*}
|\ I \ + \ \left< \mud_{\Pi} , \frac{c x^2}{2} f'  \right> \ |  & \leq & c_1\left(   \frac{||f'||_\infty}{n} \left< \mud_\Pi, x \right> \ + \ \frac{||f''||_\infty}{n} \left< \mud_\Pi, x^2 \right>  \right).
\end{eqnarray*}
Let us now deal with  term $II$ of the generator $G^n$ (again as defined in (\ref{eq:generator})). For $\Pi$ such that $|\Pi|>1$, consider the measure
\begin{equation}\label{eq:approxmu}
\nu_{\Pi}(dx dy) \ = \ \mud_{\Pi}(dx) \frac{|\Pi|}{|\Pi|-1} \left(\mud_{\Pi}(dy) - \frac{1}{|\Pi|} \delta_{x}(dy) \right),
\end{equation}
i.e., $\nu_{\Pi}$ is the measure that consists in sampling two elements with
no replacement according to the measure $\mud_{\Pi}$.
Then 
\begin{eqnarray}
II & = & 1_{|\Pi|>1} \frac{|\Pi|(|\Pi|-1)}{2 n} \int_{(\R^+)^2} \nu_{\Pi}(dx dy) \nonumber \label{eq:II}
 \\ 
 & & \times  \left( \ \frac{|\Pi|}{|\Pi|-1} \left< \mud_{\Pi},f \right> + \frac{1}{|\Pi|-1} f(x+y) - \frac{1}{|\Pi|-1}f(x) - \frac{1}{|\Pi|-1}f(y) \ - \   \left< \mud_{\Pi},f \right>  \right) \nonumber \\
 & = &
  1_{|\Pi|>1}  \frac{|\Pi|(|\Pi|-1)}{2 n} \int_{(\R^+)^2} \nu_{\Pi}(dx dy) \left( \ \frac{1}{|\Pi|-1} \left< \mud_{\Pi},f \right> + \frac{1}{|\Pi|-1} f(x+y) - \frac{1}{|\Pi|-1}f(x) - \frac{1}{|\Pi|-1}f(y)  \right) \nonumber\\
  & = &   1_{|\Pi|>1} \frac{|\Pi|}{2 n} \int_{(\R^+)^2} \nu_{\Pi}(dx dy)
 \left(  \left< \mud_{\Pi},f \right> +  f(x+y) -  f(x) - f(y) \ \right) \nonumber  
 \end{eqnarray}
In words, we coalesce two species lineages at rate $\frac{|\Pi|(|\Pi|-1)}{2 n}$. Conditional on a coalescence event,
we pick two species lineages according to the measure $\nu_{\Pi}(dx dy)$. If we pick two lineages with coordinate 
$x$ and $y$ respectively, then  the change in $\left<\mud_{\Pi},f\right>$
is readily given by the term between parenthesis.  

By using (\ref{eq:approxmu}), one gets the existence of a constant $c_2$ such that 
\begin{eqnarray*}
| \  II - \   1_{|\Pi|>1} \frac{|\Pi|}{2 n} \int_{(\R^+)^2} \mud_{\Pi}(dx) \mud_{\Pi}(dy) \left(  \left< \mud_{\Pi},f \right> +  f(x+y) -  f(x) - f(y) \ \right) 
 \ | & \leq & 
 1_{|\Pi|>1} c_2 \frac{||f||_\infty}{|\Pi|-1} 
\end{eqnarray*}
Using the fact that
$$  \int_{(\R^+)^2} \mud_{\Pi}(dx) \mud_{\Pi}(dy) \left(  \left< \mud_{\Pi},f \right> +  f(x+y) -  f(x) - f(y) \ \right)  \  = \   \int_{(\R^+)^2} \mud_{\Pi}(dx) \mud_{\Pi}(dy)
 \left(   f(x+y) -  f(x)  \ \right)  $$
this yields
\begin{eqnarray*}
|\  II \ - \  1_{|\Pi|>1} \frac{|\Pi|} {2n} \int_{(\R^+)^2} \mud_{\Pi}(dx) \mud_{\Pi}(dy)
 \left(   f(x+y) -  f(x)  \ \right) | & \leq & 
c_2   1_{|\Pi|>1} \frac{||f||_\infty}{|\Pi|-1} 
\end{eqnarray*}
which is the desired inequality (\ref{inequality:11}).
 
 \end{proof}

\begin{lem}\label{lem:moment2}
For every $ 0\leq u\leq t \leq T$ 
$$| \left<X^{n,f}\right>_t -  \left<X^{n,f}\right>_u  | \ \leq \frac{C}{n}  \int_u^t \left( ||f' ||^2_\infty  \left<\tmu_t,x^2\right>  \ + \  || f ||^2_{\infty} \right)ds  $$
\end{lem}
\begin{proof}

$h^f(\tPi)$ is a pure jump process and its bracket term $\left<X^{n,f}\right>_t -  \left<X^{n,f}\right>_u$ 
can be decomposed into two terms, i.e. 
$ \left<X^{n,f}\right>_t -  \left<X^{n,f}\right>_u\ = \  \int_u^t I'_s \ ds  \ + \ \int_u^t II'_s  \ ds$ where 
\begin{eqnarray*}
I_t' 
& = & \frac{c}{n} \sum_{i=1}^{\tilde s_t^n} \frac{n \tPi_t(i)(n \tPi_t(i)-1) }{2}\left(h^f(\tPi_t-\frac{1}{n} e_i^{\tilde s_t^n}) - h^f(\tPi_t)\right)^2  \\
& = & \frac{c}{n |\tPi_t|^2} \sum_{i=1}^{|\tPi_t|} \frac{n \tPi_t(i)(n \tPi_t(i)-1) }{2}\left(f(\tPi_t(i)-1/n) - f(\tPi_t(i))\right)^2 
\end{eqnarray*}
and 
\begin{eqnarray*}
II_t' & = &   1_{|\tPi_t|>1} \\
& \times & \frac{|\tPi_t|(|\tPi_t|-1)}{2 n} \int_{(\R^+)^2} \nu_{\tPi_t}(dx dy)
 \left( \ \frac{1}{|\tPi_t|-1} \left< \mud_{\tPi_t},f \right> + \frac{1}{|\tPi_t|-1} f(x+y) - \frac{1}{|\tPi_t|-1}f(x) - \frac{1}{|\tPi_t|-1}f(y) \ \right)^2 
\end{eqnarray*}
where the sampling measure $\nu_{\Pi}$
is defined as in (\ref{eq:approxmu}), and the expression of $II_t'$
is obtained by an argument analogous to the one for obtaining (\ref{eq:II}).
Straightforward estimates yield that
$$| \ I_t' \ | \ \leq \frac{c ||f' ||^2_\infty }{n} \left<\tmu_t,x^2\right>  $$
and 
\begin{eqnarray*}
|\  II_t' \ | & \leq & 16 \times  1_{|\tPi_t|>1}  || f ||_{\infty} \frac{|\tPi_t|}{n (|\tPi_t|-1)} \leq 32 \frac{|| f ||_{\infty}}{n} 
\end{eqnarray*}
which is the desired result.
\end{proof}

\subsection{Tightness result}

% In the following, $(M_F(\R^+),v)$ will refer to the set of Radon measures on $\R^+$
%endowed with the vague topology (i.e., the smallest topology making the map $\mud \to \left<\mud,f\right>$
%continuous for every function $f\in \cC_0(\R^+)$ -- the set of continuous functions vanishing $+\infty$);
%whereas $(M_F(\R^+),w)$ will refer to $M_F(\R^+)$
%equipped with the weak topology (i.e., the smallest topology making the map $\mud \to \left<\mud,f\right>$
%continuous for every function $f\in \cC_b(\R^+)$ -- the set of continuous bounded functions). The aim of this section is to show the following result.

The aim of this section is to show the following tightness result. This will be the key ingredient to the proof of our convergence results.

\begin{prop}\label{prop:tightness}
%Assume that there exists a deterministic $r\in(0,\infty)$ \marginpar{Can we just include the case where $r=\infty$ ? It seems to be the case ...}
%\be s^n_0/n \to r  \ \ \mbox{\ in probability.}\label{cond-r}\ee
Assume that conditions (\ref{cond:conv-L1}) and (\ref{cond:second-moment}) hold.
For every $T>0$, the sequence processes $\{\tmu\}_{n\geq0}$ is tight in $D([0,T], (M_F(\R^+),w))$
and any converging subsequence belongs to $\cC([0,T], (M_F(\R^+),w))$, 
the space of continuous functions from  $[0,T]$ to $(M_F(\R^+), w)$. Further
\begin{enumerate}
\item[(i)] Any accumulation point $\mud^\infty$ is a weak solution of the \Sm (\ref{eq:Sch}) with inverse population
$2/r$ in the sense that for every $t\geq0$ and every test function $f$
$$
0 = \  \left<\mud^\infty_t, f  \right> -  \left< \mud^\infty_0 , f  \right>  + \int_0^t \left< \mud^\infty_s , c\frac{x^2}{2} f'  \right>ds  \ - \  \int_{0}^t \frac{1}{s+\frac{2}{r}}\int_{(\R^+)^2} \mud^\infty_{s}(dx) \mud^\infty_s(dy)
 \left(   f(x+y) -  f(x)  \ \right)   ds \    
$$
\item[(ii)] For every $t\in[0,T]$, $\left<g_t^\infty,x\right><\infty$ and 
$$\left<\tmu_t,x\right> \to \left<\mud^\infty_t,x\right>  \ \ \mbox{in probability.}$$
\end{enumerate}

\end{prop}

In order to prove tightness, we follow a standard line of thoughts 
(see e.g., \cite{FM04,T08,T14}. The approach is condensed in the statement of Theorem \ref{teo:tran} 
which is cited from Tran \cite{T14} (Theorem 1.1.8).
This Theorem can be obtained by concatenating the so-called Roelly criterium \cite{R86}
(which states that the tightness of $D([0,T], (M_F(\R^+),v))$   boils down to proving that $\{\left<\mud_n,f\right>\}_{n\geq0}$
for $f\in\cC_b^\infty(\R^+)$
is tight in $D([0,T],\R)$), and a criterium due 
 M\'el\'eard and Roelly [ M\'el\'eard, Roelly], allowing to go from vague to weak convergence by checking that no mass is lost at $\infty$.

\begin{teo}\label{teo:tran}
Let $\{\tmu\}$ be a sequence in $D([0,T], (M_F(\R^+),w)$. Then the three following conditions are sufficient for the tightness
of $\{\tmu\}$ in $D([0,T], (M_F(\R^+),w))$.
\begin{enumerate}
\item[(i)] For every $f\in \cC_b^\infty(\R^+)$, the sequence $\{\left<\mud^n,f\right>\}_n$ is tight in $D([0,T], \R)$.
\item[(ii)] $\limsup_{n} \E(\sup_{[0,T]}\left<\tmu_t,x^2\right>)<\infty$.
\item[(iii)] Any accumulation point $\mud^\infty$ of $\{\tmu\}$ (in $D([0,T], (M_F(\R^+),v)$) belongs to $\cC([0,T], (M_F(\R^+), w))$.
\end{enumerate}
\end{teo}
%
%\begin{rmk}
%The ony difference with Theorem 1.1.8 in \cite{T14} is the second condition
%which is easily seen to be stronger than the analog condition in Theorem 1.1.8 in [Tran].
%\end{rmk}

\begin{proof}[Proof of Proposition \ref{prop:tightness}]

{\bf Step 1.} We first show that 
for every $f\in \cC_b^\infty(\R^+)$, the sequence of processes  $\{X^{n,f}\}_{n}$
(as defined in (\ref{def:X}))
is tight. In order to do so, we  use the classical Aldous and Rebolledo criterium \cite{A78,J86}. We first note that for every $t\geq0$ 
$$ |X^{n,f}_t| \leq ||f||_\infty  $$
so that the first requirement of Aldous criterion (i.e., for every deterministic $t$, $\{X^{n,f}_t\}_n$ is tight) is satisfied. Next, let $\gamma>0$ be an arbitrary small number and let us consider two
stopping times $(\tau,\sigma)$ such that
$$ 0 \leq \tau\leq \sigma \leq \tau +\gamma\leq T. $$
First, we decompose the semi-martingale $X^{n,f}$ into its martingale part and its drift part, namely,
\be X^{n,f}_t \ = \ M^{n,f}_t \ + \ B_t^{n,f}, \ \mbox{where $B_t^{n,f} \ := \ \int_0^t G^n h^f(\tPi_s)ds$ and $M_t^{n,f} \ := \  X^{n,f}_t - \int_0^t G^n h^f(\tPi_s)ds$}.\label{eq:def:martingale}\ee
It remains to show that the quantities 
$$ \E(| B^{n,f}_\sigma  -   B^{n,f}_\tau|) \ \ \mbox{and}  \ \  \E(| M^{n,f}_\sigma  -   M^{n,f}_\tau|)  $$ 
are bounded from above by a function of $\gamma$ (uniformly in the choice of the two stopping times $\tau$ and $\sigma$ and $n$)
going to $0$ as $\gamma$ goes to $0$. (This is the second part of Aldous and Robolledo criterium).
In order to prove this result, we now make use of some of the technical results established earlier.

First, from (\ref{eq:generator}), we note that there exists a constant $\bar K$  such that  for every $\Pi\in E_n$ and $f\in C_b^\infty(\R^+)$
$$ | G^n h^f (\Pi) | \ \leq \ \bar K\left( ||f'||_\infty \left< \mud_\Pi, x^2 \right>  \ + \ \frac{|\Pi|}{n}  ||f||_\infty \right).  $$
This implies that 
\begin{eqnarray*}
 \E\left(| B^{n,f}_\sigma  -   B^{n,f}_\tau|\right) & \leq &  \E\left(\int_\tau^\sigma |G^nh^f(\tPi_s)| ds \right) \\
 & \leq & \bar K \E\left(\int_\tau^\sigma \left( ||f'||_\infty\left<\mud_{\tPi_s},x^2\right>  \ + \ \frac{s_0^n}{n}||f||_\infty \right) ds \right) \\
 &\leq &\bar  K \gamma \left( ||f'||_\infty \E\left(\sup_{[0,T]} \left<\mud_{\tPi_s},x^2\right> \right) + \frac{s_0^n}{n} ||f||_\infty\right).
\end{eqnarray*}
Further,
\begin{eqnarray*}
 \E(| M^{n,f}_\sigma  -   M^{n,f}_\tau|)^2 & \leq &   \E(| M^{n,f}_\sigma  -   M^{n,f}_\tau|^2) \\
 &  = & \E\left( \left<X^{n,f}\right>_\sigma - \left<X^{n,f}\right>_\tau  \right) \\
 & \leq &   \frac{C}{n} \E\left(\int_\tau^\sigma \left( ||f'||^2_\infty\left<\mud_{\tPi_s},x^2\right>  \ + \ ||f||^2_\infty \right) ds \right) \\
 & \leq & \frac{C}{n} \gamma \left( ||f'||^2_\infty \E\left( \sup_{[0,T]} \left<\mud_{\tPi_s},x^2\right> \right)\ + \   ||f||^2_\infty  \right) 
\end{eqnarray*}
where the second inequality follows from Lemma  \ref{lem:moment2}.
Combining the two previous inequalities with  (\ref{cond:conv-L1}) and (\ref{cond:second-moment})
shows the tightness of $\{X^{n,f}\}_{n\geq 0}$.

\medskip

{\bf Step 2.} Let $\mud^\infty$ be an accumulation point of the sequence $\{\tmu\}$ in $D([0,T], (M_F(\R^+), v))$. 
Since $\tmu_t = \frac{1}{\ts_t^n} \sum_{i=1}^{\ts_t^n} f(\tPi_t)$ and a transition can only affect two coordinate of $\tPi$
at a time,
it is not hard to show that
$$ \sup_{t\in[0,T]} \sup_{f\in L^\infty([0,T]), || f ||_\infty  \leq 1} | \left<\tmu_{t},f\right>  - \left<\tmu_{t-},f\right> | \leq \frac {4}{\ts_{T}^n} $$
where $L^\infty([0,T])$ is the set of bounded functions from $[0,T]$ to $\R$.
Since $s_T^n$ goes to $\infty$ (in probability) as $n\to \infty$, this  implies that $\mud^\infty$ belongs to $\cC([0,T], (M_F([0,T]), w))$.
The tightness of $\{\tmu\}$ in $D([0,T], (M_F(\R^+),w))$ then
follows by a direct application of Theorem \ref{teo:tran} (using the second moment assumption $\limsup_{n} \E\left( \sup_{[0,T]} \left<\tmu,x^2\right> \right)<\infty$).

\medskip

{\bf Step 3.} 
% Let $\mud^\infty$
%denote an accumulation point of $\{\tmu\}$. Since
%$\{\tmu\}$ converges in the {\it weak} topology, we get that for every $t\in[\delta,T]$, $\left<\tmu_t,1\right> \to \left<\mud^\infty_t,1\right>$.
%Since $\left<\tmu_t,1\right>=1$,
%the boundary condition  $u(t,0)=1$ in (\ref{eq:laplace-r}) is automatically satisfied.
Next, let $f$ be an arbitrary test function in $\cC_b^\infty(\R^+)$.
For every $m\in D([0,T],( M_F(\R^+),w))$, define
$$\varphi_{f,t}(m) \ = \  \left<m_t, f  \right> -  \left< m_0 , f  \right>  + \int_0^t \left< m_s , c\frac{x^2}{2} f'  \right>ds  \ - \  \int_{0}^t \frac{1}{s+\frac{2}{r}}\int_{(\R^+)^2} m_{s}(dx) m_s(dy)
 \left(   f(x+y) -  f(x)  \ \right)   ds.    $$
In this step, we show that $\varphi_{f,t}(\mud^\infty)=0$, for every $t\in[0,T]$ and  any choice of test function $f$
in $\cC_b^\infty(\R^+)$. We first observe 
\begin{eqnarray*}
\E \left(| \varphi_{f,t}(\tmu)  |\right)  & \leq & \E (| M^{n,f}_t - M^{n,f}_0   |) \ + \  \E\left( \int_0^t \left| (G^n - \bar G)h^f(\tilde \Pi^n_s)  \right| ds  \ \right) 
 \end{eqnarray*}
 where $M^{n,f}$ is the Martingale defined in (\ref{eq:def:martingale}) and $\bar G$ is the generator approximation defined in  Lemma \ref{lem:moment11}.
 % \E | X^{n,f}_t - \int_0^t G^n h^f(\tPi_s) ds | \label{eq:tr}
% \\ +   \E \left( \int_0^t   | G^n h^f(\tPi_s) \ - \left(  \  - \left< \mud_{\tPi_s} , \frac{c x^2}{2} f' \right>  \ + \ \int_0^t \frac{|\tPi_s|}{2n} \int_{\R^2} \mud_{\tPi_s}(dx) \mud_{\tPi_s}(dy) 
% \left(   f(x+y) -  f(x)  \ \right) \right) | ds \right) \nonumber
% \\ +  2 ||f||_\infty \E \left( \ \int_0^t|\frac{r}{2+rs} - \frac{|\tPi_s|}{2n}| ds \ \right) 
When we let $n\to\infty$, the second term vanishes by Lemma \ref{lem:moment11}.
For the first term, we have 
\begin{eqnarray*}
\left( \E | M^{n,f}_t - M^{n,f}_0   |\right)^2 & \leq &  \E \left( M^{n,f}_t - M^{n,f}_0   \right)^2 \\
								      & = & \E\left(\left< X^{n,f} \right>_t -  \left< X^{n,f} \right>_0\right)
\end{eqnarray*}
and the RHS can be handled by Lemma \ref{lem:moment2} and  our second moment assumption (\ref{cond:second-moment}). This implies
$$ \lim_{n\to\infty} \E \left(| \varphi_{f,t}(\tmu)  |\right) \ = \ 0. $$
On the other hand, since any accumulation point $\mud^\infty$
must be in $\cC([0,T], (M_F(\R^+),w))$
and since $f$ and its derivative $f'$ are continuous, we must have
$$ \varphi_{f,t}(\tmu) \ \Longrightarrow  \ \varphi_{f,t}(\mud^\infty).  $$
(Here we use the fact that $f$ is a test function so that $f$ and $\psi f'$ remain bounded, and further, if $\{(\tilde m^n_t; t\geq0)\}$ converges to a continuous $(m^\infty_t; t\geq0)$,
then for every continuous and bounded in $u$, the process $\left(\left<m^n_t,u\right>; t\geq0\right)$ converges to 
$\left(\left<m_t,u\right>; t\geq0\right)$  in the uniform norm
on every finite interval)
we get that $\E(|\varphi_{f,t}(\mud^\infty)|)=0$ by a direct application of the bounded convergence theorem.

\medskip

{\bf Step 4.} In the previous step, we showed that $\varphi_{f,t}(\mud^\infty)=0$
for any test function in $\cC_b^\infty(\R^+)$. By a standard density argument, 
the result also holds for any test function, thus showing that $g^\infty$
is a weak solution of the \Sm equation (\ref{eq:Sch}) with inverse population $\delta$.

\medskip

{\bf Step 5.} Let us now show the convergence of the mean. 
The argument is quite standard and goes by approximating the function $x$
by a bounded and continuous function to make use of the weak convergence. Define
\[f^{(k)}(x) \ =  x\  \mbox{if  $x\leq k$}, \ f^{(k)}(x)  =  k  \ \mbox{otherwise},\]
and note that
\be \left<\tmu_t,x\right> \ =  \ \left<\tmu_t, f^{(k)}(x) \right> \ + \ \left<\tmu_t, x-f^{(k)}(x) \right> \label{eq:two-terms}.\ee
We now let $n$ and then $k$ go to $0$ sequentially. By using the Cauchy-Schwarz and Markov inequality, for any $k\geq 1$, we get
\begin{eqnarray}
 \left<\tmu_t, x-f^{(k)}(x) \right>^2 & \leq & \left<\tmu_t, 1_{x\geq k}\right> \left<\tmu_t,(x-k)^+\right> \nonumber \\
 & \leq & \frac{1}{k^2} \left<\tmu_t, x^2\right>^{3/2} \label{eq:cs-2}
\end{eqnarray}
and using (\ref{cond:second-moment}), the RHS of the inequality goes to $0$ (in probability) as $n$ and then $k$, go sequentially to $0$. On the other hand,
since $\{\tmu_t\}_n$ converges to $g^\infty_t$ as $n\to\infty$ in the weak topology, the first term on the RHS of (\ref{eq:two-terms})
converges to  $\left<g^\infty_t, f^{(k)} \right>$. Finally, 
as $k\to \infty$,  $\left<g^\infty_t, f^{(k)} \right>$ goes to to 
$\left<g^\infty_t, x \right>$ by the monotone convergence theorem.
This completes the proof for the convergence of the mean.
\end{proof}

\subsection{Proof of Theorem \ref{teo-local-1}}

We start by showing the convergence of $\{\left(\tmu_t; \ t\in[0,T]\right)\}_n$. By Proposition \ref{prop:tightness} (and the unicity of (\ref{eq:Sch}) with initial condition $\nu$ and $\delta=2/r$), it is enough to show that
$\limsup_n \E\left(\sup_{[0,T]} \left<\tmu_t, x^2\right> \right) \  < \ \infty$.

For every $t\leq T$, denote by $\tilde C_{t,T}^n(i)$ be the indices of the blocks at time $t/n$ partionning the 
block $i$ at time $T/n$. (In particular, $\tilde C_{0,T}^n(i)=\tilde B_T^n(i)$.)  We have
\begin{eqnarray}
\frac{1}{\ts_T^n} \sum_{i=1}^{\ts_T^n} \ \left( \sum_{k\in \tB_T^n(i)}  \tPi_0(k)  \right)^2
& = & 
\frac{1}{\ts_T^n} \sum_{i=1}^{\ts_T^n} \ \left( \sum_{j\in \tC_{t,T}^n(i)} \sum_{k\in \tB_t^n(j)} \tPi_0(k)\right)^2 \nonumber \\
& \geq &
\frac{1}{\ts_T^n} \sum_{i=1}^{\ts_T^n} \  \sum_{j\in \tC_{t,T}^n(i)} \left(\sum_{k\in \tB_t^n(j)} \tPi_0(k)\right)^2  \nonumber\\
& = &
\frac{1}{\ts_T^n} \sum_{i=1}^{\ts_t^n} \ \left(\sum_{k\in \tB_t^n(i)} \tPi_0(k)\right)^2   \nonumber \\
& \geq & \frac{1}{\ts_t^n} \sum_{i=1}^{\ts_t^n} \   \left(\sum_{k\in \tB_t^n(i)} \tPi_0(k)\right)^2. \label{ineq:uni}  
\end{eqnarray}
Thus, for every $t\leq T$, this yields
\begin{eqnarray*}
\left<\tmu_t, x^2\right> & \leq & \ \frac{1}{\ts^n_t} \sum_{i=1}^{\ts_t^n} (\sum_{j\in \tilde B_t^{n}(i)} \tilde \Pi^n_0(j))^2 \leq  \frac{1}{\ts_T^n} \sum_{i=1}^{\ts_T^n} \ \left( \sum_{j\in \tB_T^n(i)}  \tPi_0(j)  \right)^2
\end{eqnarray*}
where the first inequality is obtained by ignoring the coalescence events between gene lineages (in particular, the first inequality becomes an equality when $c=0$). Let $\{{\mathcal G}_t; t\geq0\}$ be the natural filtration generated by the species coalescent. 
From the previous arguments, we get that 
\begin{eqnarray*}
\E\left(\sup_{[0,T]} \left<\tmu_t, x^2\right> \right)  
& \leq & \E\left(\frac{1}{\ts_T^n} \sum_{i=1}^{\ts_T^n} \ \left( \sum_{j\in \tB_T^n(i)}  \tPi_0(j)  \right)^2 \right) \\
& \leq & \E\left(\frac{1}{\ts_T^n}  \sum_{i=1}^{\ts_T^n} \  \tV_T^n(i)\sum_{j\in \tB_T^n(i)} (\tilde \Pi_0^n(i))^2   \right) \\
& = & \E\left(\frac{1}{\ts_T^n}  \sum_{i=1}^{\ts_T^n} \  \tV_T^n(i)\sum_{j\in \tB_T^n(i)} \E\left( (X_i^n)^2 \ | \ {\mathcal G}_t\right)  \right) \\
& = & \E\left( (X_1^n)^2 \right)  \  \E\left(\frac{1}{\ts_T^n}  \sum_{i=1}^{\ts_T^n} \  \tV_T^n(i)^2\right),
\end{eqnarray*}
and the RHS remains bounded by assumptions and Corollary \ref{cor:second-moment-V}.

%easy if we assume that $\frac{1}{n} \mbox{gene lineages}$ remain finite foreach species as $n\to\infty$.

It remains to show the joint convergence statement (\ref{eq:teo-local-1}). The convergence of $\{\ts^n\}$ was already stated in Lemma \ref{lem:large-dev-2}.  The joint convergence follows from the fact that 
the limit of the marginals are both deterministic.

\section{Coming down from infinity in the nested Kingman coalescent}
In the following $(s_t,\Pi_t)$ will denote an $\infty$-pop. nested Kingman coalescent, and $\mud_t$
will denote the associated empirical measure.

\begin{prop}\label{lem:stoch-dom}
For every $t>0$, we have 
$\liminf_n\  R^n\circ \mud_t \geq \delta_{\frac{2}{ct}}$ in the sense that for 
every continuous, bounded and non-decreasing function $f$ 
$$ \E\left(\ \liminf_{n} \left<R^n \circ \mud_t, f\right> \right) \geq f({2/ct}).$$
\end{prop}
\begin{proof}
First, note that  $\mud_t$ stochastically dominates the case 
where each species lineage carries a single gene lineage at time $0$. Hence, we can assume w.l.o.g. 
this particular initial condition. Secondly, since the species constraint  
forbids coalescence events between gene lineages belonging to different species,
$(g_t; t\geq0)$ dominates $(K_t; t\geq0)$, where $K$ is the block counting process
of a Kingman coalescent with rate $c$. (In other words, in $K$, we allow gene lineages to coalesce 
even if they belong to different species.) Finally, since $\frac{ct}{2}K_t\to0$ a.s., the result follows.
\end{proof}

Our next aim is to show the following result.

\begin{prop}\label{prop:general-second-moment}
For every $0<\tau<T$,
$$ \limsup_{n} \ \E\left(\sup_{[\tau,T]} \left<R^n\circ \mud_t, x^2\right> \right) \  < \ \infty$$
\end{prop}

\begin{proof}
We have
\be \sup_{t\in[\tau, T]} \ \left<R^n\circ \mud_t, x^2\right>  \ = \ \sup_{t\in[\frac{\tau}{2},T-\frac{\tau}{2}]} \ \left<R^n\circ \hat \mud^n_t, x^2\right> \ \ \mbox{in law,}\label{eq:fjus}\ee
where $\hat \mud^n$ is the empirical measure associated to the nested coalescent with 
the initial number of species being equal to $\hat s_0^n  = s_{\tau/2n}$
and genetic composition vector $\Pi_{\tau/2n}$. Further, using the large deviation estimates of Lemma \ref{lem:large-dev-2},
we get
\be \hat s_0^n/n  = s_{\tau/2n}/n \ \to \frac{4}{\tau} \in(0,\infty)  \ \ \mbox{in $L^p$,} \ \ \forall p>1. \label{eq:cv-in-lp}\ee

The RHS of (\ref{eq:fjus})
 is always bounded from above by the same quantity if we replace $\hat \mud^n_t$ by the empirical
 measure associated to the nested coalescent starting with $\hat s_0^n$ species and
infinitely many gene lineages in each species. In turn, the latter model is bounded by the model starting from the infinite initial condition, but
where gene lineages can only coalesce if they belong to the same species at time $0$, i.e., even if species 1 and 2 coalesce, their respective gene lineages are forbidden to merge afterwards.
The empirical measure associated to the process is identical in law to
$$  m^n_t \ := \  \frac{1}{\hat s_t^{n}} \sum_{i =1 }^{\hat s_t^{n}} \delta_{ \sum_{j\in \hat B_t^{n}(i)}  \xi_{j}(t)},$$
where 
$\hat B_t^{n}(i)$ (w.r.t. to the species coalescent $\hat s^n$)
and $\xi_{j}'s$ are defined analogously to  Lemma \ref{lem:betan}. This yields
$$\forall t\in[\tau/2,T-\tau/2], \ \  \left<  R^n\circ \hat \mud^n_t,x^2\right> \leq  \left<  R^n\circ m^n_t,x^2\right> =_{\cL} \beta^n_t$$
(where the domination is meant in the stochastic sense)
and thus
$$\E\left(\sup_{t\in[\tau,T]} \left<R^n\circ \mud_t, x^2\right>\right) \leq  \E\left(\sup_{t\in[\tau/2,T-\tau/2]} \beta_t^n \right) $$
Proposition \ref{prop:general-second-moment}
then  follows by a direct application of Lemma \ref{lem:betan} (and  (\ref{eq:cv-in-lp})). \
\end{proof}

%We start with some definition.  In this section, we consider a sequence of nested Kingman coalescents indexed by $m$,
%such that for a given $m$, the number of species lineages at time $0$ is equal to $m$. For every $m$,  
%we denote by $\mud^m$ 
%the empirical measure associated to the number of gene lineages 
%$$\mud^m_t \ = \ \frac{1}{s^m_t} \sum_{i=1}^{s_t^m} \  \delta_{\Pi^m_t(i)}  $$
%where $\Pi^m_t(i)$ is the number of gene lineages in block $j$ (for the species coalescent) for the $m^{th}$ nested coalescent.
%We call $\Pi^m_t$ the genetic decomposition of the model at time $t$.
%
%
%We will say that  the measure valued process $(R^n\circ \nu^{m}_t; t\geq0)$ converges sequentially iff there exists  
%a limiting measure valued process $(\nu^\infty_t; t\geq0)$ such that for every $0<\tau<T$, we have 
%$$\lim_{n\to\infty} \lim_{m\to\infty} (R^{n}\circ \nu^{m}_t; t\in[\tau,T])  \ = \ \left( \nu^\infty_t; t\in[\tau,T] \right) \ \ \mbox{in law}, $$
%where the convergence is meant in the space $D([\tau,T], M_F(\R^+))$. 

\begin{proof}[Proof of Theorem \ref{prop:sequential}]

%In the following, we will say that  the measure--valued process $(R^n\circ \mud^{m}_t; t\geq0)$ is sequentially compact iff there exist two
%integer--valued sequences $(\phi_n, \psi_m)$  and 
%a limiting process $(\mud^\infty_t; t\geq0)$ such that for every $0<\tau<T$, we have 
%$$\lim_{n\to\infty} \lim_{m\to\infty} (R^{\phi_n}\circ \mud^{\psi_m}_t; t\in[\tau,T])  \ = \ \left( \mud^\infty_t; t\in[\tau,T] \right) \ \ \mbox{in law}, $$
%where the convergence is meant in the space $D([\tau,T], M_F(\R^+))$, and where $\mud_\infty$
%satisfies (\ref{}) on $\tau>0$.
%
% 
% 
%We will now show that 
%$\{R^{n}\circ \mud^{m}\}_{n,m}$ is sequentially tight and that any limit must be the global solution
%of the Schmoluchowsky equation.
%

%
%{\bf Step 1.} Let us first fix $\tau>0$ and $n\in\N$. By Lemma \ref{lem:conv1} and by continuity of the scaling operator $R^n$, up to a subsequence (in $m$),
%the sequence  
%$\{\theta_{\tau}\circ R^n\circ\mud^m\}_{m}$ converges to a measure valued process $\hat \mud^{n}$ as $m\to\infty$, where $\hat \mud^n$
%is the empirical measure associated to 
%a nested Kingman coalescent with  a random initial number of species $\S_{\tau/n}$
%and some initial genetic decomposition $\beta_{\tau/n}$ (where the latter two random variables are distributed 
%as in Lemma \ref{lem:conv1}). 
%
%Finally, using a standard diagonalization argument, one can find a subsequence in $m$, such that for every $n$, the sequence
%$\{\left(\theta_{\tau}\circ R^n\circ\mud^m; t\geq0\right)\}_{m}$ converges to a process $\left(\hat \mud^{n,(\tau)}; t\geq0\right)$ along
%this subsequence.
%
%\medskip

{\bf Step 1.} 
Let us fix $\tau>0$. Define $\hat \mud^{n,(\tau)} = \theta_\tau \circ R^n \circ \mud_t$
and let $\hat s^{n,(\tau)} = \theta_\tau \circ s_t^n$.
Proposition \ref{prop:general-second-moment} impies that 
$$\limsup_n \E\left( \sup_{t\in[0,T]} \left<\hat \mud^{n,(\tau)}_t, x^2\right> \right) \ \ = \ \E\left( \sup_{t\in[\tau,T+\tau]} \left<R^n\circ \mud_t, x^2\right> \right)   < \ \infty$$
Further, by Lemma \ref{lem:large-dev-2},
$$\underbrace{\frac{1}{n}}_{\mbox{time scaling}} \times \underbrace{\hat s^{n,(\tau)}_0 (= \S_{\tau/n})}_{\mbox{number of blocks in the species coalescent at time $0$}} \to r = \frac{2}{\tau} \ \ \mbox{in $L^p$ for every $p>1$}.$$
By Proposition \ref{prop:tightness}, it follows that the sequence 
$\{\hat \mud^{n,(\tau)}\}_n$ is tight and that any sub-sequential limit $\hat \mud^{\infty,(\tau)}$ 
is a weak solution of the \Sm equation with inverse population size $\tau/2$.
The continuous mapping theorem implies that
$\{R^n\circ\mud_t = \theta_{-\tau}\circ\hat \mud^{n,(\tau)}; t\geq\tau\}$ converges to the limit $(\theta_{-\tau}\circ\mud^{\infty,(\tau)}; t\geq\tau)$
where the latter process   
has a Laplace process satisfying 
the equation
\be\label{eq:laplace20}
\forall t\geq \tau,  \  \left<\nu_t, f  \right> -  \left< \nu_\tau , f  \right>  + \int_\tau^t \left< \nu_s , c\frac{x^2}{2} f'  \right>ds  \ - \  \int_{\tau}^t \frac{1}{s}\int_{(\R^+)^2} \nu_{s}(dx)\nu_{s}(dy)
 \left(   f(x+y) -  f(x)  \ \right)   ds = 0   
\ee
Note that  the coefficients of the IPDE do not depend on the value of $\tau$.

\medskip

{\bf Step 2.}  Let us now take a sequence of positive numbers $\{\tau_m\}_m$  going to $0$. 
For every $m$, there exists a subsequence of $\{(R^n\circ g_t; t\geq\tau_m)\}_n$
converging to $\mu^{(0),m}$ satisfying (\ref{eq:laplace20}).
By a standard diagonalization argument, this ensures 
the existence of a subsequence of $\{R^{n}\circ \mud\}$
converging to a process $\mu^{(0)}$ defined on $(0,\infty)$ (in comparaison with step 1 where the process
was defined on $[\tau,\infty)$) and satisfying
\be\label{eq:laplace2}
\forall t,\tau>0,  \  \left<\nu_t, f  \right> -  \left< \nu_\tau , f  \right>  + \int_\tau^t \left< \nu_s , c\frac{x^2}{2} f'  \right>ds  \ - \  \int_{\tau}^t \frac{1}{s}\int_{(\R^+)^2} \nu_{s}(dx) \nu_s(dy)
 \left(   f(x+y) -  f(x)  \ \right)   ds \ = \ 0, 
\ee
i.e.,  $\mu^{(0)}$ is a weak solution of the $\infty$-pop. \Sm equation.
In order to prove Theorem \ref{prop:sequential}, 
it remains to show that $\mu^{(0)}$ is the only proper solution.
This follows directly from the stochastic domination of Proposition \ref{lem:stoch-dom}.
Finally, the joint convergence with $(\tmu, \ts)$
follows form the fact that both marginals are deterministic at the limit.

\end{proof}

\begin{proof}[Proof of Theorem \ref{teo:speed}] 
Let $\rho_t$ be the number of gene lineages at time $t$. We need to show that
\[ \frac{1}{n^2}\rho_{t/n} \ \Longrightarrow \ \frac{2}{t^2} \ \int_0^\infty x \mu^{(0)}_t(x) dx \]
where $\mu^{(0)}$ is the proper solution of the \Sm equation.
By applying Proposition \ref{prop:tightness} and Theorem \ref{prop:sequential}, 
$\left(\frac{1}{n}s_{t/n}, \left< R^n\circ\mud_t,x \right> \right)$
converges to $\left(\frac{2}{t}; \frac{1}{t} \int_0^\infty x \mu_t^{(0)}(x) dx \right)$. 
The result follows from the observation that 
$$\frac{1}{n^2}\rho_{t/n} \ = \ \frac{1}{n}s_{t/n} \left< R^n\circ\mud_t,x \right>.$$
\end{proof}

%
%\begin{cor}[Speed of coming down from $\infty$]
%Let $\rho_t = s_{t} \left<\mud_t, x\right>$ be the number of gene lineages at time $t$. Then
%\be \frac{1}{n^2} \rho_{t/n} \ \Longrightarrow_{n} \ \frac{2}{t^2} N(\mathscr{M}^c)\ee
%\end{cor}
%\begin{proof}
%The result follows from the relation
%$$\frac{1}{n^2}\rho_{t/n} \ = \ \frac{1}{n}s_{t/n} \left< R^n\circ\mud_t,x \right> $$ \marginpar{check notations}
%and the fact that $\left(s_{t/n}; \left< R^n\circ\mud_t,x \right> \right)$
%converges to $\left(\frac{2}{t}; \frac{1}{t} N(\mathscr{M}^c)\right)$. \marginpar{more ref}
%\end{proof}
%
%\section{Long time behavior of the solution of the IDE (\ref{eq:smol})}
%
%In this section, we prove Theorem \ref{teo-local-2}. TO BE FILLED IN.
%

\appendix

\section{}
\label{Appendix1}

Here, we complete the proof of Theorem \ref{thm:uniqueness-weak} (ii) by showing that $\mu_T$ \emph{defined as} $F({\bf T}, (W_i); 1\le i \le N_T (\bf T))$ indeed is solution to \eqref{eq:Sch}.

Let $f$ be a test-function as defined before Definition \ref{def:weak1}, i.e., $f\in\cC^1(\R^+)$ such that $f$ and $f' \psi$ are bounded. Hereafter, we continue to denote by $\PP_T$ the joint law of the pure-birth tree ${\bf T}$ started with one particle at time 0, birth rate $a(T-t)$, stopped at time $T$, and of the iid rvs $(W_i; 1\le i \le N_T (\bf T))$ with law $\nu$. We will abbreviate  $F({\bf T}, (W_i); 1\le i \le N_T (\bf T))$ into $F({\bf T})$. In particular, denoting $\mu_T$ as the law of $F({\bf T}, (W_i); 1\le i \le N_T (\bf T))$ under $\PP_T$, we have
$$
\mu_T(f):=\int_{\R^+} f(x)\, \mu_T(dx) = \EE_T( f\circ F({\bf T})),
$$
so that 
$$
\mu_{T+\varepsilon}(f) =  \EE_{T+\varepsilon}( f\circ F({\bf T}), N_\varepsilon =1) + a(T)\varepsilon\, \EE_T^{\otimes 2}( f\circ F({\bf T}+{\bf T}'))+o(\varepsilon),
 $$
 where ${\bf T}'$ is an independent copy of ${\bf T}$ and ${\bf T}+{\bf T}'$ denotes the tree splitting at time 0 into the two subtrees ${\bf T}$ and ${\bf T}'$. First recall that 
$$
F(\tr) = \lim_{x\downarrow 0}  x^{-1} \left(1-e^{-xM^{\tr}\left(1-\exp\left(-\sum_{i=1}^{N_T(\tr)}w_iZ_T^i\right)\right)}\right)= \lim_{x\downarrow 0}  x^{-1} \left(1-Q_x^\tr\left(e^{-\sum_{i=1}^{N_T(\tr)}w_iZ_T^i}\right)\right),
$$
so that
\begin{eqnarray*}
F(\tr+\tr') &=& \lim_{x\downarrow 0}  x^{-1} \left(1-Q_x^{\tr+\tr'}\left(e^{-(\sum_{i=1}^{N_T(\tr)}w_iZ_T^i+\sum_{i=1}^{N_T(\tr')}w_i'Z_T^{i'})}\right)\right)\\
	&=& \lim_{x\downarrow 0}  x^{-1} \left(1-Q_x^\tr\left(e^{-\sum_{i=1}^{N_T(\tr)}w_iZ_T^i}\right)Q_x^{\tr'}\left(e^{-\sum_{i=1}^{N_T(\tr')}w_i'Z_T^i}\right)\right)\\
	&=&\lim_{x\downarrow 0}  x^{-1} \left(1-e^{-xM^{\tr}\left(1-\exp\left(-\sum_{i=1}^{N_T(\tr)}w_iZ_T^i\right)\right)}e^{-xM^{\tr'}\left(1-\exp\left(-\sum_{i=1}^{N_T(\tr')}w_i'Z_T^i\right)\right)}\right)\\
	&=&F(\tr) + F(\tr').
\end{eqnarray*}
Second, if we denote by $\tr+\varepsilon$ the tree obtained from $\tr$ by merely adding a length $\varepsilon$ to its root edge, then  by the Markov property of the entrance measure of the CSBP at 0,
\begin{eqnarray*}
F(\tr +\varepsilon) &=&  M^{\tr+\varepsilon}\left(1-\exp\left(-\sum_{i=1}^{N_T(\tr)}w_iZ_{T+\varepsilon}^i\right)\right)\\
	&=& \int_{(0,\infty)}N(Z_\varepsilon \in dx )\, Q_x^\tr\left(1-\exp\left(-\sum_{i=1}^{N_T(\tr)}w_iZ_T^i\right)\right)\\
	&=& \int_{(0,\infty)}N(Z_\varepsilon \in dx ) \left(1-\exp\left(-xM^\tr\left(1-\sum_{i=1}^{N_T(\tr)}w_iZ_T^i\right)\right)\right)\\
	&=&N\left(1-\exp\left(-Z_\varepsilon F(\tr)\right)\right).
\end{eqnarray*}
Now as specified at the end of Subsection \ref{subsec:csbp}, for each fixed $\lambda$, $N\left(1-\exp\left(-\lambda Z_t\right)\right)$ is solution to $\dot x =-\psi(x)$ with initial condition $x(0)=\lambda$. As a consequence,
$$
\lim_{\varepsilon \downarrow 0}\varepsilon^{-1}\left(F(\tr +\varepsilon)- F(\tr)\right)  =\lim_{\varepsilon \downarrow 0}\varepsilon^{-1}\left(N\left(1-\exp\left(-Z_\varepsilon F(\tr)\right)\right)-F(\tr)\right)=-\psi(F(\tr)).
$$
Combining the last two results, we obtain
\begin{eqnarray*}
\mu_{T+\varepsilon}(f) &= & \EE_{T+\varepsilon}( f\circ F({\bf T}), N_\varepsilon =1) + a(T)\varepsilon\, \EE_T^{\otimes 2}( f\circ F({\bf T}+{\bf T}'))+o(\varepsilon)\\
 &=& (1-a(T)\varepsilon)\,\EE_{T}( f\circ F({\bf T}+\varepsilon))+ a(T)\varepsilon\, \EE_T^{\otimes 2}( f\circ (F({\bf T})+F({\bf T}')))+o(\varepsilon)\\
 &=& \mu_{T}(f) + (1-a(T)\varepsilon)\, \EE_{T}( f\circ F({\bf T}+\varepsilon)- f\circ F({\bf T})) +a(T)\varepsilon\, (\mu_T^{\star 2}(f)- \mu_T(f))+o(\varepsilon).
 \end{eqnarray*}
Next, since $\psi f'$ is bounded, by dominated convergence, we get
$$
\lim_{\varepsilon \downarrow 0}\varepsilon^{-1}\EE_{T}( f\circ F({\bf T}+\varepsilon)- f\circ F({\bf T})) = %\EE_{T} \lim_{\varepsilon \downarrow 0}\varepsilon^{-1}(f\circ F({\bf T}+\varepsilon)- f\circ F({\bf T}))
 -\EE_{T}\big(\psi (F({\bf T}))\, f'(F({\bf T}))\big)=-\mu_T(\psi f').
$$
As a consequence,
$$
\lim_{\varepsilon \downarrow 0}\varepsilon^{-1}(\mu_{T+\varepsilon}(f)- \mu_{T}(f)) = -\mu_T(\psi f')+a(T)\, (\mu_T^{\star 2}(f)- \mu_T(f)).
$$
So $t\mapsto\mu_t(f)$ is right-differentiable with continuous right-derivative equal to 
$$
\partial_t \mu_t(f) = -\mu_t(\psi f')+a(t)\, (\mu_t^{\star 2}(f)- \mu_t(f))\qquad t\ge 0.
$$
Also note that 
$$
F(\varnothing+\varepsilon) = M^{\varnothing+\varepsilon}\left(1-\exp\left(-w_1Z_\varepsilon\right)\right) =N \left(1-\exp\left(-w_1Z_\varepsilon\right)\right),
$$
so that $F(\varnothing)= w_1$ and $\mu_0(f)= \EE_0(f(F({\bf T})))=\EE_0(f(W))=\nu(f)$. This shows that $\mu_0=\nu$ so that $(\mu_t(f);t\ge 0)$ satisfies \eqref{eq:weak}.

\bibliographystyle{abbrv}
\bibliography{bib}

\end{document}